\documentclass[letter,reqno,11pt]{amsart}
\usepackage{amssymb,amsmath,amsfonts}
\usepackage{mathrsfs}
\usepackage[left=0.75 in, right=0.75 in,top=0.75 in, bottom=0.75 in]{geometry}

\usepackage{nameref,hyperref,cleveref}

\usepackage{pgfplots,xcolor}

\newtheorem{theorem}{Theorem}[section]
\newtheorem{lemma}[theorem]{Lemma}

\newtheorem{corollary}[theorem]{Corollary}

\makeatletter
\renewcommand \theequation {%
\ifnum \c@section>\z@ \@arabic\c@section.%
\fi\@arabic\c@equation} \@addtoreset{equation}{section}
\makeatother

\providecommand{\abs}[1]{\left\vert#1\right\vert}
\providecommand{\nm}[1]{\left\Vert#1\right\Vert}

\providecommand{\bro}[1]{\left\langle #1 \right\rangle}

\providecommand{\hm}[2]{\left\Vert#1\right\Vert_{H^{#2}(\Omega_0)}}
\providecommand{\hms}[2]{\left\Vert#1\right\Vert_{H^{#2}(-\ell,\ell)}}

\providecommand{\hme}[2]{\left\Vert#1\right\Vert_{H^{#2}(\Sigma_{0})}}

\providecommand{\hmw}[2]{\left\Vert#1\right\Vert_{W^{#2}_{\delta}(\Omega_0)}}
\providecommand{\hmws}[2]{\left\Vert#1\right\Vert_{W^{#2}_{\delta}(\p\Omega_0)}}
\providecommand{\hmwb}[2]{\left\Vert#1\right\Vert_{W^{#2}_{\delta}(\Sigma_{0b})}}
\providecommand{\hmwe}[2]{\left\Vert#1\right\Vert_{W^{#2}_{\delta}(\Sigma_0)}}
\providecommand{\hmwss}[2]{\left\Vert#1\right\Vert_{W^{#2}_{\delta}(-\ell,\ell)}}

\providecommand{\wwd}[2]{W^{#1}_{\delta}(#2)}

\providecommand{\lm}[2]{\left\Vert#1\right\Vert_{L^{#2}(\Omega_0)}}
\providecommand{\lms}[2]{\left\Vert#1\right\Vert_{L^{#2}(-\ell,\ell)}}

\providecommand{\lme}[2]{\left\Vert#1\right\Vert_{L^{#2}(\Sigma_{0})}}

\def\ud{\mathrm{d}}
\def\dt{\partial_t}
\def\p{\partial}
\def\ls{\lesssim}
\def\gs{\gtrsim}

\def\rt{\rightarrow}
\def\r{\mathbb{R}}
\def\no{\nonumber}
\def\ue{\mathrm{e}}

\def\ds{\displaystyle}

\def\z{\zeta}
\def\e{\eta}

\def\na{\nabla}
\def\de{\Delta}
\def\dm{\mathbb{D}}

\def\v{\mathcal{V}}
\def\pp{\mathcal{P}}
\def\w{\mathcal{W}}

\def\be{\bar\e}

\def\a{\mathcal{A}}
\def\ep{\epsilon}
\def\d{\delta}
\def\n{\mathcal{N}}
\def\t{\mathcal{T}}

\def\s{\mathcal{F}}

\def\rr{\mathcal{R}}
\def\qq{\mathcal{Q}}
\def\ss{\mathcal{S}}
\def\oo{\mathcal{O}}
\def\lp{\mathscr{L}}
\def\rp{\mathscr{R}}
\def\kp{\mathscr{K}}
\def\ap{\mathfrak{a}}
\def\ep{\varrho}

\def\hd{{}_0\mathcal{H}^1(\Omega_0)}

\def\hl{\mathring{H}^0(\Omega_0)}

\def\g{\mathcal{G}}

\def\d{\delta}

\def\en{\mathcal{E}}
\def\di{\mathcal{D}}
\def\enp{\mathcal{E}_{\|}}
\def\dip{\mathcal{D}_{\|}}
\def\dipt{\tilde{\mathcal{D}}_{\|}}
\def\sen{\mathfrak{E}}
\def\sdi{\mathfrak{D}}
\def\sf{\mathfrak{F}}

\def\vv{\mathscr{V}}
\def\ww{\mathscr{W}}

\begin{document}

\title{Dynamics and stability of sessile drops with contact points}

\author[I. Tice]{Ian Tice}
\address[I. Tice]{
   \newline\indent Department of Mathematical Sciences, Carnegie Mellon University
\newline\indent Pittsburgh, PA 15213, USA}
\email{iantice@andrew.cmu.edu}
\thanks{I. Tice was supported by an NSF CAREER Grant (DMS \#1653161). }

\author[L. Wu]{Lei Wu}
\address[L. Wu]{
   \newline\indent Department of Mathematics, Lehigh University
\newline\indent Bethlehem, PA 18015, USA}
\email{lew218@lehigh.edu}
\thanks{L. Wu was supported by an NSF Grant (DMS \#1853002).}



\begin{abstract}
In an effort to study the stability of contact lines in fluids, we consider the dynamics of a drop of incompressible viscous Stokes fluid evolving above a one-dimensional flat surface under the influence of gravity. This is a free boundary problem: the interface between the fluid on the surface and the air above (modeled by a trivial fluid) is free to move and experiences capillary forces. The three-phase interface where the fluid, air, and solid vessel wall meet is known as a contact point, and the angle formed between the free interface and the flat surface is called the contact angle. We consider a model of this problem that allows for fully dynamic contact points and angles. We develop a scheme of a priori estimates for the model, which then allow us to show that for initial data sufficiently close to equilibrium, the model admits global solutions that decay to a shifted equilibrium exponentially fast.
\end{abstract}

\maketitle


\section{Problem formulation}

Consider a two-dimensional droplet of viscous incompressible fluid evolving above a one-dimensional flat surface.   Denote the spatial variable $z=(z_1,z_2) \in\r^2$.  We then assume that at time $t \ge 0$ the fluid occupies the moving droplet domain
\begin{align}
\Omega(t):=\{z\in\r^2: 0<z_2<\z(t,z_1)\},
\end{align}
where the free surface of the droplet is given by the unknown function $\z(t,\cdot):  [L(t),R(t)]\rt \r^+$, which  satisfies $\z(t,L(t))=\z(t,R(t))=0$.  Here $L(t)$ and $R(t)$ are the left and right end points of the moving droplet domain.  We write the free surface at the top of the droplet as
\begin{align}
\Sigma(t):=\{(z_1,z_2): L(t)<z_1<R(t), z_2=\z(t,z_1) \},
\end{align}
and at the bottom as
\begin{align}
\Sigma_b(t):=\{(z_1,z_2): L(t)<z_1<R(t), z_2=0 \}.
\end{align}
See Figure \ref{omega_figure} for an example of such a fluid droplet domain.   For each $t\ge 0$, the fluid is described by its velocity and pressure $(u(t,\cdot),P(t,\cdot)): \Omega(t)\rt\r^2\times\r$. The viscous stress tensor is determined in term of $P$ and $u$ according to $S(P,u)=PI-\mu\dm_z u$, where $I$ is the $2\times2$ identity matrix, $\dm_z u=\na_z u+(\na_z u)^T$ is the symmetric gradient of $u$, and $\mu>0$ is the viscosity of the fluid.  We note that a simple computation reveals that  if $\na_z\cdot u=0$, then $\na_z\cdot S(P,u)=-\mu\de_z u+\na_z P$.

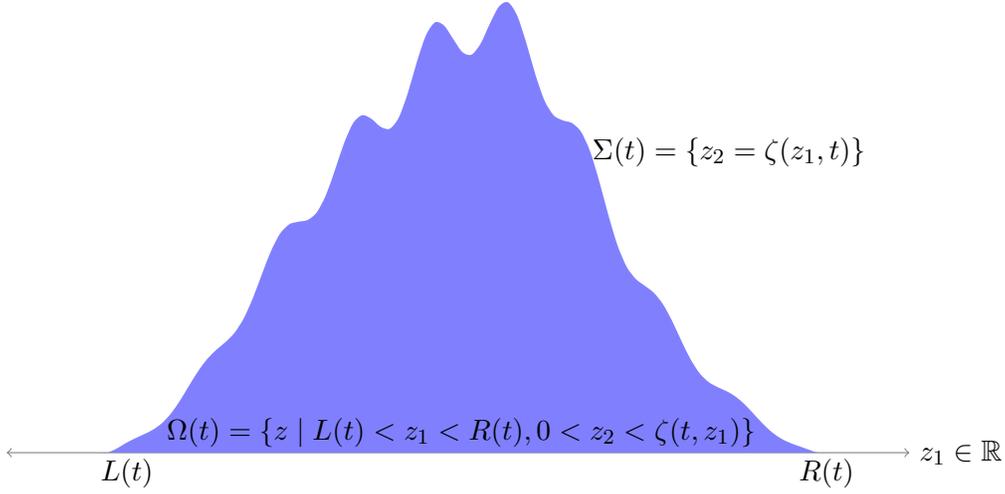
\begin{figure}

\begin{tikzpicture}[xscale=2,yscale=2]
\draw[<->,gray] (-3,0) -- (3,0);
\node[anchor=west] at (3,0) {$z_1 \in \mathbb{R}$};

\path [fill=blue!50] plot [smooth,samples=100, domain=-3:3] (\x, {max(0,(3+.2*sin(210+700*\x) + .2*sin(190*\x))*exp(-(\x)^2/2)-.2)})  -- (3,0) -- (-3,0);

\node[anchor=north] at (0.02,.30) {$\Omega(t) = \{z \;\vert\; L(t) < z_1 < R(t), 0  < z_2 < \zeta(t,z_1)\}$};

\node[black] at (1.8,2) {$\Sigma(t) = \{z_2 = \zeta(z_1,t)\}$};

\node[anchor=south] at (-2.2,-.3) {$L(t)$};
\node[anchor=south] at (2.45,-.3) {$R(t)$};
\end{tikzpicture}

\caption{An example of a droplet domain.}\label{omega_figure}

\end{figure}

Before stating the equations of motion, we define a number of terms that will appear. We will write $g>0$ for the strength of gravity, $\sigma>0$ for the surface tension coefficient along the free surface, and $\beta>0$ for the Navier slip friction coefficient on the flat surface. The coefficients $\gamma_{sv}, \gamma_{sf}\in\r$ are a measure
of the free-energy per unit length associated to the solid-vapor and solid-fluid interaction, respectively.  We set $[\gamma]=\gamma_{sv}-\gamma_{sf}$ and make the crucial assumption that
\begin{equation}\label{young_sign}
0 < \frac{[\gamma]}{\sigma}<1.
\end{equation}
This is equivalent to the classical Young's law, together with the extra assumption that $[\gamma] >0$.  The former is a necessary condition for the existence of any equilibrium state, while the latter is a technical condition that guarantees that any equilibrium state can be described as above by the graph of a function.    Finally, we define the contact point velocity response function $\vv:\r\rt\r$ to be a $C^2$ increasing diffeomorphism such that $\vv(0)=0$. We will refer to its inverse as $\ww=\vv^{-1}\in C^2(\r)$.

We require that $(u,P,\z,L,R)$ satisfy the gravity-driven free-boundary incompressible Stokes equations in $\Omega(t)$ for $t>0$:
\begin{align}\label{system 1}
\left\{
\begin{array}{ll}
\na_z\cdot S(P,u)=-\mu\de_z u+\na_z P=0&\ \ \text{in}\ \ \Omega(t),\\
\na_z\cdot u=0&\ \ \text{in}\ \ \Omega(t),\\
S(P,u)\nu=g\z\nu-\sigma \mathcal{H}(\z)\nu&\ \ \text{on}\ \ \Sigma(t),\\
\Big(S(P,u)\nu-\beta u\Big)\cdot\tau=0&\ \ \text{on}\ \ \Sigma_b(t),\\
u\cdot\nu=0&\ \ \text{on}\ \ \Sigma_b(t),\\
\dt\z+u_1\p_{z_1}\z-u_2=0&\ \ \text{on}\ \ \Sigma(t),\\
\dt L=\vv\left(\sigma\dfrac{1}{\sqrt{1+\abs{\p_{z_1}\z}^2}}\Bigg|_{z_1=L}-[\gamma]\right),\\
\dt R=\vv\left([\gamma]-\sigma\dfrac{1}{\sqrt{1+\abs{\p_{z_1}\z}^2}}\Bigg|_{z_1=R}\right),
\end{array}
\right.
\end{align}
for $\nu$ the outward-pointing unit normal vector, $\tau$ the associated unit tangent vector, and
\begin{align}
\mathcal{H}(\z):=\p_{z_1}\left(\dfrac{\p_{z_1}\z}{\sqrt{1+\abs{\p_{z_1}\z}^2}}\right),
\end{align}
being twice the mean-curvature operator. Note that here we have already shifted the gravitational force to eliminate the atmospheric pressure $P_{atm}$ by adjusting the actual pressure $\bar P$ according to $P=\bar P+gz_2-P_{atm}$. Also, it is easy to see that the contact point equations are equivalent to
\begin{equation}
\ww(\dt L)=\sigma\dfrac{1}{\sqrt{1+\abs{\p_{z_1}\z}^2}}\Bigg|_{z_1=L}-[\gamma] \text{ and }
\ww(\dt R)=[\gamma]-\sigma\dfrac{1}{\sqrt{1+\abs{\p_{z_1}\z}^2}}\Bigg|_{z_1=R}.
\end{equation}
The mass of the fluid is conserved in time since the transport equation in \eqref{system 1} is equivalent to $\dt\z=(u\cdot\nu)\sqrt{1+\abs{\p_{z_1}\z}^2}$:
\begin{align}
\frac{\ud}{\ud{t}}\abs{\Omega(t)}&=\frac{\ud}{\ud{t}}\int_{L(t)}^{R(t)}\z(t,z_1)\ud{z_1}=\int_{L(t)}^{R(t)}\dt\z(t,z_1)\ud{z_1}+\dt L\z\Big(t,L(t)\Big)-\dt R\z\Big(t,R(t)\Big)\\
&=\int_{L(t)}^{R(t)}\dt\z(t,z_1)\ud{z_1}=\int_{\Sigma(t)}u\cdot\nu=\int_{\p\Omega(t)}u\cdot\nu=\int_{\Omega(t)}\na_z\cdot u=0.\no
\end{align}
We denote this conserved mass
\begin{align}
M=\abs{\Omega(t)}=\int_{L(t)}^{R(t)}\z(t,z_1)\ud{z_1}.
\end{align}

\subsection{Background and model discussion}

The study of triple interfaces between fluid, solid, and vapor phases is rather old, dating to the work of Young \cite{young}, Laplace \cite{laplace}, and Gauss \cite{gauss} in the nineteenth century.  In the subsequent two centuries this problem has attracted the attention of far too many researchers for us to attempt a full survey of the literature here.  Instead we refer to the exhaustive survey by de Gennes \cite{degennes} for a more thorough discussion.

The initial work of Young, Laplace, and Gauss showed that equilibrium configurations not only solve a particular equation, known as the gravity-capillary equation, but also satisfy fixed contact angle conditions determined via
\begin{equation}\label{young_relat}
 \cos(\theta_{eq}) = - \frac{[\gamma]}{\sigma}.
\end{equation}
See Figure \ref{fig:contact_angle} for a schematic.  Note in particular that this, together with our sign assumption \eqref{young_sign} allows for the possibility of describing the fluid-vapor interface by a graph.

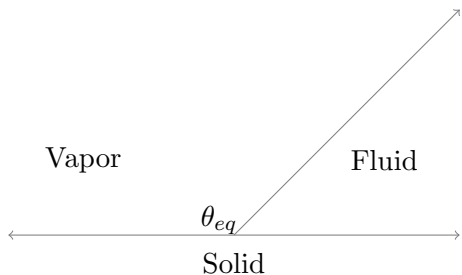
\begin{figure}
\begin{tikzpicture}[xscale=2,yscale=2]
\draw[<->,gray] (-1.5,0) -- (1.5,0);
\draw[->,gray] (0,0) -- (1.5,1.5);

\node[anchor=north] at (0,-.05) {\text{Solid}};

\node at (1,.5) {\text{Fluid}};

\node at (-1,.5) {\text{Vapor}};

\node at (-.1,.1) {$\theta_{eq}$};

\end{tikzpicture}

\caption{Equilibrium contact angle very near the contact point}\label{fig:contact_angle}
\end{figure}

The dynamics at a contact point are a much more complicated issue.  The first issue to deal with in the context of viscous fluids is that the usual no-slip condition ($u=0$ at the fluid-solid interface) combines with the free boundary kinematic equation (the sixth equation in \eqref{system 1}) to disallow contact point motion, which is obvious nonsense.  The no-slip condition must then be replaced with the more general Navier-slip conditions, the fourth and fifth equations in \eqref{system 1}, which do allow the fluid to slip along the interface at the expense of a dissipative frictional force but do not allow the fluid to penetrate the solid.

Much work has gone into the study of contact point (and line) motion: we refer to the surveys of Dussan \cite{dussan} and Blake \cite{blake} for a thorough discussion of theoretical and experimental studies.  What has emerged from these studies is the understanding that deviation of the dynamic contact angle $\theta_{dyn}$ from the equilibrium angle $\theta_{eq}$ (which we recall is determined via Young's relation \eqref{young_relat}) causes the contact point to move in an attempt to return to the equilibrium value.  More precisely, these quantities are related via
 \begin{equation}\label{cl_motion}
  V_{cl} = \vv( \sigma ( \cos(\theta_{dyn}) - \cos(\theta_{eq}) ) ),
 \end{equation}
where $V_{cl}$ is the contact point normal velocity and $\vv$ is the increasing diffeomorphism such that $\vv(0)=0$, mentioned above.  Equations of the form \eqref{cl_motion} have been derived in a number of ways.  Blake-Haynes \cite{blake_haynes} combined thermodynamic and molecular kinetics arguments to arrive at $\vv(z) = A \sinh(Bz)$ for material constants $A,B>0$.  Cox \cite{cox} used matched asymptotic analysis and hydrodynamic arguments to derive \eqref{cl_motion} with a different $\vv$ but of the same general form.  Ren-E \cite{ren_e} performed molecular dynamics simulations to probe the physics near the contact point and also found an equation of the form \eqref{cl_motion}.  Ren-E \cite{ren_e_deriv} also derived \eqref{cl_motion} from constitutive equations and thermodynamic principles.  The last pair of equations in \eqref{system 1} implement \eqref{cl_motion} in the context of the droplet problem.

In recent work, Guo-Tice \cite{guo_tice_QS} studied a version of \eqref{system 1}, coupling the incompressible Stokes equations to the Navier-slip conditions and a contact point equation of the form \eqref{cl_motion}, within the context of the fluid evolving inside a vessel.  They proved that for data starting sufficiently close to equilibrium (measured in an appropriate Sobolev norm), solutions exist globally in time and decay to equilibrium at an exponential rate.  This provides some evidence in support of the model, which was identified in total for the first time by Ren-E \cite{ren_e}, as one expects asymptotic stability of equilibria for most physically meaningful models.  The purpose of the present paper is to further press the Ren-E model by examining its behavior when used in droplet geometries.  These are more complicated than the vessel geometries of \cite{guo_tice_QS} since there is an extra degree of freedom corresponding to the motion of the droplet endpoints.

There has also been much prior work devoted to studying contact lines and points in simplified thin-film models; we will not attempt to enumerate these results here and instead refer to the survey by Bertozzi \cite{bertozzi}.  By contrast, there are relatively few results in the literature related to models in which the full fluid mechanics are considered, and to the best of our knowledge none that allow for both dynamic contact point and dynamic contact angle.  Schweizer \cite{schweizer} studied a 2D Navier-Stokes problem with a fixed contact angle of $\pi/2$.  Bodea \cite{bodea} studied a similar problem with fixed $\pi/2$ contact angle in 3D channels with periodicity in one direction.  Kn\"upfer-Masmoudi \cite{knupfer_masmoudi} studied the dynamics of a 2D drop with fixed contact angle when the fluid is assumed to be governed by Darcy's law.  Related analysis of the fully stationary Navier-Stokes system with free, but unmoving boundary, was carried out in 2D by Solonnikov \cite{solonnikov} with contact angle fixed at $\pi$, by Jin \cite{jin} in 3D with angle $\pi/2$,  and by Socolowsky \cite{socolowsky} for 2D coating problems with fixed contact angles.

\subsection{Energy-dissipation structure}

The system \eqref{system 1} has a natural energy-dissipation structure, which we present in the following theorem, the proof of which we postpone to  Appendix \ref{appendix section 2}.

\begin{theorem}\label{intro theorem 1}
We have
\begin{align}\label{energy-dissipation}
\frac{d}{dt}\left(\int_{L}^{R}\frac{g}{2}\abs{\z}^2+\int_{L}^{R}\sigma\sqrt{1+\abs{\p_{z_1}\z}^2}-[\gamma](R-L)\right)\\
+\left(\int_{\Omega(t)}\frac{\mu}{2}\abs{\dm_z u}^2+\int_{\Sigma_b(t)}\beta\abs{u\cdot\tau}^2+\bigg(\ww(\dt L)\dt L+\ww(\dt R)\dt R\bigg)\right)&=0.\no
\end{align}
\end{theorem}

Note that in light of the assumption \eqref{young_sign} we have that $[\gamma] >0$, and so at first glance it is possible for the energetic term (the term acted on by the time derivative) in Theorem \ref{intro theorem 1} to be negative.  However, this does not happen due to Young's condition \eqref{young_sign}; indeed,
\begin{equation}
\int_{L}^{R}\sigma\sqrt{1+\abs{\p_{z_1}\z}^2}-[\gamma](R-L) \ge \int_L^R \sigma - [\gamma](R-L) = (\sigma - [\gamma])(R-L) >0.
\end{equation}

\subsection{Equilibrium}

Note that in light of Theorem \ref{intro theorem 1} any steady state equilibrium solution to \eqref{system 1} must satisfy $\zeta(t,z_1) = \zeta_0(z_1)$, $u(t,z) =0$, $P(t,z) = P_0$, and $L(t) = L_0$, $R(t) = R_0$, with $\zeta_0$ and $P_0$ satisfying a number of equilibrium conditions.  Given such a solution, we define the equilibrium domain to be the set $\Omega_0=\{(z_1,z_2): L_0\leq z_1\leq R_0,\ \ 0\leq z_2\leq\z_0(z_1)\}$.  Our next result provides for the existence of an equilibrium.

\begin{theorem}\label{intro theorem 2}
Assume \eqref{young_sign}.  Then there exists a smooth equilibrium satisfying $u=0$ and
\begin{align}\label{equilibrium}
\left\{
\begin{array}{l}
g\z_0-\sigma\p_{z_1}\left(\dfrac{\p_{z_1}\z_0}{\sqrt{1+\abs{\p_{z_1}\z_0}^2}}\right)=P_0,\\
\abs{\p_{z_1}\z_0(L_0)}=\abs{\p_{z_1}\z_0(R_0)}=\dfrac{\sqrt{\sigma^2-[\gamma]^2}}{[\gamma]}, \;\;
\z_0(L_0)=\z_0(R_0)=0,\\
\ds\int_{L_0}^{R_0}\z_0(z_1)\ud{z_1}=M,
P_0=C_1\Big(M,g,\sigma,[\gamma]\Big),\ R_0-L_0=C_2\Big(M,g,\sigma,[\gamma]\Big),
\end{array}
\right.
\end{align}
where $C_1\Big(M,g,\sigma,[\gamma]\Big)$ and $C_2\Big(M,g,\sigma,[\gamma]\Big)$ are positive constants that depends on $M,g,\sigma,[\gamma]$.  Moreover, the equilibrium is unique up to a common shift of $L_0$ and $R_0$ and a corresponding translation of $\zeta_0$.
\end{theorem}

The proof of this theorem may be found  in Appendix \ref{appendix section 1}. Note that there is a translation invariance in the equilibrium that does not uniquely determine its location on the surface.  Once one of the endpoints is selected, though, the equilibrium is unique.

\subsection{Geometric reformulation}

In order to analyze the PDE system \eqref{system 1} we will reformulate the equations in the equilibrium domain $\Omega_0$.  The basic idea is to define a time-dependent diffeomorphism transforming the moving domain $\Omega(t)$ into the fixed equilibrium domain $\Omega_0$.   We intend to construct a geometric mapping $\Omega_0\rt\Omega(t): x\rt z$. Without loss of generality, we assume the center of the bottom of the equilibrium domain is located at the origin. Let $\ell=\dfrac{R_0-L_0}{2}$. Then the equilibrium domain is
\begin{align}
\Omega_0:=\{x=(x_1,x_2): -\ell< x_1< \ell,\ \ 0< x_2<\z_0(x_1)\}.
\end{align}
Correspondingly, we denote the equilibrium top
\begin{align}
\Sigma_0:=\{(x_1,x_2): x_2=\z_0(x_1), -\ell< x_1< \ell\},
\end{align}
the equilibrium bottom
\begin{align}
\Sigma_{0b}:=\{(x_1,x_2): x_2=0, -\ell< x_1< \ell\}.
\end{align}

The mapping is be constructed as follows.
\begin{itemize}
\item
Define the mapping from $\tilde\Omega(t)=\{y=(y_1,y_2):-\ell< y_1< \ell,\ \ 0< y_2<\z(t,y_1)\}$ to $\Omega(t)$ as
\begin{align}
\Phi: y_1\rt z_1=\frac{R(t)-L(t)}{2\ell}y_1+\frac{R(t)+L(t)}{2},\ \ y_2\rt z_2=y_2.
\end{align}
This is a dilation in the horizontal direction.
\item
Let the free surface be given as a small perturbation of $\z_0$, i.e. $\z=\z_0+\e$ for some $\e(t,\cdot): \r^+\times [-\ell,\ell]\rt \r$. Firstly, we extend $\e$ from $H^s(L_0,R_0)$ to $H^s(\r)$ by means of a standard extension operator $E : H^s(L_0,R_0) \to H^s(\r)$, which is bounded for all $0\leq s\leq 3$. Then we define the upper harmonic extension of $\e(t,x_1)$ from $x_2=0$ to $x_2\geq0$ as $\be(t,x_1,x_2)$ given by
\begin{align}
\be(t,x):=\pp\Big[E[\e]\Big](t,x_1,\z_0(x_1)-x_2)
\end{align}
for
\begin{equation}\label{poisson_def}
 \mathcal{P}f(z_1,z_2) := \int_{\mathbb{R}} \hat{f}(\xi) e^{-2\pi \abs{\xi}z_2} e^{2\pi i z_1 \xi} d\xi.
\end{equation}
It is easy to check that $\be\Big(t,x_1,\z_0(x_1)\Big)=\pp\Big[E[\e]\Big](t,x_1,0)=\e(t,x_1)$, which means that $\e(t,x_1)=\be(t,x_1,x_2)$ at $\Sigma_0$.
\item
Define the mapping from $\Omega_0$ to $\tilde\Omega(t)$ as
\begin{align}
\Psi: x_1\rt y_1,\ \ x_2\rt y_2=\left(1+\frac{\be(t,x_1,x_2)}{\z_0(x_1)}\right)x_2.
\end{align}
This maps $x_2=\z_0(x_1)$ to $y_2=\z(y_1)$ and $x_2=0$ to $y_2=0$.
\item
Then the composition $\Pi=\Phi\circ\Psi: \Omega_0\rt\Omega(t)$
\begin{align}
\Pi: x_1\rt z_1=\frac{R(t)-L(t)}{2\ell}x_1+\frac{R(t)+L(t)}{2},\ \ x_2\rt z_2=\left(1+\frac{\be(t,x_1,x_2)}{\z_0(x_1)}\right)x_2,
\end{align}
is the desired geometric mapping.
\end{itemize}

Write
\begin{equation}\label{JKA_def}
J_1=\dfrac{R(t)-L(t)}{2\ell},\ \ K_1=\dfrac{1}{J_1},\quad J_2=1+\dfrac{\be}{\z_0}+\dfrac{x_2\p_2\be}{\z_0},\ \ K_2=\dfrac{1}{J_2},\quad A=\left(\dfrac{\p_1\be}{\z_0}-\dfrac{\be\p_1\z_0}{\z_0^2}\right)x_2,
\end{equation}
and define $K_i=\dfrac{1}{J_i}$ with $i=1,2$.  Then the Jacobi and transform matrices are then
\begin{align}
\na\Pi=\left(
\begin{array}{cc}
J_1&0\\
A&J_2
\end{array}
\right)
\text{ and }
\a=(\na\Pi)^{-T}=\left(
\begin{array}{cc}
K_1&-AK\\
0&K_2
\end{array}
\right),
\end{align}
with the Jacobian $J=\det(\na\Pi)=J_1J_2$ and $K= \det(\mathcal{A}) =\dfrac{1}{J}=K_1K_2$.  We will work within a small-energy regime that guarantees that $\Pi$ is a diffeomorphism and $J,K >0$.

In the new coordinates, $u(t,x)$, $P(t,x)$, $\z(t,x_1)$,  $L(t)$, and $R(t)$ satisfy
\begin{align}\label{system 2}
\left\{
\begin{array}{ll}
\na_{\a}\cdot S_{\a}(P,u)=0&\ \ \text{in}\ \ \Omega_0,\\
\na_{\a}\cdot u=0&\ \ \text{in}\ \ \Omega_0,\\
S_{\a}(P,u)\n=g\z\n-\sigma \p_{\a_1}\left(\dfrac{\p_{\a_1}\z}{\sqrt{1+\abs{\p_{\a_1}\z}^2}}\right)\n&\ \ \text{on}\ \ \Sigma_0,\\
\Big(S_{\a}(P,u)\n-\beta u\Big)\cdot\t=0&\ \ \text{on}\ \ \Sigma_{0b},\\
u\cdot\n=0&\ \ \text{on}\ \ \Sigma_{0b},\\
\dt\z-K_1\tilde a\p_1\z+u_1\p_{\a_1}\z-u_2=0&\ \ \text{on}\ \ \Sigma_0,\\
\ww(\dt L)=\sigma\dfrac{1}{\sqrt{1+\abs{\p_{\a_1}\z}^2}}\Bigg|_{x_1=-\ell}-[\gamma],\\
\ww(\dt R)=[\gamma]-\sigma\dfrac{1}{\sqrt{1+\abs{\p_{\a_1}\z}^2}}\Bigg|_{x_1=\ell},
\end{array}
\right.
\end{align}
where here
\begin{align}\label{intro 2}
\tilde a=\dfrac{\dt R-\dt L}{2\ell}x_1+\dfrac{\dt R+\dt L}{2}
\end{align}
and
$\na_{\a}$, $\de_{\a}$ and $\dm_{\a}$ are weighted operators defined as follows with Einstein summation employed
\begin{equation}
(\nabla_{\mathcal{A}}f)_i=\mathcal{A}_{ij}\partial_jf, \;
\nabla_{\mathcal{A}}\cdot\vec g=\mathcal{A}_{ij}\partial_jg_i, \;
\Delta_{\mathcal{A}}f=\nabla_{\mathcal{A}}\cdot\nabla_{\mathcal{A}}f, \text{ and }
(\mathbb{D}_{\mathcal{A}}u)_{ij}=\mathcal{A}_{ik}\partial_ku_j+\mathcal{A}_{jk}\partial_ku_i.
\end{equation}
The tensor $S_{\a}(P,u)=PI-\mu\dm_{\a}u$ satisfies $\na_{\a}\cdot S_{\a}(P,u)=-\mu\de_{\a}u+\na_{\a}P$ whenever $u$ is such that $\na_{\a}\cdot u=0$.  Also, $\p_{\a_1}=K_1\p_1$ denotes the weighted derivative in $x_1$ direction. Moreover, $\n=J\a\nu_0$ and $\t=J\a\tau_0$, where $\nu_0$ and $\tau_0$ are the unit normal and tangential vectors on $\p\Omega_0$. In particular, on $\Sigma_0$,
\begin{align}
\nu_0=\dfrac{(-\p_1\z_0,1)}{\sqrt{1+\abs{\p_1\z_0}^2}},\quad\tau_0=\dfrac{(1,\p_1\z_0)}{\sqrt{1+\abs{\p_1\z_0}^2}}.
\end{align}
We may directly verify that
\begin{align}\label{intro 3}
\n=\frac{(-\p_1\z,J_1)}{\sqrt{1+\abs{\p_1\z_0}^2}}\ \ \text{on}\ \ \Sigma_0.
\end{align}

The transport equation can be rewritten as
\begin{align}\label{intro 4}
\dt\z-K_1\tilde a\p_1\z=K_1(u\cdot\n)\sqrt{1+\abs{\p_1\z_0}^2},
\end{align}
or
\begin{align}\label{intro 5}
J_1\dt\z-\tilde a\p_1\z=(u\cdot\n)\sqrt{1+\abs{\p_1\z_0}^2}.
\end{align}
Note that the mass conservation equation in the new coordinates reads
\begin{align}\label{intro 1}
M=\int_{-\ell}^{\ell}J_1\z(t,x_1)\ud{x_1}.
\end{align}

In these new coordinates, there is a natural variant of energy-dissipation structure of Theorem \ref{intro theorem 1}.  For the sake of brevity we will not prove this here but merely state it.
\begin{theorem}\label{intro theorem 3}
We have
\begin{align}
\dt\left(\int_{\Omega_0}\frac{J}{2}\abs{u}^2+\int_{-\ell}^{\ell}\frac{gJ_1}{2}\abs{\z}^2+\int_{-\ell}^{\ell}\sigma J_1\sqrt{1+K_1^2\abs{\p_{1}\z}^2}-[\gamma](R-L)\right)\\
+\left(\int_{\Omega_0}\frac{\mu J}{2}\abs{\dm_{\a} u}^2+\int_{\Sigma_{0b}}\beta\abs{u\cdot\tau_0}^2+\Big(\ww(\dt L)\dt L+\ww(\dt R)\dt R\Big)\right)&=0.\no
\end{align}
\end{theorem}

\subsection{Perturbation form}

We will consider solutions to the full problem as perturbations around the equilibrium state $(0,P_0,\z_0,L_0,R_0)$ given by Theorem \ref{intro theorem 1}. We will now reformulate the equations
in \eqref{system 2} in terms of the perturbed unknowns.

Define the perturbation variables
\begin{align}
p:=P-P_0,\ \ \e:=\z-\z_0,\ \ l:=L+\ell,\ \ r:=R-\ell.
\end{align}
Define $\kappa=\ww'(0)>0$, and let
\begin{align}
\hat\ww(z)=\frac{1}{\kappa}\ww(z)-z.
\end{align}

Since
\begin{align}
\p_{\a_1}\left(\dfrac{\p_{\a_1}\z}{\sqrt{1+\abs{\p_{\a_1}\z}^2}}\right)=\p_{1}\left(\dfrac{K_1^2\Big(\p_{1}\z_0+\p_1\e\Big)}{\sqrt{1+K_1^2\Big(\p_{1}\z_0+\p_1\e\Big)^2}}\right),
\end{align}
we may linearize
\begin{align}\label{R_def}
\p_{\a_1}\left(\dfrac{\p_{\a_1}\z}{\sqrt{1+\abs{\p_{\a_1}\z}^2}}\right)=\dfrac{\p_{1}\z_0}{\sqrt{1+\abs{\p_{1}\z_0}^2}}
+\dfrac{k_1\p_{1}\z_0}{\sqrt{1+\abs{\p_{1}\z_0}^2}}+\dfrac{k_1\p_{1}\z_0+\p_1\e}{\Big(\sqrt{1+\abs{\p_{1}\z_0}^2}\Big)^3}+\rr,
\end{align}
where
\begin{align}\label{k1_def}
k_1=K_1-1=\frac{2\ell}{R-L}-1=-\frac{r-l}{2\ell+r-l},
\end{align}
and the nonlinear terms are included in $\rr$.  See Appendix \ref{rqso_appendix} for an alternate expansion of $\rr$.

Similarly, since
\begin{align}
\dfrac{1}{\sqrt{1+\abs{\p_{\a_1}\z}^2}}=\dfrac{1}{\sqrt{1+K_1^2\Big(\p_{1}\z_0+\p_1\e\Big)^2}},
\end{align}
we may linearize
\begin{align}\label{Q_def}
\dfrac{1}{\sqrt{1+\abs{\p_{\a_1}\z}^2}}=\dfrac{1}{\sqrt{1+\abs{\p_{1}\z_0}^2}}
-\dfrac{k_1\abs{\p_1\z_0}^2+\p_1\z_0\p_1\e}{\Big(\sqrt{1+\abs{\p_{1}\z_0}^2}\Big)^3}+\qq,
\end{align}
where $\qq$ contains the nonlinear terms.  Again we refer to Appendix \ref{rqso_appendix} for an alternate expansion of $\qq$.

For the convenience of linearized energy-dissipation estimates, we intend to connect $k_1$ with $\tilde a$. Define a modification of $\tilde a$ as
\begin{align}
a=\frac{2\ell(\dt r-\dt l)}{(2\ell+r-l)^2}x_1+\dfrac{\dt r+\dt l}{2},
\end{align}
which satisfies the relation $\dt k_1=-\p_1 a$, and let
\begin{align}\label{O_def}
\oo=a-\tilde a=-\frac{(r-l)(4\ell+r-l)}{2(2\ell+r-l)^2}(\dt r-\dt l)x_1
\end{align}
represent the nonlinear perturbation terms. We may linearize the transport equation in \eqref{system 2}
\begin{align}\label{S_def}
\dt\e- a\p_1\z_0=(u\cdot\n)\sqrt{1+\abs{\p_1\z_0}^2}+\ss,
\end{align}
where $\ss$ contains the nonlinear terms.  See Appendix \ref{rqso_appendix} for another expression of $\ss$.

Considering the conservation of mass
\begin{align}
M=&\int_{-\ell}^{-\ell}\z_0(x_1)\ud{x_1}=\int_{-\ell}^{\ell}J_1\z(t,x_1)\ud{x_1}=\int_{-\ell}^{\ell}J_1\Big(\z_0(x_1)+\e(t,x_1)\Big)\ud{x_1},
\end{align}
we have
\begin{align}\label{linearize 2}
J_1\int_{-\ell}^{\ell}\e(t,x_1)\ud{x_1}=\int_{-\ell}^{-\ell}(1-J_1)\z_0(x_1)\ud{x_1},
\text{ or }
\int_{-\ell}^{\ell}\e(t,x_1)\ud{x_1}=\int_{-\ell}^{-\ell}(K_1-1)\z_0(x_1)\ud{x_1}.
\end{align}
Hence, we know the linearization of mass conservation:
\begin{align}\label{linearize mass}
\int_{-\ell}^{\ell}\e=k_1\int_{-\ell}^{\ell}\z_0=k_1M.
\end{align}

Therefore, we can rewrite the system \eqref{system 2} in the linearized variables $u(t,x)$, $p(t,x)$ and $\e(t,x_1)$ with $l(t)$ and $r(t)$:
\begin{eqnarray}\label{system 4}
\left\{
\begin{array}{ll}
\na_{\a}\cdot S_{\a}(p,u)=0&\ \ \text{in}\ \ \Omega_0,\\
\na_{\a}\cdot u=0&\ \ \text{in}\ \ \Omega_0,\\
S_{\a}(p,u)\n=g\e\n-\sigma \p_{1}\left(\dfrac{k_1\p_{1}\z_0}{\sqrt{1+\abs{\p_{1}\z_0}^2}}+\dfrac{k_1\p_{1}\z_0+\p_1\e}{\Big(\sqrt{1+\abs{\p_{1}\z_0}^2}\Big)^3}+\rr\right)\n&\ \ \text{on}\ \ \Sigma_0,\\
\Big(S_{\a}(p,u)\n-\beta u\Big)\cdot\t=0&\ \ \text{on}\ \ \Sigma_{0b},\\
u\cdot\n=0&\ \ \text{on}\ \ \Sigma_{0b},\\
\dt\e- a\p_1\z_0=(u\cdot\n)\sqrt{1+\abs{\p_1\z_0}^2}+\ss&\ \ \text{on}\ \ \Sigma_0,\\
\kappa\dt l+\kappa\hat\ww(\dt l)=\sigma\left(-\dfrac{k_1\abs{\p_1\z_0}^2+\p_1\z_0\p_1\e}{\Big(\sqrt{1+\abs{\p_{1}\z_0}^2}\Big)^3}+\qq\right)\bigg|_{x_1=-\ell},\\
\kappa\dt r+\kappa\hat\ww(\dt r)=-\sigma\left(-\dfrac{k_1\abs{\p_1\z_0}^2+\p_1\z_0\p_1\e}{\Big(\sqrt{1+\abs{\p_{1}\z_0}^2}\Big)^3}+\qq\right)\bigg|_{x_1=\ell}.
\end{array}
\right.
\end{eqnarray}

\section{Main results and discussion}

\subsection{Energy and dissipation}

In order to state our main results we must first define a number of energy and dissipation functionals. We define the basic or parallel (since temporal derivatives are the only ones parallel to the boundary) energy as
\begin{align}
\enp=&\sum_{j=0}^2\hms{\dt^j\e}{1}^2,
\end{align}
and the basic dissipation as
\begin{align}
\dipt=&\sum_{j=0}^2\hm{\dt^ju}{1}^2+\sum_{j=0}^2\nm{\dt^ju}_{H^0(\Sigma_{0b})}^2+\sum_{j=0}^2\bigg(\abs{\dt^{j+1}l}^2+\abs{\dt^{j+1}r}^2\bigg).
\end{align}
We also define the improved basic dissipation as
\begin{align}
\dip=&\dipt+\sum_{j=0}^2\hm{\dt^jp}{0}^2+\sum_{j=0}^2\hms{\dt^j\e}{\frac{3}{2}}^2\\
&+\sum_{j=0}^2\bigg(\abs{\dt^{j}\p_1\e(-\ell)}^2+\abs{\dt^{j}\p_1\e(\ell)}^2\bigg)+\sum_{j=0}^2\bigg(\abs{\dt^{j}u(-\ell,0)\cdot\n}^2+\abs{\dt^{j}u(\ell,0)\cdot\n}^2\bigg).\no
\end{align}
The basic energy and dissipation arise through a version of the energy-dissipation relation \eqref{energy-dissipation}. However, once we control these terms we are then able to control much more. This extra control is encoded in the full energy and dissipation, which are defined as follows:
\begin{align}
\en=&\enp+\hmw{u}{2}^2+\hm{\dt u}{1}^2+\hmw{p}{1}^2+\hm{\dt p}{0}^2\\
&+\hmwss{\e}{\frac{5}{2}}^2+\hms{\dt\e}{\frac{3}{2}}^2+\sum_{j=0}^1\bigg(\abs{\dt^{j+1}l(t)}^2+\abs{\dt^{j+1}r(t)}^2\bigg)\no\\
&+\sum_{j=0}^1\bigg(\abs{\dt^{j}\p_1\e(-\ell)}^2+\abs{\dt^{j}\p_1\e(\ell)}^2\bigg)+\sum_{j=0}^1\bigg(\abs{\dt^{j}u(-\ell,0)\cdot\n}^2+\abs{\dt^{j}u(\ell,0)\cdot\n}^2\bigg),\no
\end{align}
and
\begin{align}
\di=&\dip+\hmw{u}{2}^2+\hmw{\dt u}{2}^2+\hmw{p}{1}^2+\hmw{\dt p}{1}^2\\
&+\hmwss{\e}{\frac{5}{2}}^2+\hmwss{\dt\e}{\frac{5}{2}}^2+\hmwss{\dt^2\e}{\frac{3}{2}}^2+\hmwss{\dt^3\e}{\frac{1}{2}}^2.\no
\end{align}
The spaces $W^r_{\d}$ are weighted Sobolev spaces, as defined in Appendix A, for a fixed weight parameter $\d\in(0,1)$.

\subsection{Main theorems}

Our main result establishes the existence of global-in-time solutions that decay to equilibrium at an exponential rate.

\begin{theorem}\label{intro main}
Fix $0 < \delta < 1$.  There exists a universal constant $\vartheta>0$ such that if $\en(0)\leq\vartheta,$ then there exists a unique global solution $(u,p,\e)$ to \eqref{system 4} such that
\begin{align}
\sup_{t\geq0}\Bigg(\en(t)+\ue^{\lambda t}\bigg(\enp(t)+\hm{u(t)}{1}^2+\nm{u(t)\cdot\tau_0}_{H^0(\Sigma_{0b})}^2+\hm{p(t)}{0}^2
+\abs{\dt l(t)}^2+\abs{\dt r(t)^2}\\
+\abs{\p_1\e(t,-\ell)}^2+\abs{\p_1\e(t,\ell)}^2+\abs{u(t,-\ell,0)\cdot\n}^2+\abs{u(t,\ell,0)\cdot\n}^2\bigg)\Bigg)+\int_0^{\infty}\di(t)\ud{t}\ls&\en(0).\no
\end{align}
\end{theorem}

Some remarks are in order.  First, the result can be interpreted as saying that the equilibrium constructed in Theorem \ref{intro theorem 2} is asymptotically stable in the dynamics \eqref{system 4}.  Second, although the equilibrium droplet is asymptotically stable in the sense described in the theorem (namely, in the fixed coordinate system), its final location in Eulerian coordinates does not have to coincide with where it began.  Indeed, let $X=(X_1,X_2)$ be the mass center of the fluid. Then we know
\begin{align}
\dt X=\frac{1}{M}\int_{\Omega_0}J(t)u(t,x)\ud{x}.
\end{align}
Based on the exponential decay of $u$, we know there exists $\lambda>0$ such that
\begin{align}
\ue^{\lambda t}\abs{\dt X}\ls \en(0).
\end{align}
Therefore, we know
\begin{align}
\abs{X(\infty)-X(0)}\ls \frac{\en(0)}{\lambda}<\infty.
\end{align}
In other words, the mass center is shifted for a finite distance.  This is consistent with the fact that the equilibrium is only unique once one of its endpoints is fixed.  The translation invariance gives a one-parameter family of equilibria, and our result then says that this family is stable.

Third, the theorem is valid for any weight parameter $0 < \delta < 1$.  This is in contrast with the result in \cite{guo_tice_QS} for the vessel problem, which required $\delta$ to be tuned to the equilibrium contact angle.  The reason for this is that in the present paper, in order to employ the graph formulation for the droplet we have had to enforce \eqref{young_sign}, which restricts to $\theta_{eq} \in (\pi/2,\pi)$.  In this range there are no restrictions placed on $\delta \in (0,1)$ in the weighted elliptic theory.  It is only for acute $\theta_{eq}$ that restrictions are needed, as in \cite{guo_tice_QS}.

Our strategy for proving Theorem \ref{intro main} is in broad strokes the same as that employed for the vessel problem in \cite{guo_tice_QS}: we use a nonlinear energy method based on the physical energy-dissipation structure of \ref{intro theorem 3}, coupled to weighted elliptic estimates in the equilibrium domain.  This results in a closed scheme of a priori estimates (Theorems \ref{main theorem 1} and \ref{main theorem 2}) that couples to a local existence theory (Theorem \ref{local theorem}) via a continuation argument to give global decaying solutions.

All of the difficulties (other than the $\delta$ restriction issue described above) from the vessel problem carry over to the droplet problem, so we refer to the introduction of \cite{guo_tice_QS} for a summary of these and the strategy for dealing with them.  There are two principal new difficulties caused by the droplet geometry.  The first is due to the extra degree of freedom present in the horizontal motion of the droplet endpoints.  In our geometric coordinate system this motion is manifest in the first part of the Jacobian, $J_1$, due to dilation and contraction of the interval $(L(t),R(t))$.  While the dissipation provides control of the time derivatives of $L,R$, it is not clear from Theorem \ref{intro theorem 3} that dissipative control of $J_1$ is possible.  This is analogous to the problem of controlling $\eta$ with the dissipation when it naturally only controls $\dt \eta$.  Acquiring dissipative control of $J_1$ and $\eta$, which couple together in a nontrivial way, is one of the key steps in our a priori estimates.  The second new difficulty is manifest in the fact that the free surface graph intersects the flat line on which the droplet resides.  This causes technical problems with the geometric coefficients near the contact points, and we must resort to delicate analysis to handle these.

The rest of the paper is organized as follows.  In Section \ref{sec:linear_analysis} we develop essential linearized analysis tools.  In Section \ref{sec:basic_estimates} we record various auxiliary estimates that couple directly to the energy-dissipation structure to give enhanced control.  In Section \ref{nonlinear section} we estimate the various nonlinearities appearing in the energy-dissipation structure.  For the sake of brevity we have attempted to borrow as many of these as possible from \cite{guo_tice_QS}, but many terms in the droplet problem do not appear in the vessel problem and require delicate analysis.  In Section \ref{sec:nlin_stokes} we estimate the various nonlinear terms appearing in the elliptic estimates.  Again, we have borrowed what we could from \cite{guo_tice_QS}, but much new work was needed.  Finally, in Section \ref{sec:gwp_dec} we synthesize our a priori estimates and complete the proof of Theorem \ref{intro main}.

\section{Linear analysis}\label{sec:linear_analysis}

In this section we analyze the problem \eqref{system 4} and its time derivatives in linear form.  To begin we record the form of \eqref{system 4} when time derivatives are applied.  Upon doing this we find that $v(t,x):=\dt^mu(t,x)$, $q(t,x):=\dt^mp(t,x)$, $\ep(t,x_1):=\dt^m\e(t,x_1)$, $\lp(t):=\dt^ml(t)$ and $\rp(t):=\dt^mr(t)$ satisfy
\begin{align}\label{system 3}
\left\{
\begin{array}{ll}
\na_{\a}\cdot S_{\a}(q,v)=\s_1&\ \ \text{in}\ \ \Omega_0,\\
\na_{\a}\cdot v=\s_2&\ \ \text{in}\ \ \Omega_0,\\
S_{\a}(q,v)\n=g\ep\n-\sigma \p_1\left(\dfrac{\kp_1\p_{1}\z_0}{\sqrt{1+\abs{\p_{1}\z_0}^2}}+\dfrac{\kp_1\p_{1}\z_0+\p_1\ep}{\Big(\sqrt{1+\abs{\p_{1}\z_0}^2}\Big)^3}+\dt^m\rr\right)\n+\s_3&\ \ \text{on}\ \ \Sigma_0,\\
\Big(S_{\a}(q,v)\n-\beta v\Big)\cdot\t=\s_4&\ \ \text{on}\ \ \Sigma_{0b},\\
v\cdot\n=0&\ \ \text{on}\ \ \Sigma_{0b},\\
\dt\ep-\ap\p_1\z_0=(v\cdot\n)\sqrt{1+\abs{\p_1\z_0}^2}+\dt^m\ss+\s_5&\ \ \text{on}\ \ \Sigma_0,\\
\kappa\dt\lp=\sigma\left(-\dfrac{\kp_1\abs{\p_1\z_0}^2+\p_1\z_0\p_1\ep}{\Big(\sqrt{1+\abs{\p_{1}\z_0}^2}\Big)^3}+\dt^m\qq\right)\bigg|_{-\ell}+\s_6,\\
\kappa\dt\rp=-\sigma\left(-\dfrac{\kp_1\abs{\p_1\z_0}^2+\p_1\z_0\p_1\ep}{\Big(\sqrt{1+\abs{\p_{1}\z_0}^2}\Big)^3}+\dt^m\qq\right)\bigg|_{\ell}+\s_7,
\end{array}
\right.
\end{align}
where $\kp_1:=\dt^mk_1$ and $\ap:=\dt^m a$ satisfying $\p_1\ap=-\dt\kp_1$, with the mass conservation
\begin{align}\label{linearize 1}
\int_{-\ell}^{\ell}\ep=\kp_1\int_{-\ell}^{\ell}\z_0.
\end{align}
Here, the nonlinear terms $\s_1-\s_7$ are defined in Appendix \ref{ap section 1}.

In what follows we will need the following spaces.  We define the time-dependent spaces
\begin{align}
\hd=&\left\{u:\Omega_0\rt\r^2\ \bigg|\ \int_{\Omega_0}\frac{\mu J}{2}\abs{\dm_{\a}u}^2+\int_{\Sigma_{0b}}\beta\abs{u\cdot\tau_0}^2<\infty,\ \ u\cdot\n=0\ \ \text{on}\ \ \Sigma_{0b}\right\},
\end{align}
and
\begin{align}
\w(t)=&\left\{v\in\hd\ |\ v\cdot\n\in H^1(-\ell,\ell)\right\},\\
\v(t)=&\left\{v\in \w\ |\ \na_{\a}\cdot v=0\right\}.
\end{align}

\subsection{Weak formulation}

We now aim to justify a weak formulation of \eqref{system 3}.

\begin{lemma}\label{linear theorem 1}
Suppose that $\e$ is given and $\a$ and $\n$ are determined in terms of $\e$. If $(v,q,\ep,\lp,\rp)$ are sufficiently regular and satisfy \eqref{system 3}, then for any $w\in \w(t)$,
\begin{align}\label{linearize 4}
&\int_{\Omega_0}\frac{\mu J}{2}\dm_{\a}v:\dm_{\a}w-\int_{\Omega_0}Jq(\na_{\a}\cdot w)+\int_{\Sigma_{0b}}\beta(v\cdot\tau_0)(w\cdot\tau_0)+\int_{\Sigma_0}g\ep(w\cdot\n)\\
&+\int_{\Sigma_0}\kp_1(P_0-g\z_0)(w\cdot\n)+ \int_{\Sigma_0} \sigma\left(\dfrac{\kp_1\p_{1}\z_0+\p_1\ep}{\Big(\sqrt{1+\abs{\p_{1}\z_0}^2}\Big)^3}\right)\p_1(w\cdot\n)\no\\
&+\bigg(-\frac{\kappa}{\p_1\z_0}\dt\rp(w\cdot\n)\bigg|_{\ell}-\frac{\kappa}{\p_1\z_0}\dt\lp(w\cdot\n)\bigg|_{-\ell}\bigg)\no\\
=&\int_{\Omega_0}Jw\cdot \s_1-\int_{\Sigma_0}w\cdot\s_3-\int_{\Sigma_{0b}}\s_4\bigg(w\cdot\frac{\t}{\abs{\t}^2}\bigg)- \int_{\Sigma_0} \sigma\p_1\dt^m\rr(w\cdot\n)\no\\
&+\frac{\sigma}{\p_1\z_0}\dt^m\qq(w\cdot\n)\bigg|_{-\ell}^{\ell}
-\frac{1}{\p_1\z_0}\s_7(w\cdot\n)\bigg|_{\ell}-\frac{1}{\p_1\z_0}\s_6(w\cdot\n)\bigg|_{-\ell}.\no
\end{align}
\end{lemma}
\begin{proof}
Multiplying $Jw$ on both sides of the Stokes equation and integrating over $\Omega_0$ imply
\begin{align}
\int_{\Omega_0}J\Big(\na_{\a}\cdot S_{\a}(q,v)\Big)\cdot w=&\int_{\Omega_0}Jw\cdot \s_1.\no
\end{align}
Using the simple identity
\begin{align}
J\a\nu_0=\left\{
\begin{array}{ll}
J\nu_0&\ \ \text{on}\ \ \Sigma_{0b},\\
\n&\ \ \text{on}\ \ \Sigma_{0},
\end{array}
\right.
\end{align}
and integrating by parts yields
\begin{align}
\int_{\Omega_0}J\Big(\na_{\a}\cdot S_{\a}(q,v)\Big)\cdot w
=&\bigg(\int_{\Omega_0}\frac{\mu J}{2}\dm_{\a}v:\dm_{\a}w-\int_{\Omega_0}Jq(\na_{\a}\cdot w)\bigg)\\
&+\int_{\Sigma_{0b}}\Big(S_{\a}(q,v)\n\Big)\cdot w+\int_{\Sigma_0}\Big(S_{\a}(q,v)\n\Big)\cdot w=I_1+I_2+I_3.\no
\end{align}
We may then simplify
\begin{align}
I_2=&\int_{\Sigma_{0b}}\Big(S_{\a}(q,v)\n\Big)\cdot w\\
=&\int_{\Sigma_{0b}}\bigg(\Big(S_{\a}(q,v)\n\Big)\cdot \n\bigg)\bigg(w\cdot\frac{\n}{\abs{\n}^2}\bigg)
+\int_{\Sigma_{0b}}\bigg(\Big(S_{\a}(q,v)\n\Big)\cdot \t\bigg)\bigg(w\cdot\frac{\t}{\abs{\t}^2}\bigg)\no\\
=&\int_{\Sigma_{0b}}\bigg(\Big(S_{\a}(q,v)\n\Big)\cdot \t\bigg)\bigg(w\cdot\frac{\t}{\abs{\t}^2}\bigg)=\int_{\Sigma_{0b}}\beta\bigg(v\cdot\frac{\t}{\abs{\t}}\bigg)\bigg(w\cdot\frac{\t}{\abs{\t}}\bigg)
+\int_{\Sigma_{0b}}\s_4\bigg(w\cdot\frac{\t}{\abs{\t}^2}\bigg)\no\\
=&\int_{\Sigma_{0b}}\beta(v\cdot\tau_0)(w\cdot\tau_0)+\int_{\Sigma_{0b}}\s_4\bigg(w\cdot\frac{\t}{\abs{\t}^2}\bigg).\no
\end{align}
On the other hand, we may decompose
\begin{align}
I_3=\int_{\Sigma_0}\Big(S_{\a}(q,v)\n\Big)\cdot w=\int_{\Sigma_0}g\ep(w\cdot\n)+H+\int_{\Sigma_0}w\cdot\s_3,
\end{align}
where the equilibrium equation \eqref{equilibrium} and integrating by parts yield
\begin{align}
H=&-\int_{\Sigma_0} \sigma \p_1\left(\dfrac{\kp_1\p_{1}\z_0}{\sqrt{1+\abs{\p_{1}\z_0}^2}}+\dfrac{\kp_1\p_{1}\z_0+\p_1\ep}{\Big(\sqrt{1+\abs{\p_{1}\z_0}^2}\Big)^3}+\dt^m\rr\right)(w\cdot\n)\\
=&\int_{\Sigma_0}\kp_1(P_0-g\z_0)(w\cdot\n)+ \int_{\Sigma_0} \sigma\left(\dfrac{\kp_1\p_{1}\z_0+\p_1\ep}{\Big(\sqrt{1+\abs{\p_{1}\z_0}^2}\Big)^3}\right)\p_1(w\cdot\n)
- \int_{\Sigma_0} \sigma\p_1\dt^m\rr(w\cdot\n)\no\\
&-\sigma \left(\dfrac{\kp_1\p_{1}\z_0+\p_1\ep}{\Big(\sqrt{1+\abs{\p_{1}\z_0}^2}\Big)^3}\right)(w\cdot\n)\bigg|_{-\ell}^{\ell}.\no
\end{align}
In particular, using the contact point equation, we have
\begin{align}
&-\sigma \left(\dfrac{\kp_1\p_{1}\z_0+\p_1\ep}{\Big(\sqrt{1+\abs{\p_{1}\z_0}^2}\Big)^3}\right)(w\cdot\n)\bigg|_{-\ell}^{\ell}\\
=&\frac{1}{\p_1\z_0}\bigg(-\kappa\dt\rp-\sigma\dt^m\qq+\s_7\bigg)(w\cdot\n)\bigg|_{\ell}+\frac{1}{\p_1\z_0}\bigg(-\kappa\dt\lp+\sigma\dt^m\qq+\s_6\bigg)(w\cdot\n)\bigg|_{-\ell}\no\\
=&\bigg(-\frac{\kappa}{\p_1\z_0}\dt\rp(w\cdot\n)\bigg|_{\ell}-\frac{\kappa}{\p_1\z_0}\dt\lp(w\cdot\n)\bigg|_{-\ell}\bigg)\no\\
&+\frac{1}{\p_1\z_0}\bigg(-\sigma\dt^m\qq+\s_7\bigg)(w\cdot\n)\bigg|_{\ell}+\frac{1}{\p_1\z_0}\bigg(\sigma\dt^m\qq+\s_6\bigg)(w\cdot\n)\bigg|_{-\ell}.\no
\end{align}
\ \\
Collecting all of the above terms, we have the weak formulation (\ref{linearize 4}).
\end{proof}


\subsection{Linearized energy-dissipation structure}

We have the following equation for the evolution of the energy of $(v,q,\ep,\lp,\rp)$:

\begin{theorem}\label{linear theorem 2}
Suppose that $\e$ is given and $\a$ and $\n$ are determined in terms of $\e$. If $(v,q,\ep,\lp,\rp)$ satisfy \eqref{system 3}, we have
\begin{align}\label{linearize 5}
&\dt\left(\int_{-\ell}^{\ell}\frac{\sigma}{2}\dfrac{\Big(\kp_1\p_{1}\z_0
+\p_1\ep\Big)^2}{\Big(\sqrt{1+\abs{\p_{1}\z_0}^2}\Big)^3}+\frac{\kp_1^2}{2}\left(P_0M+\sigma\int_{-\ell}^{\ell}\dfrac{\abs{\p_{1}\z_0}^2}{\sqrt{1+\abs{\p_{1}\z_0}^2}}\right)
+\frac{g}{2}\int_{-\ell}^{\ell}\left(\kp_1\z_0-\ep\right)^2\right)  \\
&+\left(\int_{\Omega_0}\frac{J\mu}{2}\abs{\dm_{\a} v}^2+\int_{\Sigma_{0b}}\beta(v\cdot\tau_0)^2+\kappa\Big((\dt\lp)^2+(\dt\rp)^2\Big)\right)\no\\
=&\int_{\Omega_0}Jv\cdot\s_1+\int_{\Omega_0}Jq\s_2-\int_{\Sigma_0}v\cdot\s_3-\int_{\Sigma_{0b}}\s_4\bigg(v\cdot\frac{\t}{\abs{\t}^2}\bigg)\no\\
&-\sigma\int_{\Sigma_0}\p_1\dt^m\rr(v\cdot\n)
+\bigg(-\sigma\dt^m\qq\Big(\dt\lp+\dt^m\oo\Big)\bigg|_{\ell}+\sigma\dt^m\qq\Big(\dt\rp+\dt^m\oo\Big)\bigg|_{-\ell}\bigg)\no\\
&-\kappa\bigg(\dt\rp\dt^m\oo\bigg|_{\ell}+\dt\lp\dt^m\oo\bigg|_{-\ell}\bigg)
-\s_7\Big(\dt\rp+\dt^m\oo\Big)\bigg|_{\ell}-\s_6\Big(\dt\lp+\dt^m\oo\Big)\bigg|_{-\ell}\no\\
&+\int_{-\ell}^{\ell}\left(g\ep-\p_1\left(\dfrac{\kp_1\p_{1}\z_0}{\sqrt{1+\abs{\p_{1}\z_0}^2}}
+\dfrac{\kp_1\p_{1}\z_0+\p_1\ep}{\Big(\sqrt{1+\abs{\p_{1}\z_0}^2}\Big)^3}\right)\right)(\dt^m\ss+\s_5) \no.
\end{align}
\end{theorem}
\begin{proof}
Letting $w=v$ in the weak formulation \eqref{linearize 4}, we know
\begin{align}
&\int_{\Omega_0}\frac{\mu J}{2}\abs{\dm_{\a}v}^2+\int_{\Sigma_{0b}}\beta\abs{v\cdot\tau_0}^2)+\int_{\Sigma_0}g\ep(v\cdot\n)\\
&+\int_{\Sigma_0}\kp_1(p_0-g\z_0)(v\cdot\n)+ \int_{\Sigma_0} \sigma\left(\dfrac{\kp_1\p_{1}\z_0+\p_1\ep}{\Big(\sqrt{1+\abs{\p_{1}\z_0}^2}\Big)^3}\right)\p_1(v\cdot\n)\no\\
&+\bigg(-\frac{\kappa}{\p_1\z_0}\dt\rp(v\cdot\n)\bigg|_{\ell}-\frac{\kappa}{\p_1\z_0}\dt\lp(v\cdot\n)\bigg|_{-\ell}\bigg)\no\\
=&\int_{\Omega_0}Jv\cdot \s_1+\int_{\Omega_0}Jq\s_2-\int_{\Sigma_0}v\cdot\s_3-\int_{\Sigma_{0b}}\s_4\bigg(v\cdot\frac{\t}{\abs{\t}^2}\bigg)- \int_{\Sigma_0} \sigma\p_1\dt^m\rr(v\cdot\n)\no\\
&+\frac{\sigma}{\p_1\z_0}\dt^m\qq(v\cdot\n)\bigg|_{-\ell}^{\ell}
-\frac{1}{\p_1\z_0}\s_7(v\cdot\n)\bigg|_{\ell}-\frac{1}{\p_1\z_0}\s_6(v\cdot\n)\bigg|_{-\ell}.\no
\end{align}
We will focus on the simplification of gravitational, surface tension, and contact point terms. We may directly verify the relation
\begin{align}\label{linearize 6}
\p_1\ap=-\dt \kp_1,
\end{align}
which will be used frequently in the following.

\emph{Step 1: Gravitational term:}
We may decompose
\begin{align}
G:=&\int_{\Sigma_0}g\ep(v\cdot\n)=\int_{-\ell}^{\ell}g\ep(v\cdot\n)\sqrt{1+\abs{\p_1\z_0}^2}=\int_{-\ell}^{\ell}g\ep\Big(\dt\ep-\ap\p_1\z_0-\dt^m\ss-\s_5\Big)\\
=&G_1+G_2-\int_{-\ell}^{\ell}g\ep(\dt^m\ss+\s_5).\no
\end{align}
Integration by parts and using \eqref{linearize 6}, we obtain
\begin{align}
G_1=&\int_{-\ell}^{\ell}g\ep\dt\ep=\dt\left(\int_{-\ell}^{\ell}\frac{g}{2}\abs{\ep}^2\right),
\end{align}
\begin{align}
G_2=&-\int_{-\ell}^{\ell}g\ep\ap\p_1\z_0=\int_{-\ell}^{\ell}g\ap\z_0\p_1\ep+\int_{-\ell}^{\ell}g\z_0\ep\p_1\ap=\int_{-\ell}^{\ell}g\ap\z_0\p_1\ep-\int_{-\ell}^{\ell}g\dt \kp_1\z_0\ep=A_1+A_2.
\end{align}
We cannot simplify $A_1$ and $A_2$ at this stage and have to wait until the surface tension terms.

\emph{Step 2: Surface tension term - First stage:}
The transport equation in \eqref{system 3} and integrating by parts yield
\begin{align}
H_1:=& \int_{\Sigma_0} \sigma\left(\dfrac{\kp_1\p_{1}\z_0+\p_1\ep}{\Big(\sqrt{1+\abs{\p_{1}\z_0}^2}\Big)^3}\right)\p_1(v\cdot\n)\\
=&\int_{-\ell}^{\ell}\sigma\left(\dfrac{\kp_1\p_{1}\z_0+\p_1\ep}{\Big(\sqrt{1+\abs{\p_{1}\z_0}^2}\Big)^3}\right)\p_1\Big(\dt\ep-\ap\p_1\z_0-\dt^m\ss-\s_5\Big)\no\\
=&\int_{-\ell}^{\ell}\sigma\left(\dfrac{\kp_1\p_{1}\z_0+\p_1\ep}{\Big(\sqrt{1+\abs{\p_{1}\z_0}^2}\Big)^3}\right)\p_{1}
\Big(\dt\ep-\ap\p_1\z_0\Big)-\int_{-\ell}^{\ell}\sigma\left(\dfrac{\kp_1\p_{1}\z_0+\p_1\ep}{\Big(\sqrt{1+\abs{\p_{1}\z_0}^2}\Big)^3}\right)\p_{1}(\dt^m\ss+\s_5).\no
\end{align}
We may use \eqref{linearize 6} and \eqref{equilibrium} to simplify
\begin{align}
&\int_{-\ell}^{\ell}\sigma\left(\dfrac{\kp_1\p_{1}\z_0+\p_1\ep}{\Big(\sqrt{1+\abs{\p_{1}\z_0}^2}\Big)^3}\right)\p_{1}
\Big(\dt\ep-\ap\p_1\z_0\Big)
=\int_{-\ell}^{\ell}\sigma\left(\dfrac{\kp_1\p_{1}\z_0+\p_1\ep}{\Big(\sqrt{1+\abs{\p_{1}\z_0}^2}\Big)^3}\right)
\Big(\dt\p_{1}\ep-\p_{1}\ap\p_1\z_0-\ap\p_{11}\z_0\Big)\no\\
=&\int_{-\ell}^{\ell}\sigma\left(\dfrac{\kp_1\p_{1}\z_0+\p_1\ep}{\Big(\sqrt{1+\abs{\p_{1}\z_0}^2}\Big)^3}\right)\Big((\dt\p_{1}\ep+\dt \kp_1\p_1\z_0)-\ap\p_{11}\z_0\Big)=B_1+B_2,\no
\end{align}
where
\begin{align}
B_1=&\int_{-\ell}^{\ell}\sigma\left(\dfrac{\kp_1\p_{1}\z_0+\p_1\ep}{\Big(\sqrt{1+\abs{\p_{1}\z_0}^2}\Big)^3}\right)\Big(\dt\p_{1}\ep+\dt \kp_1\p_1\z_0\Big)\\
=&\int_{-\ell}^{\ell}\sigma\left(\dfrac{\kp_1\p_{1}\z_0+\p_1\ep}{\Big(\sqrt{1+\abs{\p_{1}\z_0}^2}\Big)^3}\right)\dt\Big(\p_{1}\ep+\kp_1\p_1\z_0\Big)
=\dt\left(\int_{-\ell}^{\ell}\frac{\sigma}{2}\dfrac{\Big(\kp_1\p_{1}\z_0+\p_1\ep\Big)^2}{\Big(\sqrt{1+\abs{\p_{1}\z_0}^2}\Big)^3}\right),\no
\end{align}
and
\begin{align}
B_2=&-\int_{-\ell}^{\ell}\sigma\left(\dfrac{\kp_1\p_{1}\z_0+\p_1\ep}{\Big(\sqrt{1+\abs{\p_{1}\z_0}^2}\Big)^3}\right)\ap\p_{11}\z_0
=-\int_{-\ell}^{\ell}\sigma\ap(\kp_1\p_{1}\z_0+\p_1\ep)\p_1\left(\dfrac{\p_1\z_0}{\sqrt{1+\abs{\p_{1}\z_0}^2}}\right)\\
=&\int_{-\ell}^{\ell}\ap(\kp_1\p_{1}\z_0+\p_1\ep)(P_0-g\z_0)=C_1+C_2+C_3+C_4.\no
\end{align}
We may use \eqref{linearize 6} to directly compute
\begin{equation}
C_1= \kp_1P_0\int_{-\ell}^{\ell}\ap\p_1\z_0=-\kp_1P_0\int_{-\ell}^{\ell}\p_1\ap\z_0=\kp_1P_0\int_{-\ell}^{\ell}\dt\kp_1\z_0=\dt\bigg(\frac{P_0M}{2}\kp_1^2\bigg),
\end{equation}
\begin{align}
C_2 &=-\kp_1g\int_{-\ell}^{\ell}\ap\z_0\p_1\z_0 =
-\frac{\kp_1g}{2}\int_{-\ell}^{\ell}\ap\p_1\abs{\z_0}^2
=\frac{\kp_1g}{2}\int_{-\ell}^{\ell}\p_1\ap\z_0^2 \\
&=-\frac{\kp_1g}{2}\int_{-\ell}^{\ell}\dt\kp_1\z_0^2
=-\dt\bigg(\frac{\kp_1^2g}{4}\int_{-\ell}^{\ell}\z_0^2\bigg), \no
\end{align}
and
\begin{equation}
C_3 = P_0\int_{-\ell}^{\ell}\ap\p_{1}\ep=-P_0\int_{-\ell}^{\ell}\p_{1}\ap\ep=P_0\int_{-\ell}^{\ell}\dt \kp_1\ep=P_0\kp_1\dt\kp_1\int_{-\ell}^{\ell}\z_0=\dt\bigg(\frac{P_0M}{2}\kp_1^2\bigg),
\end{equation}
which means
\begin{align}
C_1+C_3=\dt\bigg(P_0M\kp_1^2\bigg) \text{ and }
C_4=-\int_{-\ell}^{\ell}g\ap\z_0\p_{1}\ep,
\end{align}
and so
\begin{align}
A_1+C_4=\int_{-\ell}^{\ell}g\ap\z_0\p_1\ep-\int_{-\ell}^{\ell}g\ap\z_0\p_{1}\ep=0.
\end{align}

\emph{Step 3: Surface tension term - Second stage:}
The transport equation in \eqref{system 3} and integrating by parts imply
\begin{align}
H_2:=&\int_{\Sigma_0}\kp_1(P_0-g\z_0)(v\cdot\n)=\int_{-\ell}^{\ell}\kp_1(P_0-g\z_0)\Big(\dt\ep-\ap\p_1\z_0-\dt^m\ss-\s_5\Big)\\
=&D_1+D_2+D_3+D_4-\int_{-\ell}^{\ell}\kp_1(P_0-g\z_0)\Big(\dt^m\ss+\s_5\Big),\no
\end{align}
where by \eqref{linearize 6},
\begin{align}
D_1=&\kp_1P_0\int_{-\ell}^{\ell}\dt\ep=\kp_1P_0\bigg(\dt\kp_1\int_{-\ell}^{\ell}\z_0\bigg)=\dt\bigg(\frac{P_0M}{2}\kp_1^2\bigg),
\end{align}
\begin{align}
D_2=&-g\kp_1\int_{-\ell}^{\ell}\z_0\dt\ep=-g\kp_1\dt\bigg(\int_{-\ell}^{\ell}\z_0\ep\bigg),
\end{align}
which means
\begin{align}
A_2+D_2=-g\dt \kp_1\bigg(\int_{-\ell}^{\ell}\z_0\ep\bigg)-g\kp_1\dt\bigg(\int_{-\ell}^{\ell}\z_0\ep\bigg)=-\dt\bigg(g\kp_1\int_{-\ell}^{\ell}\z_0\ep\bigg),
\end{align}
\begin{align}
D_3=&-\kp_1P_0\int_{-\ell}^{\ell}\ap\p_1\z_0=\kp_1P_0\int_{-\ell}^{\ell}\p_1\ap\z_0=-P_0\kp_1\int_{-\ell}^{\ell}\dt\kp_1\z_0=-\dt\bigg(\frac{P_0M}{2}\kp_1^2\bigg),
\end{align}
from which we see that
\begin{align}
D_1+D_3=0,
\end{align}
and
\begin{align}
D_4=&g\kp_1\int_{-\ell}^{\ell}\z_0\ap\p_1\z_0=\frac{g\kp_1}{2}\int_{-\ell}^{\ell}\ap\p_1\abs{\z_0}^2=-\frac{g\kp_1}{2}\int_{-\ell}^{\ell}\p_1\ap\z_0^2\\
=&\frac{g\kp_1}{2}\int_{-\ell}^{\ell}\dt\kp_1\z_0^2=\dt\bigg(\frac{\kp_1^2g}{4}\int_{-\ell}^{\ell}\z_0^2\bigg),\no
\end{align}
which in turn implies that
\begin{align}
C_2+D_4=0.
\end{align}
In summary, we have
\begin{align}
G+H_1+H_2
=&\dt\left(\int_{-\ell}^{\ell}\frac{g}{2}\abs{\ep}^2+P_0M\kp_1^2-g\kp_1\int_{-\ell}^{\ell}\z_0\ep\right)-\int_{-\ell}^{\ell}g\ep(\dt^m\ss+\s_5)\\
&-\int_{-\ell}^{\ell}\sigma\left(\dfrac{\kp_1\p_{1}\z_0+\p_1\ep}{\Big(\sqrt{1+\abs{\p_{1}\z_0}^2}\Big)^3}\right)\p_{1}(\dt^m\ss+\s_5)
-\int_{-\ell}^{\ell}\kp_1(P_0-g\z_0)\Big(\dt^m\ss+\s_5\Big).\no
\end{align}

\emph{Step 4: Contact point term - First stage:}
Note that $\dt\ep(-\ell)=\dt\ep(\ell)=0$,
and
$\ap(-\ell)=\dt\lp+\dt^m\oo(-\ell)$, $\ap(\ell)=\dt\rp+\dt^m\oo(\ell)$.
Using these, we have
\begin{align}
&-\frac{\kappa}{\p_1\z_0}\dt\rp(v\cdot\n)\bigg|_{\ell}-\frac{\kappa}{\p_1\z_0}\dt\lp(v\cdot\n)\bigg|_{-\ell}\\
=&-\frac{\kappa}{\p_1\z_0}\dt\rp\Big(\dt\ep-\ap\p_1\z_0-\dt^m\ss-\s_5\Big)\bigg|_{\ell}-\frac{\kappa}{\p_1\z_0}\dt\lp\Big(\dt\ep-\ap\p_1\z_0-\dt^m\ss-\s_5\Big)\bigg|_{-\ell}\no\\
=&\kappa\Big((\dt\lp)^2+(\dt\rp)^2\Big)+\kappa\bigg(\dt\rp\dt^m\oo\bigg|_{\ell}+\dt\lp\dt^m\oo\bigg|_{-\ell}\bigg)\no\\
&+\bigg(\frac{\kappa}{\p_1\z_0}\dt\rp\Big(\dt^m\ss+\s_5\Big)\bigg|_{\ell}+\frac{\kappa}{\p_1\z_0}\dt\lp\Big(\dt^m\ss+\s_5\Big)\bigg|_{-\ell}\bigg).\no
\end{align}

\emph{Step 5: Contact point term - Second stage:}
Using the transport equation in \eqref{system 3}, we have
\begin{align}
&\frac{\sigma}{\p_1\z_0}\dt^m\qq(v\cdot\n)\bigg|_{-\ell}^{\ell}=\frac{\sigma}{\p_1\z_0}\dt^m\qq\Big(\dt\ep-\ap\p_1\z_0-\dt^m\ss-\s_5\Big)\bigg|_{-\ell}^{\ell}\\
=&\bigg(-\sigma\dt^m\qq\Big(\dt\lp+\dt^m\oo\Big)\bigg|_{\ell}+\sigma\dt^m\qq\Big(\dt\rp+\dt^m\oo\Big)\bigg|_{-\ell}\bigg)
-\frac{\sigma}{\p_1\z_0}\dt^m\qq\Big(\dt^m\ss+\s_5\Big)\bigg|_{-\ell}^{\ell},\no
\end{align}
and
\begin{align}
&-\frac{1}{\p_1\z_0}\s_7(v\cdot\n)\bigg|_{\ell}-\frac{1}{\p_1\z_0}\s_6(v\cdot\n)\bigg|_{-\ell}\\
=&-\frac{1}{\p_1\z_0}\s_7\Big(\dt\ep-\ap\p_1\z_0-\dt^m\ss-\s_5\Big)\bigg|_{\ell}-\frac{1}{\p_1\z_0}\s_6\Big(\dt\ep-\ap\p_1\z_0-\dt^m\ss-\s_5\Big)\bigg|_{-\ell}\no\\
=&\s_7\Big(\dt\rp+\dt^m\oo\Big)\bigg|_{\ell}+\s_6\Big(\dt\lp+\dt^m\oo\Big)\bigg|_{-\ell}
+\frac{1}{\p_1\z_0}\s_7\Big(\dt^m\ss+\s_5\Big)\bigg|_{\ell}+\frac{1}{\p_1\z_0}\s_6\Big(\dt^m\ss+\s_5\Big)\bigg|_{-\ell}.\no
\end{align}
Combining the terms related to $\dt^m\ss+\s_5$, and using the equilibrium equation \eqref{equilibrium}, we can get \eqref{linearize 5}.

\end{proof}

\section{Basic estimates}\label{sec:basic_estimates}

In this section we record a number of essential estimates needed for the nonlinear analysis of \eqref{system 4}.  These include auxiliary estimates for the pressure, free surface function, and contact points, as well as elliptic estimates for the Stokes problem.

\subsection{Pressure estimates}

The linearized energy-dissipation structure cannot control the pressure directly, so we need a separate argument. It is well-known that the pressure can be regarded as the Lagrangian multiplier in the fluid equation, and that this gives rise to a pressure estimate.  Here we will employ estimates proved in \cite{guo_tice_QS}, which we restate without proof here.

\begin{theorem}\label{pressure theorem}
If $(v,\ep)$ satisfies
\begin{align}
&\int_{\Omega_0}\frac{\mu J}{2}\dm_{\a}v:\dm_{\a}w+\int_{\Sigma_{0b}}\beta(v\cdot\tau_0)(w\cdot\tau_0)+\int_{\Sigma_0}g\ep(w\cdot\n)\\
&+\int_{\Sigma_0}\kp_1(p_0-g\z_0)(w\cdot\n)+ \int_{\Sigma_0} \sigma\left(\dfrac{\kp_1\p_{1}\z_0+\p_1\ep}{\Big(\sqrt{1+\abs{\p_{1}\z_0}^2}\Big)^3}\right)\p_1(w\cdot\n)\no\\
&+\bigg(-\frac{\kappa}{\p_1\z_0}\dt\rp(w\cdot\n)\bigg|_{\ell}-\frac{\kappa}{\p_1\z_0}\dt\lp(w\cdot\n)\bigg|_{-\ell}\bigg)\no\\
=&\int_{\Omega_0}Jw\cdot \s_1-\int_{\Sigma_0}w\cdot\s_3-\int_{\Sigma_{0b}}\s_4\bigg(w\cdot\frac{\t}{\abs{\t}^2}\bigg)- \int_{\Sigma_0} \sigma\p_1\dt^m\rr(w\cdot\n)\no\\
&+\frac{\sigma}{\p_1\z_0}\dt^m\qq(w\cdot\n)\bigg|_{-\ell}^{\ell}
-\frac{1}{\p_1\z_0}\s_7(w\cdot\n)\bigg|_{\ell}-\frac{1}{\p_1\z_0}\s_6(w\cdot\n)\bigg|_{-\ell}.\no
\end{align}
for all $w\in\v(t)$, then there exists a unique $q\in\hl$ such that
\begin{align}
&\int_{\Omega_0}\frac{\mu J}{2}\dm_{\a}v:\dm_{\a}w-\int_{\Omega_0}Jq(\na_{\a}\cdot w)+\int_{\Sigma_{0b}}\beta(v\cdot\tau_0)(w\cdot\tau_0)+\int_{\Sigma_0}g\ep(w\cdot\n)\\
&+\int_{\Sigma_0}\kp_1(p_0-g\z_0)(w\cdot\n)+ \int_{\Sigma_0} \sigma\left(\dfrac{\kp_1\p_{1}\z_0+\p_1\ep}{\Big(\sqrt{1+\abs{\p_{1}\z_0}^2}\Big)^3}\right)\p_1(w\cdot\n)\no\\
&+\bigg(-\frac{\kappa}{\p_1\z_0}\dt\rp(w\cdot\n)\bigg|_{\ell}-\frac{\kappa}{\p_1\z_0}\dt\lp(w\cdot\n)\bigg|_{-\ell}\bigg)\no\\
=&\int_{\Omega_0}Jw\cdot \s_1-\int_{\Sigma_0}w\cdot\s_3-\int_{\Sigma_{0b}}\s_4\bigg(w\cdot\frac{\t}{\abs{\t}^2}\bigg)- \int_{\Sigma_0} \sigma\p_1\dt^m\rr(w\cdot\n)\no\\
&+\frac{\sigma}{\p_1\z_0}\dt^m\qq(w\cdot\n)\bigg|_{-\ell}^{\ell}
-\frac{1}{\p_1\z_0}\s_7(w\cdot\n)\bigg|_{\ell}-\frac{1}{\p_1\z_0}\s_6(w\cdot\n)\bigg|_{-\ell}.\no
\end{align}
for all $w\in\w(t)$. Moreover,
\begin{align}
\hm{q}{0}^2\ls \hm{v}{1}^2+\nm{\s}_{(H^1)^{\ast}}^2,
\end{align}
where $\s\in (H^1)^{\ast}$ is given by
\begin{align}
\bro{\s,w}=&\int_{\Omega_0}Jw\cdot \s_1-\int_{\Sigma_0}w\cdot\s_3-\int_{\Sigma_{0b}}\s_4\bigg(w\cdot\frac{\t}{\abs{\t}^2}\bigg).
\end{align}
\end{theorem}
\begin{proof}
See the proof of \cite[Theorem 4.6]{guo_tice_QS}.
\end{proof}

\subsection{Free surface estimates}


In this section, we prove various estimates regarding the free surface function. We begin with a useful inequality.
\begin{lemma}\label{surface lemma 1}
Suppose that $\ep\in H^1_0(-\ell,\ell)$ and $\kp_1\in\r$ satisfy \eqref{linearize 1}. Then we have the norm equivalence
\begin{align}
\hms{\p_1\ep}{0} \ls\hms{\kp_1\p_1\z_0+\p_1\ep}{0} \ls \hms{\p_1\ep}{0}.
\end{align}
\end{lemma}
\begin{proof}
Assume that the first inequality is not true. Then we can find a sequence $\{(\kp_1^n, \ep^n)\}_{n=1}^{\infty}\subset\r\times H^1_0(-\ell,\ell)$ such that
\begin{align}\label{free surface 1}
\kp_1^n\int_{-\ell}^{\ell}\z_0(x_1)\ud{x_1}=\int_{-\ell}^{\ell}\ep^n(x_1)\ud{x_1},
\end{align}
with
\begin{align}
\hms{\p_1\ep^n}{0}=1 \text{ and } \hms{\kp_1^n\p_1\z_0+\p_1\ep^n}{0}\leq \frac{1}{n}.
\end{align}
Then by Poincar\'e's inequality and the relation \eqref{linearize 1}, we have that
\begin{align}
\abs{\kp_1^n}= \abs{\int_{-\ell}^{\ell}\ep^n(x_1)\ud{x_1}} \abs{\int_{-\ell}^{\ell}\z_0(x_1)\ud{x_1} }^{-1} \ls\hms{\ep^n}{0}\ls\hms{\p_1\ep^n}{0}\ls 1.
\end{align}
Hence, we know $\kp_1^n$ is bounded and up to the extraction a subsequence, we may assume $\kp_1^n\rt \kp_1$. On the other hand, by Poincar\'e's inequality, we may estimate
\begin{align}
\hms{\kp_1^n\z_0+\ep^n}{1}\ls\hms{\p_1\Big(\kp_1^n\z_0+\ep^n\Big)}{0}\ls \frac{1}{n}.
\end{align}
Hence, $\kp_1^n\z_0+\ep^n\to 0$ in $H^1$, but since $\kp_1^n\to \kp_1$, we conclude that $\ep^n\to\ep$ in $H^1$. Sending $n\to\infty$ in \eqref{free surface 1} and using \eqref{linearize 1}, we conclude that
\begin{align}
\kp_1\int_{-\ell}^{\ell}\z_0(x_1)\ud{x_1}=\int_{-\ell}^{\ell}\ep(x_1)\ud{x_1}=-\kp_1\int_{-\ell}^{\ell}\z_0(x_1)\ud{x_1}.
\end{align}
Hence, $\kp_1=0$ and $\ep=0$.
This contradicts the fact that $\hms{\p_1\ep}{0}=\lim_{n\rt\infty}\hms{\p_1\ep^n}{0}=1$, which proves the first inequality. On the other hand, the second inequality follows trivially from \eqref{linearize 1}.
\end{proof}

Next, we prove an elliptic estimate for the free surface function.
\begin{lemma}\label{surface lemma 2}
Suppose that $\ep\in H^1_0(-\ell,\ell)$ and $\kp_1\in\r$ satisfy \eqref{linearize 1}, and
\begin{align}\label{free surface 2}
&\int_{-\ell}^{\ell}g\ep\theta+\int_{-\ell}^{\ell}\kp_1(P_0-g\z_0)\theta+ \int_{-\ell}^{\ell} \sigma \left(\dfrac{\kp_1\p_{1}\z_0+\p_1\ep}{\Big(\sqrt{1+\abs{\p_{1}\z_0}^2}\Big)^3}\right)\p_1\theta=\bro{\s,\theta},
\end{align}
for any $\theta\in H^1_0(-\ell,\ell)$. If $\s\in H^{-s}(-\ell,\ell)$ for $s\in[0,1]$, then $\ep\in H^{2-s}_0(-\ell,\ell)$ and
\begin{align}
\hms{\ep}{2-s}+\abs{\kp_1}\ls \hms{\s}{-s}.
\end{align}
\end{lemma}
\begin{proof}
We begin with an $H^1$ estimate for $\ep$. We rearrange the first two terms of \eqref{free surface 2} to see the structure
\begin{align}
\int_{-\ell}^{\ell}g\Big(\ep-\kp_1\z_0\Big)\theta+P_0\kp_1\int_{-\ell}^{\ell}\theta+
\int_{-\ell}^{\ell}\sigma\left(\dfrac{\kp_1\p_{1}\z_0+\p_1\ep}{\Big(\sqrt{1+\abs{\p_{1}\z_0}^2}\Big)^3}\right)\p_1\theta
=\bro{\s,\theta}.
\end{align}
Then we take the test function $\theta=\kp_1\z_0+\ep$ and use Lemma \ref{surface lemma 1} to obtain
\begin{align}
\int_{-\ell}^{\ell}\sigma\left(\dfrac{\kp_1\p_{1}\z_0+\p_1\ep}{\Big(\sqrt{1+\abs{\p_{1}\z_0}^2}\Big)^3}\right)\p_1\theta
=&\int_{-\ell}^{\ell}\sigma\dfrac{\abs{\kp_1\p_{1}\z_0+\p_1\ep}^2}{\Big(\sqrt{1+\abs{\p_{1}\z_0}^2}\Big)^3}\gs\hms{\kp_1\p_{1}\z_0+\p_1\ep}{0}^2\gs\hms{\ep}{1}^2.
\end{align}
On the other hand,
\begin{align}
&\int_{-\ell}^{\ell}g\Big(\ep-\kp_1\z_0\Big)\theta+P_0\kp_1\int_{-\ell}^{\ell}\theta
=g\int_{-\ell}^{\ell}\Big(\ep-\kp_1\z_0\Big)\Big(\ep+\kp_1\z_0\Big)+P_0\kp_1\int_{-\ell}^{\ell}\Big(\ep+\kp_1\z_0\Big)\\
=&g\int_{-\ell}^{\ell}\Big(\ep^2-\kp_1^2\z_0^2\Big)+P_0\kp_1\int_{-\ell}^{\ell}\ep+P_0\kp_1^2\int_{-\ell}^{\ell}\z_0
=g\int_{-\ell}^{\ell}\abs{\ep}^2-g\kp_1^2\int_{-\ell}^{\ell}\z_0^2+2P_0\kp_1^2M\no\\
=&g\int_{-\ell}^{\ell}\abs{\ep}^2+\kp_1^2\bigg(2P_0M-g\int_{-\ell}^{\ell}\z_0^2\bigg).\no
\end{align}
Consider the second equation in \eqref{equilibrium}.
Multiplying by $\z_0$ and integrating, we obtain the identity
\begin{align}
P_0M-g\int_{-\ell}^{\ell}\z_0^2=&\sigma\int_{-\ell}^{\ell}\dfrac{\abs{\p_{1}\z_0}^2}{\sqrt{1+\abs{\p_{1}\z_0}^2}}.
\end{align}
Therefore, since $P_0M>0$, we may use Lemma \ref{surface lemma 1} to derive
\begin{align}
&g\int_{-\ell}^{\ell}\ep\theta+\int_{-\ell}^{\ell}\kp_1\Big(P_0-g\z_0\Big)\theta+
\sigma\int_{-\ell}^{\ell}\left(\dfrac{\kp_1\p_{1}\z_0+\p_1\ep}{\Big(\sqrt{1+\abs{\p_{1}\z_0}^2}\Big)^3}\right)\p_1\theta\\
=&\sigma\int_{-\ell}^{\ell}\dfrac{\abs{\kp_1\p_{1}\z_0+\p_1\ep}^2}{\Big(\sqrt{1+\abs{\p_{1}\z_0}^2}\Big)^3}+g\int_{-\ell}^{\ell}\abs{\ep}^2
+\kp_1^2\left(P_0M+\sigma\int_{-\ell}^{\ell}\dfrac{\abs{\p_{1}\z_0}^2}{\sqrt{1+\abs{\p_{1}\z_0}^2}}\right)\gs\hms{\ep}{1}^2+\kp_1^2.\no
\end{align}
Also, we use Lemma \ref{surface lemma 1} to directly estimate
\begin{align}
\abs{ \bro{\s,\theta} }\ls \hms{\s}{-1}\hms{\theta}{1}=\hms{\s}{-1}\hms{\kp_1\z_0+\ep}{1}\ls \hms{\s}{-1}\hms{\ep}{1}.
\end{align}
Hence,
\begin{align}\label{free surface 3}
\hms{\ep}{1}+\abs{\kp_1}\ls \hms{\s}{-1}.
\end{align}
We now obtain an $H^2$ estimate for our solution, supposing that $\s\in H^0$.
Let $b(x_1)=\Big(\sqrt{1+\abs{\p_{1}\z_0}^2}\Big)^3$ and take an arbitrary smooth function $\psi\in C_c^{\infty}(-\ell,\ell)$. Plugging in the test function $\theta=\psi b$ and rearranging \eqref{free surface 2}, we find that
\begin{align}
\\
\sigma\int_{-\ell}^{\ell}\left(\kp_1\p_{1}\z_0+\p_1\ep\right)\p_1\psi
=\bro{\s,\psi b}-\int_{-\ell}^{\ell}g\ep b\psi-\int_{-\ell}^{\ell}\kp_1b\Big(P_0-g\z_0\Big){\psi}-\sigma\int_{-\ell}^{\ell}\p_1b\left(\frac{\kp_1\p_{1}\z_0+\p_1\ep}{b}\right)\psi.\no
\end{align}
Hence, we know that $\kp_1\p_{1}\z_0+\p_1\ep$ is weakly differentiable and
\begin{align}
\kp_1\p_{11}\z_0+\p_{11}\ep=-\frac{1}{\sigma}\bigg(\s b-g\ep b-\kp_1b\Big(P_0-g\z_0\Big)\bigg)+\p_1b\left(\frac{\kp_1\p_{1}\z_0+\p_1\ep}{b}\right).
\end{align}
Then this and the estimate \eqref{free surface 3} imply that
\begin{align}\label{free surface 4}
\hms{\ep}{2} & \ls\hms{\ep}{1}+\abs{\kp_1}+\hms{\kp_1\p_{11}\z_0+\p_{11}\ep}{0}  \\
&\ls \hms{\s}{0}+\hms{\ep}{1}+\abs{\kp_1}\ls \hms{\s}{0}.\no
\end{align}
With the bounds \eqref{free surface 3} and \eqref{free surface 4} in hand, we may apply a standard interpolation argument to get the desired estimate.
\end{proof}

The next lemma is a variant of the previous elliptic regularity result.
\begin{lemma}\label{surface lemma 3}
Suppose that $\ep\in H^1_0(-\ell,\ell)$ and $\kp_1\in\r$ satisfy \eqref{linearize 1}, and
\begin{align}
&g\int_{\Sigma_0}\ep\theta+\kp_1\int_{\Sigma_0}(P_0-g\z_0)\theta+\sigma \int_{\Sigma_0} \left(\dfrac{\kp_1\p_{1}\z_0+\p_1\ep}{\Big(\sqrt{1+\abs{\p_{1}\z_0}^2}\Big)^3}\right)\p_1\theta=\bro{f,\p_1\theta},
\end{align}
for any $\theta\in H^1_0(-\ell,\ell)$. If $f\in H^{\frac{1}{2}}(-\ell,\ell)$, then $\ep\in H^{\frac{3}{2}}(-\ell,\ell)$ and
\begin{align}
\hms{\ep}{\frac{3}{2}} \ls \hms{f}{\frac{1}{2}}.
\end{align}
\end{lemma}
\begin{proof}
Arguing as in Lemma \ref{surface lemma 2}, we may deduce the bound $\hms{\ep}{1}+\abs{\kp_1}\ls \hms{f}{0}$.
Plugging in the test function $\theta=\psi\in C_c^{\infty}(-\ell,\ell)$ implies that
\begin{align}
\sigma\int_{-\ell}^{\ell}\left(\frac{\kp_1\p_{1}\z_0+\p_1\ep}{b}-f\right)\p_1\psi
=-\int_{-\ell}^{\ell}g\ep \psi-\int_{-\ell}^{\ell}\kp_1\Big(P_0-g\z_0\Big){\psi}.
\end{align}
Hence, we know that $\chi=\dfrac{\kp_1\p_{1}\z_0+\p_1\ep}{b}-f$ is weakly differentiable and
\begin{align}
\p_1\chi=\frac{1}{\sigma}\bigg(g\ep +\kp_1\Big(P_0-g\z_0\Big)\bigg).
\end{align}
Then this implies that
\begin{align}
\hms{\chi}{1}\ls&\hms{\chi}{0}+\hms{\p_1\chi}{0}
\ls \abs{\kp_1}+\hms{\ep}{1}+\hms{f}{0}\ls \hms{f}{0}.
\end{align}
Therefore, we have
\begin{align}
\hms{\p_{1}\ep}{\frac{1}{2}}\ls& \abs{\kp_1}+\hms{fb}{\frac{1}{2}}+\hms{\chi b}{\frac{1}{2}}
\ls\hms{f}{\frac{1}{2}}+\hms{\chi}{\frac{1}{2}}\ls \hms{f}{\frac{1}{2}}.
\end{align}
The desired estimate follows.
\end{proof}

We have now developed all of the tools needed to state a central estimate of the free surface function.
\begin{theorem}\label{surface theorem}
If $(v,q,\ep)$ satisfies
\begin{align}
&\int_{\Omega_0}\frac{\mu J}{2}\dm_{\a}v:\dm_{\a}w-\int_{\Omega_0}Jq(\na_{\a}\cdot w)+\int_{\Sigma_{0b}}\beta(v\cdot\tau_0)(w\cdot\tau_0)+\int_{\Sigma_0}g\ep(w\cdot\n)\\
&+\int_{\Sigma_0}\kp_1(p_0-g\z_0)(w\cdot\n)+ \int_{\Sigma_0} \sigma\left(\dfrac{\kp_1\p_{1}\z_0+\p_1\ep}{\Big(\sqrt{1+\abs{\p_{1}\z_0}^2}\Big)^3}\right)\p_1(w\cdot\n)\no\\
&+\bigg(-\frac{\kappa}{\p_1\z_0}\dt\rp(w\cdot\n)\bigg|_{\ell}-\frac{\kappa}{\p_1\z_0}\dt\lp(w\cdot\n)\bigg|_{-\ell}\bigg)\no\\
=&\int_{\Omega_0}Jw\cdot \s_1-\int_{\Sigma_0}w\cdot\s_3-\int_{\Sigma_{0b}}\s_4\bigg(w\cdot\frac{\t}{\abs{\t}^2}\bigg)- \int_{\Sigma_0} \sigma\p_1\dt^m\rr(w\cdot\n)\no\\
&+\frac{\sigma}{\p_1\z_0}\dt^m\qq(w\cdot\n)\bigg|_{-\ell}^{\ell}
-\frac{1}{\p_1\z_0}\s_7(w\cdot\n)\bigg|_{\ell}-\frac{1}{\p_1\z_0}\s_6(w\cdot\n)\bigg|_{-\ell}.\no
\end{align}
for all $w\in\w(t)$, then for each $\theta\in H^1_0(-\ell,\ell)$, there exists $w[\theta]\in\w(t)$ such that the following hold:
\begin{enumerate}
\item
$w[\theta]$ depends linearly on $\theta$.
\item
$w[\theta]\cdot\n=\theta$ on $\Sigma_0$.
\item
We have the estimate
\begin{align}
\hm{w[\theta]}{1}^2\ls \hms{\theta}{\frac{1}{2}}^2\ \ \text{and}\ \ \nm{w[\theta]}_{\w(t)}^2\ls \hms{\theta}{1}^2.
\end{align}
\item
We have the identity
\begin{align}
\\
&\int_{\Sigma_0}g\ep\theta+\kp_1\int_{\Sigma_0}(P_0-g\z_0)\theta+\sigma \int_{\Sigma_0} \left(\dfrac{\kp_1\p_{1}\z_0+\p_1\ep}{\Big(\sqrt{1+\abs{\p_{1}\z_0}^2}\Big)^3}\right)\p_1\theta=\bro{\g,w[\theta]}+\bro{\s,w[\theta]}+\int_{\Sigma_0}\sigma\dt^m\rr\p_1\theta,\no
\end{align}
where $\g$ is given by
\begin{align}
\bro{\g,w[\theta]}=&-\int_{\Omega_0}\frac{\mu J}{2}\dm_{\a}v:\dm_{\a}w[\theta]+\int_{\Omega_0}Jq(\na_{\a}\cdot w[\theta])-\int_{\Sigma_{0b}}\beta(v\cdot\tau_0)(w[\theta]\cdot\tau_0),
\end{align}
and $\s\in (H^1)^{\ast}$ is given by
\begin{align}
\bro{\s,w[\theta]}=&\int_{\Omega_0}Jw[\theta]\cdot \s_1-\int_{\Sigma_0}w[\theta]\cdot\s_3-\int_{\Sigma_{0b}}\s_4\bigg(w[\theta]\cdot\frac{\t}{\abs{\t}^2}\bigg).
\end{align}
\item
We have the estimate
\begin{align}\label{free surface 5}
\hms{\ep}{\frac{3}{2}}^2\ls \hm{v}{1}^2+\hm{q}{0}^2+\nm{\s}_{(H^1)^{\ast}}^2+\hms{\dt^m\rr}{\frac{1}{2}}^2.
\end{align}
\end{enumerate}
\end{theorem}
\begin{proof}
Let $\theta\in H^1_0(-\ell,\ell)$. Then we may use standard elliptic arguments (see Theorem 4.11 in \cite{guo_tice_QS}) to find $w[\theta]\in\w(t)$ satisfying
\begin{align}
\left\{
\begin{array}{ll}
\na_{\a}\cdot w= C\ds\int_{\ell}^{\ell}\theta(x_1)\ud x_1&\ \ \text{in}\ \ \Omega_0,\\
w\cdot\n=\theta&\ \ \text{on}\ \ \Sigma_0,\\
w\cdot\nu_0=0&\ \ \text{on}\ \ \Sigma_{0b},
\end{array}
\right.
\end{align}
where $C$ is chosen depending on $\Omega_0$ in order to enforce the compatibility condition for this equation. The resulting $w[\theta]$ satisfies the requirements of the statements. An integration by parts yields
\begin{align}
-\int_{\Sigma_0} \sigma\p_1\dt^m\rr(w\cdot\n)=&\int_{\Sigma_0} \sigma\dt^m\rr\p_1(w\cdot\n)-\sigma\dt^m\rr(w\cdot\n)\bigg|_{-\ell}^{\ell}.
\end{align}
Also, since $\theta\in H^1_0(-\ell,\ell)$, in the weak formulation, all the contributions at contact points vanish. Hence, we have
\begin{align}
&\int_{\Sigma_0}g\ep\theta+\int_{\Sigma_0}\kp_1(P_0-g\z_0)\theta+ \int_{\Sigma_0}\sigma \left(\dfrac{\kp_1\p_{1}\z_0+\p_1\ep}{\Big(\sqrt{1+\abs{\p_{1}\z_0}^2}\Big)^3}\right)\p_1\theta=\bro{\g,w[\theta]}+\bro{\s,w[\theta]}+\int_{\Sigma_0}\sigma\dt^m\rr\p_1\theta.
\end{align}
Then the estimate \eqref{free surface 5} follows in light of Lemmas \ref{surface lemma 2} and \ref{surface lemma 3}.
\end{proof}

\subsection{Contact point estimates}

In this section, we will prove the estimates of derivatives of $\e$ and $u\cdot\n$ at the contact points, which are much stronger than the results obtained through the usual trace theorem.

We start with the $\p_1\e$ estimates.
\begin{theorem}\label{contact point theorem 1}
There exists a universal $\vartheta>0$ such that if $\hms{\e}{0}+\abs{\dt L}+\abs{\dt R}<\vartheta$, then
\begin{align}
\sum_{j=0}^2\bigg(\abs{\dt^{j}\p_1\e(-\ell)}^2+\abs{\dt^{j}\p_1\e(\ell)}^2\bigg)\ls&  \sum_{j=0}^2\hms{\dt^j\e}{0}^2+\sum_{j=0}^2\bigg(\abs{\dt^{j+1}l}^2+\abs{\dt^{j+1}r}^2\bigg).
\end{align}
\end{theorem}
\begin{proof}
First note the simple fact that $\dt L=\dt l$ and $\dt R=\dt r$. Consider the equations for $\dt L$ and $\dt R$ in \eqref{system 2}.
Solving for $\p_1\e(-\ell)$ in the $\dt L$ equation, we obtain
\begin{align}\label{contact point 3}
\p_1\e(-\ell)=&J_1\sqrt{\frac{\sigma^2}{\Big(\ww(\dt L)+[\gamma]\Big)^2}-1}-\p_1\z_0(-\ell),
\end{align}
where here we assume $\vartheta$ is sufficiently small for the term in the square root to be nonnegative.
From the equilibrium equations \eqref{equilibrium}, we may compute
\begin{align}
\p_1\z_0(-\ell)=\sqrt{\frac{\sigma^2}{[\gamma]^2}-1}.
\end{align}
Further restricting the value of $\vartheta$ if necessary, employing a Taylor expansion, and using Lemma \ref{appendix lemma 1}, we conclude from these that
\begin{align}
\abs{\p_1\e(-\ell)}\ls\abs{J_1-1}+\abs{\dt L}\ls \hms{\e}{0}+\abs{\dt l}.
\end{align}
When $\dt$ is applied in \eqref{contact point 3}, we know
\begin{align}\label{contact point 1}
\dt\p_1\e(-\ell)\Big(\p_1\z_0(-\ell)+\p_1\e(-\ell)\Big)=&J_1\dt J_1\left(\frac{\sigma^2}{\Big(\ww(\dt L)+[\gamma]\Big)^2}-1\right)-J_1^2\frac{\sigma^2\ww'(\dt L)\dt^2L}{\Big(\ww(\dt L)+[\gamma]\Big)^3}.
\end{align}
Again restricting the value of $\vartheta$ if necessary, and considering $\p_1\z_0(-\ell)\gs 1$, we know
\begin{align}\label{contact point 2}
\abs{\p_1\z_0(-\ell)+\p_1\e(-\ell)}\gs 1.
\end{align}
Then we may use \eqref{contact point 1} and \eqref{contact point 2} together with \eqref{linearize 2} to estimate
\begin{align}
\abs{\dt\p_1\e(-\ell)}\ls \abs{\dt J_1}+\abs{\dt^2L}\ls\hms{\dt\e}{0}+\abs{\dt^2l}.
\end{align}
When $\dt^2$ is applied in \eqref{contact point 3}, we know
\begin{align}
&\dt^2\p_1\e(-\ell)\Big(\p_1\z_0(-\ell)+\p_1\e(-\ell)\Big)+\Big(\dt\p_1\e(-\ell)\Big)^2\\
=&\Big(J_1\dt^2 J_1+(\dt J_1)^2\Big)\left(\frac{\sigma^2}{\Big(\ww(\dt L)+[\gamma]\Big)^2}-1\right)-J_1\dt J_1\frac{\sigma^2\ww'(\dt L)\dt^2L}{\Big(\ww(\dt L)+[\gamma]\Big)^3}\no\\
&+J_1^2\frac{3\sigma^2\Big(\w'(\dt L)\Big)^2\Big(\dt^2L\Big)^2}{\Big(\ww(\dt L)+[\gamma]\Big)^4}-J_1^2\frac{\sigma^2\bigg(\ww''(\dt L)\Big(\dt^2L\Big)^2+\ww'(\dt L)\dt^3L\bigg)}{\Big(\ww(\dt L)+[\gamma]\Big)^3}.\no
\end{align}
Hence, as above, we may estimate
\begin{align}
\abs{\dt^2\p_1\e(-\ell)} &\ls \abs{\dt\p_1\e(-\ell)}^2+\abs{\dt^2J_1}+\abs{\dt J_1}^2+\abs{\dt J_1}\abs{\dt^2L}+\abs{\dt^2L}^2+\abs{\dt^3L} \\
& \ls\hms{\dt^2\e}{0}+\abs{\dt^2l}+\abs{\dt^3l}.\no
\end{align}
In summary, we have proved that if $\vartheta$ is sufficiently small, then
\begin{align}
\sum_{j=0}^2\abs{\dt^{j}\p_1\e(-\ell)}\ls \sum_{j=0}^2\hms{\dt^j\e}{0}+\sum_{j=0}^2\abs{\dt^{j+1}l}.
\end{align}
A similar argument with the $\dt R$ equation provides the bound
\begin{align}
\sum_{j=0}^2\abs{\dt^{j}\p_1\e(\ell)}\ls \sum_{j=0}^2\hms{\dt^j\e}{0}+\sum_{j=0}^2\abs{\dt^{j+1}r}.
\end{align}
\end{proof}

Next we consider the $u\cdot\n$ estimates.
\begin{theorem}\label{contact point theorem 2}
There exists a universal $\vartheta>0$ such that if $\hms{\e}{0}+\abs{\dt l}+\abs{\dt r}+\abs{\dt^2 l}+\abs{\dt^2 r}<\vartheta$, then
\begin{align}
\sum_{j=0}^2\bigg(\abs{\dt^{j}u(-\ell,0)\cdot\n}^2+\abs{\dt^{j}u(\ell,0)\cdot\n}^2\bigg)\ls&  \sum_{j=0}^2\hms{\dt^j\e}{0}^2+\sum_{j=0}^2\bigg(\abs{\dt^{j+1}l}^2+\abs{\dt^{j+1}r}^2\bigg).
\end{align}
\end{theorem}
\begin{proof}
Throughout the proof we will abuse notation by suppressing the time dependence of the unknowns; for example, we will write $\partial_t \eta(\pm \ell)$ in place of $\partial_t \eta(\pm \ell,t)$.  The transport equation reads
\begin{align}\label{contact point 4}
(u\cdot\n)b=J_1\dt\e-\tilde a(\p_1\z_0+\p_1\e),
\end{align}
for $b=\sqrt{1+\abs{\p_1\z_0}^2}$.
Hence, noting that $\dt\e(-\ell)=0$, using Theorem \ref{contact point theorem 1}, and taking $\vartheta$ to be sufficiently small, we know
\begin{align}
\abs{u(-\ell,0)\cdot\n}=\abs{\tilde a(\p_1\z_0+\p_1\e(-\ell))}\ls\abs{\tilde a}\Big(1+\abs{\p_1\e(-\ell)}\Big)\ls\abs{\tilde a} \ls \abs{\dt l}+\abs{\dt r}.
\end{align}
When $\dt$ is applied in \eqref{contact point 4}, we have
\begin{align}
(\dt u\cdot\n)b=\dt J_1\dt\e+J_1\dt^2\e-\dt\tilde a(\p_1\z_0+\p_1\e)-\tilde a\dt\p_1\e-(u\cdot\dt\n)b.
\end{align}
It is easy to check that $u(-\ell,0)=(\dt l,0)$ and $u(\ell,0)=(\dt r,0)$. Hence, noting that $\dt^2\e(-\ell)=0$, using Theorem \ref{contact point theorem 1}, and taking $\vartheta$ to be sufficiently small, we know
\begin{align}
\abs{\dt u(-\ell,0)\cdot\n}\ls&\abs{\dt\tilde a(\p_1\z_0+\p_1\e(-\ell))}+\abs{\tilde a\dt\p_1\e(-\ell)}+\abs{(u(-\ell,0)\cdot\dt\n)b}\\
\ls&\abs{\dt\tilde a}\Big(1+\abs{\p_1\e(-\ell)}\Big)+\abs{\tilde a}\abs{\dt\p_1\e(-\ell)}+\abs{u(-\ell,0)}\abs{\dt\p_1\e(-\ell)}\no\\
\ls&\abs{\dt\tilde a}+\abs{\dt\p_1\e(-\ell)}\ls \sum_{j=0}^1\hms{\dt^j\e}{0}+\sum_{j=0}^1\abs{\dt^{j+1}l}+\sum_{j=0}^1\abs{\dt^{j+1}r}\no.
\end{align}
When $\dt^2$ is applied in \eqref{contact point 4}, we have
\begin{align}
(\dt^2 u\cdot\n)b &=\dt^2 J_1\dt\e+\dt J_1\dt^2\e+J_1\dt^3\z-\dt^2\tilde a(\p_1\z_0+\p_1\e)-\dt\tilde a\dt\p_1\e  \\ &- \tilde a\dt^2\p_1\e-(u\cdot\dt^2\n)b-(\dt u\cdot\dt\n)b.\no
\end{align}
Hence, noting that $\dt^3\e(-\ell)=0$, using Theorem \ref{contact point theorem 1} and taking $\vartheta$ to be sufficiently small, we know
\begin{align}
& \abs{\dt u(-\ell,0)\cdot\n}\\
\ls& \abs{\dt^2\tilde a}\abs{(\p_1\z_0+\p_1\e(-\ell))}+\abs{\dt\tilde a}\abs{\dt\p_1\e(-\ell)} +\abs{\tilde a}\abs{\dt^2\p_1\e(-\ell)}+\abs{u(-\ell,0)}\abs{\dt^2\n}+\abs{\dt u(-\ell,0)}\abs{\dt\n}\no\\
\ls&\abs{\dt^2\tilde a}\Big(1+\abs{\p_1\e(-\ell)}\Big)+\abs{\dt\tilde a}\abs{\dt\p_1\e(-\ell)}+\abs{\tilde a}\abs{\dt^2\p_1\e(-\ell)}  \no\\
+&\abs{u(-\ell,0)}\abs{\dt^2\p_1\e}+\abs{\dt u(-\ell,0)}\abs{\dt\p_1\e(-\ell)}\no\\
\ls& \abs{\dt^2\tilde a}+\abs{\dt\p_1\e(-\ell)}+\abs{\dt^2\p_1\e(-\ell)}\ls \sum_{j=0}^2\hms{\dt^j\e}{0}+\sum_{j=0}^2\abs{\dt^{j+1}l}+\sum_{j=0}^2\abs{\dt^{j+1}r}\no.
\end{align}
In summary, we have now shown
\begin{align}
\sum_{j=0}^2\abs{\dt^{j}u(-\ell,0)\cdot\n}^2\ls&  \sum_{j=0}^2\hms{\dt^j\e}{0}^2+\sum_{j=0}^2\bigg(\abs{\dt^{j+1}l}^2+\abs{\dt^{j+1}r}^2\bigg).
\end{align}
A similar argument with the $\dt R$ equation provides the bound
\begin{align}
\sum_{j=0}^2\abs{\dt^{j}u(\ell,0)\cdot\n}^2\ls&  \sum_{j=0}^2\hms{\dt^j\e}{0}^2+\sum_{j=0}^2\bigg(\abs{\dt^{j+1}l}^2+\abs{\dt^{j+1}r}^2\bigg).
\end{align}
\end{proof}

\subsection{Weighted elliptic estimates for the Stokes problem}

We now turn our attention to some weighted elliptic estimates for solutions to the Stokes problem.

\begin{theorem}\label{elliptic theorem}
There exists a $\vartheta>0$ such that if $\e\in\wwd{\frac{5}{2}}{\Sigma_0}$ and $\nm{\e}_{\wwd{\frac{5}{2}}{\Sigma_0}}<\vartheta$, then there is a unique solution $(v,q,\ep)\in \wwd{2}{\Omega_0}\times\mathring{W}^1_{\d}(\Omega_0)\times\wwd{\frac{5}{2}}{\Sigma_0}$ to the equation
\begin{equation}
\left\{
\begin{array}{ll}
\na_{\a}\cdot S_{\a}(q,v)\n=G_1& \text{in } \Omega_0,\\
\na_{\a}\cdot v=G_2 & \text{in } \Omega_0,\\
v\cdot\n=G_3^+ & \text{on } \Sigma_0,\\
S_{\a}(q,v)\n=g\ep\n-\sigma \p_1\left(\dfrac{\kp_1\p_{1}\z_0}{\sqrt{1+\abs{\p_{1}\z_0}^2}}+\dfrac{\kp_1\p_{1}\z_0+\p_1\ep}{\Big(\sqrt{1+\abs{\p_{1}\z_0}^2}\Big)^3}+\dt^m\rr\right)\n  \\
\qquad \qquad \qquad +G_4^+\dfrac{\t}{\abs{\t}}
+G_5\dfrac{\n}{\abs{\n}}& \text{on } \Sigma_0,\\
v\cdot\n=G_3^-& \text{on } \Sigma_{0b},\\
\Big(S_{\a}(q,v)\n-\beta v\Big)\cdot\t=G_4^-& \text{on } \Sigma_{0b},\\
\ep\Big|_{-\ell}=\ep\Big|_{\ell}=0,
\end{array}
\right.
\end{equation}
such that for any $\d\in(0,1)$,
\begin{align}\label{elliptic 1}
\hmw{v}{2}^2+\hmw{q}{1}^2+&\nm{\ep}_{\wwd{\frac{5}{2}}{\Sigma_0}}^2\ls \hms{\ep}{0}^2+\hmw{G_1}{0}^2+\hmw{G_2}{1}^2\\
&+\hmws{G_3}{\frac{3}{2}}^2+\hmws{G_4}{\frac{1}{2}}^2+\hmws{G_5}{\frac{1}{2}}^2+\hmws{\p_1\dt^m\rr}{\frac{1}{2}}^2.\no
\end{align}
\end{theorem}
\begin{proof}
Using Theorem 5.10 in \cite{guo_tice_QS}, we may find $(v,q)\in \wwd{2}{\Omega_0}\times\mathring{W}^1_{\d}(\Omega_0)$ satisfying
\begin{align}\label{elliptic 10}
\hmw{v}{2}^2+\hmw{q}{1}^2\ls \hmw{G_1}{0}^2+\hmw{G_2}{1}^2+\hmws{G_3}{\frac{3}{2}}^2+\hmws{G_4}{\frac{1}{2}}^2.
\end{align}
The only remaining estimate is for $\ep$, which satisfies an elliptic equation
\begin{align}\label{elliptic 5}
g\ep-\sigma \p_1\left(\dfrac{\p_1\ep}{\Big(\sqrt{1+\abs{\p_{1}\z_0}^2}\Big)^3}\right)=&(qI-\mu\dm_{\a} v)\n\cdot\frac{\n}{\abs{\n}^2}\\
&+\sigma \p_1\left(\dfrac{\kp_1\p_{1}\z_0}{\sqrt{1+\abs{\p_{1}\z_0}^2}}+\dfrac{\kp_1\p_{1}\z_0}{\Big(\sqrt{1+\abs{\p_{1}\z_0}^2}\Big)^3}+\dt^m\rr\right)
-\dfrac{G_5}{\abs{\n}},\no
\end{align}
with Dirichlet boundary conditions $\ep(-\ell)=\ep(\ell)=0.$   We may use \eqref{elliptic 1} to verify that
\begin{align}\label{elliptic 7}
\hmwe{(qI-\mu\dm_{\a} v)\n\cdot\frac{\n}{\abs{\n}^2}}{\frac{1}{2}}^2\ls \hmw{v}{2}^2+\hmw{q}{1}^2.
\end{align}
Using \eqref{linearize 1}, we may directly estimate
\begin{align}\label{elliptic 11}
\hmwe{\sigma \p_1\left(\dfrac{\kp_1\p_{1}\z_0}{\sqrt{1+\abs{\p_{1}\z_0}^2}}+\dfrac{\kp_1\p_{1}\z_0}{\Big(\sqrt{1+\abs{\p_{1}\z_0}^2}\Big)^3}\right)}{\frac{1}{2}}^2\ls\abs{\kp_1}\ls\hms{\ep}{0}.
\end{align}
We then combine \eqref{elliptic 5}, \eqref{elliptic 7}, and \eqref{elliptic 11} to deduce that
\begin{align}\label{elliptic 9}
 \nm{\ep}_{\wwd{\frac{5}{2}}{\Sigma_0}}^2 & \ls \hmwe{(qI-\mu\dm_{\a} v)\n\cdot\frac{\n}{\abs{\n}^2}}{\frac{1}{2}}^2 \\
&+ \hmwe{\sigma \p_1\left(\dfrac{\kp_1\p_{1}\z_0}{\sqrt{1+\abs{\p_{1}\z_0}^2}}+\dfrac{\kp_1\p_{1}\z_0}{\Big(\sqrt{1+\abs{\p_{1}\z_0}^2}\Big)^3}\right)}{\frac{1}{2}}^2  +\hmws{\p_1\dt^m\rr}{\frac{1}{2}}^2+ \hmwe{\dfrac{G_5}{\abs{\n}}}{\frac{1}{2}}^2 \no\\
& \ls \hmw{v}{2}^2+\hmw{q}{1}^2+\hms{\ep}{0}^2+\hmws{\p_1\dt^m\rr}{\frac{1}{2}}^2+\hmws{G_5}{\frac{1}{2}}^2.\no
\end{align}
Combining \eqref{elliptic 10} and \eqref{elliptic 9}, we get the desired result.
\end{proof}

\section{Nonlinear estimates in the energy-dissipation structure}\label{nonlinear section}

We will employ the basic energy estimate of Theorem \ref{linear theorem 2} as the starting point for our a priori estimates. In order for this to be effective, we must be able to estimate the interaction terms appearing on the right side of \eqref{linearize 5} when the $\s_i$ terms are given as in Appendix \ref{ap section 1}. For the sake of brevity we will only present these estimates when the $\s_i$ terms are given for the twice temporally differentiated problem. The corresponding estimates for the once temporally differentiated problem follow from similar, though often simpler, arguments. When possible, we will present our estimates in the most general form, as estimates for general functionals generated by the $\s_i$ terms. It is only for a few essential terms that we must resort to employing the special structure of the interaction terms in order to close our estimates.

Throughout this section, we always assume that $\e$ is given and satisfies
\begin{align}\label{nonlinear 1}
\sup_{0\leq t\leq T}\bigg(\enp(t)+\hmwss{\e(t)}{\frac{5}{2}}+\hmwss{\dt\e(t)}{\frac{3}{2}}\bigg)\leq\vartheta<1,
\end{align}
for some $\vartheta>0$ sufficiently small.  Also, throughout this section, we will repeatedly use without explicit statement the following techniques in our estimates:
\begin{itemize}
\item Sobolev embedding theorems and trace theorems for both the usual Sobolev spaces and weighted Sobolev spaces;
\item the pointwise bound $\abs{\dt^mk_1}\ls\abs{\dt^{m}l}+\abs{\dt^{m}r}$, as well as the bounds, for any $1\leq r<\infty$, $\abs{\dt^mk_1}\ls \lms{\dt^m\e}{r}\ls \lms{\dt^m\p_1\e}{r}$, which follow due to \eqref{linearize 1} and Poincar\'e's inequality;
\item for $s>\dfrac{1}{2}$, the estimate $\hm{\be}{s}\ls\hms{\e}{s-\frac{1}{2}}$, which is due to the definition of harmonic extension;
\item the bound $\abs{\na\a}\ls\abs{\na^2\be}+\abs{\na\be}\abs{\dfrac{1}{\z_0}}$;
\item Theorems \ref{contact point theorem 1} and \ref{contact point theorem 2}.
\end{itemize}

Moreover, we will repeatedly use the following two lemmas:
\begin{lemma}\label{nonlinear lemma 1'}
Let $d=dist(\cdot, M)$, where $M=\Big\{(-\ell,0), (\ell,0)\Big\}$ is the set of the corner points. Suppose that $0<\d<1$. Then $d^{-\d}\in L^r(\Omega_0)$ for $1\leq r<\dfrac{2}{\d}$.
\end{lemma}
\begin{proof}
See Lemma 6.1 in \cite{guo_tice_QS}.
\end{proof}

\begin{lemma}\label{nonlinear lemma 2'}
Let $1<p<2$. Then $\dfrac{1}{\z_0}\in L^p(\Omega_0)$.
\end{lemma}
\begin{proof}
The difficulty concentrates on the neighborhood of the contact points. Near $(-\ell,0)$, using the mean value theorem, we have for $c\in(-\ell,x_1)$,
\begin{align}
\z_0(x_1)=\z_0(-\ell)+\p_1\z_0(c)(x_1+\ell)=\p_1\z_0(c)(x_1+\ell).
\end{align}
Then using polar coordinates, we compute
\begin{align}
\int_{x\in\Omega_0,d\leq R}\frac{1}{\z_0^p(x_1)}\ud{x_1}\ud{x_2} \leq \int_0^{\Theta}\int_{0}^R\frac{r}{r^p\cos^p\theta}\ud{r}\ud{\theta}
=\int_0^{\Theta}\int_{0}^R\frac{1}{r^{p-1}\cos^p\theta}\ud{r}\ud{\theta}<\infty,
\end{align}
where $\tan\Theta=\p_1\z_0(-\ell)$. The integral with respect to $r$ is finite since $0<p-1<1$ and the integral with respect to $\theta$ is finite since $0<\theta<\Theta<\dfrac{\pi}{2}$.
\end{proof}

\subsection{Estimate of the $\s_1$ term}

The following estimate is the same as that of Proposition 6.2 of \cite{guo_tice_QS}, but the argument needed to arrive at this estimate is slightly different due to the different structure of $\mathcal{A}$.  In particular, the appearance of the term $1/\zeta_0$ in terms involving $\nabla \mathcal{A}$ is novel to the droplet problem.
\begin{lemma}\label{nonlinear lemma 1}
Let $\s_1$ be given by \eqref{apf 1} or \eqref{apf 1'}. We have the estimate
\begin{align}
\abs{\int_{\Omega_0}Jv\cdot\s_1}\ls \hm{v}{1}\Big(\en+\sqrt{\en}\Big)\sqrt{\di}.
\end{align}
\end{lemma}
\begin{proof}
We will only present the estimate of the term $-\na_{\dt^2\a}\cdot\Big(pI-\mu\dm_{\a}u\Big)$ in order to highlight how to deal with the appearance of $1/\zeta_0$.   The remaining terms can be handled similarly, following the general blueprint of Proposition 6.2 of \cite{guo_tice_QS} with appropriate modifications to handle $1/\zeta_0$ as indicated in what follows.

To estimate  $-\na_{\dt^2\a}\cdot\Big(pI-\mu\dm_{\a}u\Big)$, we begin by splitting
\begin{align}
& \abs{\int_{\Omega_0}Jv\bigg(-\na_{\dt^2\a}\cdot\Big(pI-\mu\dm_{\a}u\Big)\bigg)}\\
\ls& \int_{\Omega_0}\abs{v}\abs{\dt^2\a}\Big(\abs{\na^2 u}+\abs{\na p}\Big) +\int_{\Omega_0}\abs{v}\abs{\dt^2\a}\abs{\na\a}\Big(\abs{\na u}+\abs{ p}\Big)\no\\
\ls& \int_{\Omega_0}\abs{v}\abs{\dt^2\na\be}\Big(\abs{\na^2 u}+\abs{\na p}\Big) +\int_{\Omega_0}\abs{v}\abs{\dt^2\na\be}\abs{\na^2\be}\Big(\abs{\na u}+\abs{ p}\Big) \no \\
& +\int_{\Omega_0}\abs{v}\abs{\dt^2\na\be}\abs{\na\be}\abs{\frac{1}{\z_0}}\Big(\abs{\na u}+\abs{ p}\Big)
 =: I+II+III.\no
\end{align}
For $I$, we choose $q\in[1,\infty)$ and $2<r<\dfrac{2}{\d}$ such that $\dfrac{2}{q}+\dfrac{1}{r}=\dfrac{1}{2}$. Then we have
\begin{align}
I=&\int_{\Omega_0}\abs{v}\abs{\dt^2\na\be}\Big(\abs{\na^2 u}+\abs{\na p}\Big)\\
\ls& \lm{v}{q}\lm{\dt^2\na\be}{q}\lm{d^{-\d}}{r}\Big(\lm{d^{\d}\na^2 u}{2}+\lm{d^{\d}\na p}{2}\Big)\no\\
\ls& \hm{v}{1}\hm{\dt^2\na\be}{1}\Big(\hmw{ u}{2}+\hmw{ p}{1}\Big)\no\\
\ls& \hm{v}{1}\hms{\dt^2\e}{\frac{3}{2}}\Big(\hmw{ u}{2}+\hmw{ p}{1}\Big)
\ls \hm{v}{1}\sqrt{\di}\sqrt{\en}.\no
\end{align}
For $II$, we choose $m=\dfrac{2}{2-s}$ and $2<r<\dfrac{2}{\d}$ such that $\dfrac{1}{m}+\dfrac{1}{r}<1$, which is possible since $\d<1<s$. Then choosing $q\in[1,\infty)$ such that $\dfrac{3}{q}+\dfrac{1}{m}+\dfrac{1}{r}=1$, we have
\begin{align}
II=&\int_{\Omega_0}\abs{v}\abs{\dt^2\na\be}\abs{\na^2\be}\Big(\abs{\na u}+\abs{ p}\Big)\\
\ls&\lm{v}{q}\lm{\dt^2\na\be}{q}\lm{\na^2\be}{m}\lm{d^{-\d}}{r}\Big(\lm{d^{\d}\na u}{q}+\lm{d^{\d} p}{q}\Big)\no\\
\ls& \hm{v}{1}\hm{\dt^2\na\be}{1}\hm{\na^2\be}{s-1}\Big(\hmw{ u}{2}+\hmw{ p}{1}\Big)\no\\
\ls& \hm{v}{1}\hms{\dt^2\e}{\frac{3}{2}}\hms{\e}{s+\frac{1}{2}}\Big(\hmw{ u}{2}+\hmw{ p}{1}\Big)\no\\
\ls& \hm{v}{1}\sqrt{\di}\sqrt{\en}\sqrt{\en}=\hm{v}{1}\en\sqrt{\di}.\no
\end{align}
For $III$, we choose $2<r<\dfrac{2}{\d}$ and $1<p<2$ such that $\dfrac{1}{p}+\dfrac{1}{r}<1$, which is possible since $r>2$. Then choosing $q\in[1,\infty)$ such that $\dfrac{4}{q}+\dfrac{1}{p}+\dfrac{1}{r}=1$, we have
\begin{align}
III=&\int_{\Omega_0}\abs{v}\abs{\dt^2\na\be}\abs{\na\be}\abs{\frac{1}{\z_0}}\Big(\abs{\na u}+\abs{ p}\Big)\\
\ls& \lm{v}{q}\lm{\dt^2\na\be}{q}\lm{\na\be}{q}\lm{\frac{1}{\z_0}}{p}\lm{d^{-\d}}{r}\Big(\lm{d^{\d}\na u}{q}+\lm{d^{\d} p}{q}\Big)\no\\
\ls& \hm{v}{1}\hm{\dt^2\na\be}{1}\hm{\na\be}{1}\Big(\hmw{ u}{2}+\hmw{ p}{1}\Big)\no\\
\ls& \hm{v}{1}\hms{\dt^2\e}{\frac{3}{2}}\hms{\e}{\frac{3}{2}}\Big(\hmw{ u}{2}+\hmw{ p}{1}\Big)\no\\
\ls& \hm{v}{1}\sqrt{\di}\sqrt{\en}\sqrt{\en}=\hm{v}{1}\en\sqrt{\di}.\no
\end{align}

\end{proof}

\subsection{Estimate of the $\s_2$ term}

The estimate of the $\s_2$ term is available from \cite{guo_tice_QS}.  We record it now.

\begin{lemma}\label{nonlinear lemma 2}
Let $\s_2$ be given by \eqref{apf 2} or \eqref{apf 2'}. We have the estimate
\begin{align}
\abs{\int_{\Omega_0}J\psi\s_2}\ls \hm{\psi}{0}\sqrt{\en}\sqrt{\di}.
\end{align}
\end{lemma}
\begin{proof}
The estimate is proved in Theorem 6.8 of \cite{guo_tice_QS}.
\end{proof}

\subsection{Estimate of the $\s_3$ term}

We now estimate the $\s_3$ term.

\begin{lemma}\label{nonlinear lemma 3}
Let $\s_3$ be given by \eqref{apf 3} or \eqref{apf 3'} and $\rr$ be given by \eqref{ap rr}. We have the estimate
\begin{align}
\abs{\int_{\Sigma_0}v\cdot\s_3}\ls \hm{v}{1}\Big(\en+\sqrt{\en}\Big)\sqrt{\di}.
\end{align}
\end{lemma}
\begin{proof}
In the following, we choose $p=\dfrac{3+\d}{2+2\d}$, $q=\dfrac{6+2\d}{1-\d}$ and $r=\dfrac{9+3\d}{1-\d}$ such that $\dfrac{1}{p}+\dfrac{2}{q}=1$ and $\dfrac{1}{p}+\dfrac{3}{r}=1$.

Estimates of the integral involving the terms $\mu\dm_{\dt^2\a}u\n$, $\mu\dm_{\dt\a}u\dt\n$
$2\mu\dm_{\dt\a}\dt u\n$, $-( pI-\mu\dm_{\a} u)\dt^2\n$, $-2(\dt pI-\mu\dm_{\a} \dt u)\dt\n$, $g\e\dt^2\n$, and $2g\dt\e\dt\n$ may be found in the proof of Proposition 6.4 of \cite{guo_tice_QS}.  The remaining three terms are novel, and we will present the estimates here.

\ \\
{\bf{Term: $-\sigma \p_{1}\left(\dfrac{k_1\p_{1}\z_0}{\sqrt{1+\abs{\p_{1}\z_0}^2}}+\dfrac{k_1\p_{1}\z_0+\p_1\e}{\Big(\sqrt{1+\abs{\p_{1}\z_0}^2}\Big)^3}\right)\dt^2\n$}\\}
We estimate
\begin{align}
&\abs{\int_{\Sigma_0}v\left(-\sigma \p_{1}\left(\dfrac{k_1\p_{1}\z_0}{\sqrt{1+\abs{\p_{1}\z_0}^2}}+\dfrac{k_1\p_{1}\z_0+\p_1\e}{\Big(\sqrt{1+\abs{\p_{1}\z_0}^2}\Big)^3}\right)\dt^2\n\right)}
\ls  \int_{\Sigma_{0}}\abs{v}\Big(\abs{k_1}+\abs{\p_1^2\e}\Big)\abs{\dt^2\n} \\
\ls& \int_{\Sigma_{0}}\abs{v}\Big(\abs{k_1}+\abs{\p_1^2\e}\Big)\abs{\dt^2\p_1\e}
\ls \lme{v}{q}\Big(\lms{\e}{p}+\lms{\p_1^2\e}{p}\Big)\lms{\dt^2\p_1\e}{q}\no\\
\ls& \hme{v}{\frac{1}{2}}\hmwss{\e}{\frac{5}{2}}\hms{\dt^2\e}{\frac{3}{2}}
\ls \hm{v}{1}\hmwss{\e}{\frac{5}{2}}\hms{\dt^2\e}{\frac{3}{2}} \no\\
\ls& \hm{v}{1}\sqrt{\en}\sqrt{\di}.\no
\end{align}
\ \\
{\bf{Term: $-\sigma (\p_{1}\rr)\dt^2\n$}\\}
We may directly bound
\begin{equation}
\abs{\p_1\rr}\ls \abs{\p_1\e}\abs{\p_1^2\e}+\abs{k_1}\abs{\p_1^2\e}.
\end{equation}
Then we estimate
\begin{align}
\abs{\int_{\Sigma_0}v\bigg(-\sigma (\p_{1}\rr)\dt^2\n\bigg)}
\ls&\int_{\Sigma_{0}}\abs{v}\abs{\p_1\rr}\abs{\dt^2\n}\\
\ls&\int_{\Sigma_{0}}\abs{v}\abs{\p_1\e}\abs{\p_1^2\e}\abs{\dt^2\p_1\e}+\int_{\Sigma_{0}}\abs{v}\abs{k_1}\abs{\p_1^2\e}\abs{\dt^2\p_1\e} \no \\
\ls&\lme{v}{r}\lms{\p_1\e}{r}\lms{\p_1^2\e}{p}\lms{\dt^2\p_1\e}{r}\no\\
\ls&\hme{v}{\frac{1}{2}}\hms{\p_1\e}{\frac{1}{2}}\hmwss{\e}{\frac{5}{2}}\hms{\dt^2\e}{\frac{3}{2}}\no\\
\ls&\hm{v}{1}\hms{\e}{\frac{3}{2}}\hmwss{\e}{\frac{5}{2}}\hms{\dt^2\e}{\frac{3}{2}}\no\\
\ls& \hm{v}{1}\sqrt{\en}\sqrt{\en}\sqrt{\di}=\hm{v}{1}\en\sqrt{\di}.\no
\end{align}
{\bf{Term: $-2\sigma \p_{1}\left(\dfrac{\dt k_1\p_{1}\z_0}{\sqrt{1+\abs{\p_{1}\z_0}^2}}+\dfrac{\dt k_1\p_{1}\z_0+\dt\p_1\e}{\Big(\sqrt{1+\abs{\p_{1}\z_0}^2}\Big)^3}\right)\dt\n$}\\}
We estimate
\begin{align}
&\abs{\int_{\Sigma_0}v\left(-2\sigma \p_{1}\left(\dfrac{\dt k_1\p_{1}\z_0}{\sqrt{1+\abs{\p_{1}\z_0}^2}}+\dfrac{\dt k_1\p_{1}\z_0+\dt\p_1\e}{\Big(\sqrt{1+\abs{\p_{1}\z_0}^2}\Big)^3}\right)\dt\n\right)}
\\
\ls& \int_{\Sigma_{0}}\abs{v}\Big(\abs{\dt k_1}+\abs{\dt\p_1^2\e}\Big)\abs{\dt\n}
\ls \int_{\Sigma_{0}}\abs{v}\Big(\abs{\dt k_1}+\abs{\dt\p_1^2\e}\Big)\abs{\dt\p_1\e} \no \\
\ls& \lme{v}{q}\Big(\lms{\dt\e}{p}+\lms{\dt\p_1^2\e}{p}\Big)\lms{\dt\p_1\e}{q} \no \\
\ls& \hme{v}{\frac{1}{2}}\hmwss{\dt\e}{\frac{5}{2}}\hms{\dt\e}{\frac{3}{2}}
\ls \hm{v}{1}\hmwss{\dt\e}{\frac{5}{2}}\hms{\dt\e}{\frac{3}{2}} \no \\
\ls& \hm{v}{1}\sqrt{\di}\sqrt{\en} \no.
\end{align}
\ \\
{\bf{Term: $-2\sigma (\dt\p_{1}\rr)\dt\n$}\\}
We may directly obtain
\begin{equation}
\abs{\dt\p_1\rr} \ls 6\abs{\dt k_1}\abs{\p_1^2\e}+\abs{k_1}\abs{\dt \p_1^2\e}+\abs{\p_1^2\e}\abs{\dt\p_1\e}+\abs{\p_1\e}\abs{\dt\p_1^2\e}.
\end{equation}
Then we estimate
\begin{align}
&\abs{\int_{\Sigma_0}v\bigg(-2\sigma (\dt\p_{1}\rr)\dt\n\bigg)}
\ls\int_{\Sigma_{0}}\abs{v}\abs{\dt\p_1\rr}\abs{\dt\n} \\
\ls&\int_{\Sigma_{0}}\abs{v}\bigg(\abs{\p_1^2\e}\abs{\dt\p_1\e}+\abs{\p_1\e}\abs{\dt\p_1^2\e}\bigg)
\abs{\dt\p_1\e}\no\\
\ls& \lme{v}{r}\bigg(\lms{\p_1^2\e}{p}\lms{\dt\p_1\e}{r}+\lms{\p_1\e}{r}\lms{\dt\p_1^2\e}{p}\bigg)
\lms{\dt\p_1\e}{r}\no\\
\ls& \hme{v}{\frac{1}{2}}\hmwss{\e}{\frac{5}{2}}\hmwss{\dt\e}{\frac{5}{2}}\hms{\dt\e}{\frac{3}{2}}\no\\
\ls& \hm{v}{1}\sqrt{\en}\sqrt{\di}\sqrt{\en}=\hm{v}{1}\en\sqrt{\di}.\no
\end{align}
\end{proof}

\subsection{Estimate of the $\s_4$ term}

The estimate for $\s_4$ is again available from \cite{guo_tice_QS}.

\begin{lemma}\label{nonlinear lemma 4}
Let $\s_4$ be given by \eqref{apf 4} or \eqref{apf 4'}. We have the estimate
\begin{align}
\abs{\int_{\Sigma_{0b}}\s_4\bigg(v\cdot\frac{\t}{\abs{\t}^2}\bigg)}\ls \hm{v}{1}\sqrt{\en}\sqrt{\di}.
\end{align}
\end{lemma}
\begin{proof}
The estimate is proved in Proposition 6.4 of \cite{guo_tice_QS}, though there the nonlinearity is named $F_5$ instead of $\s_4$.
\end{proof}

\subsection{Estimate of the $\s_6$ and $\s_7$ terms}

We now turn to the terms $\s_6$ and $\s_7$.

\begin{lemma}\label{nonlinear lemma 5}
Let $\s_6$ be given by \eqref{apf 6} or \eqref{apf 6'}, and $\s_7$ be given by \eqref{apf 7} or \eqref{apf 7'}. We have
\begin{equation}
\abs{-\s_7\Big(\dt\rp+\dt^2\oo\Big)\bigg|_{\ell}-\s_6\Big(\dt\lp+\dt^2\oo\Big)\bigg|_{-\ell}} \ls  \sqrt{\en}\di.
\end{equation}
\end{lemma}
\begin{proof}
We know that
\begin{equation}
\abs{ \dt\lp } = \abs{ \dt^3l} \ls\sqrt{\di} \text{ and }
\abs{\dt\rp} =\abs{ \dt^3r} \ls\sqrt{\di},
\end{equation}
and that
\begin{equation}
\abs{\dt^2\oo} \ls \Big(\abs{\dt l}+\abs{\dt r}\Big)\Big(\abs{\dt^2 l}+\abs{\dt^2 r}\Big)+\abs{k_1}\Big(\abs{\dt^3 l}+\abs{\dt^3 r}\Big)\ls\sqrt{\en}\sqrt{\di}.
\end{equation}
Hence, we have the estimates
\begin{equation}
\abs{\dt\lp+\dt^2\oo\Big|_{-\ell} }  \ls \sqrt{\di} \text{ and }
\abs{\dt\rp+\dt^2\oo\Big|_{\ell} } \ls \sqrt{\di}.
\end{equation}
We may easily check that $\abs{\tilde\ww'(z)}+\abs{\tilde\ww''(z)}\ls z$ for $z$ small, and so we know
\begin{equation}
\abs{\s_6} \ls \abs{\tilde\ww'(\dt l)}\abs{\dt^3l}+\abs{\tilde\ww''(\dt l)}\abs{\dt^2l}^2 \ls \abs{\dt l}\abs{\dt^3l}+\abs{\dt l}\abs{\dt^2l}^2
\ls \sqrt{\en}\sqrt{\di}.\no
\end{equation}
Next we bound
\begin{equation}
\abs{\s_7} \ls \abs{\tilde\ww'(\dt r)}\abs{\dt^3r}+\abs{\tilde\ww''(\dt r)}\abs{\dt^2r}^2
 \ls \abs{\dt r}\abs{\dt^3r}+\abs{\dt r}\abs{\dt^2r}^2
\ls \sqrt{\en}\sqrt{\di},
\end{equation}
from which we deduce that
\begin{equation}
\abs{-\s_7\Big(\dt\rp+\dt^2\oo\Big)\bigg|_{\ell}-\s_6\Big(\dt\lp+\dt^2\oo\Big)\bigg|_{-\ell}} \ls  \sqrt{\en}\di.
\end{equation}

\end{proof}

\subsection{Estimate of the $\s_5$ term}

Next we handle $\s_5$.

\begin{lemma}\label{nonlinear lemma 6}
Let $\s_5$ be given by \eqref{apf 5} or \eqref{apf 5'} and $\ss$ be given by \eqref{ap ss}. We have
\begin{equation}
\abs{\int_{-\ell}^{\ell}\left(g\ep-\p_1\left(\dfrac{\kp_1\p_{1}\z_0}{\sqrt{1+\abs{\p_{1}\z_0}^2}}
+\dfrac{\kp_1\p_{1}\z_0+\p_1\ep}{\Big(\sqrt{1+\abs{\p_{1}\z_0}^2}\Big)^3}\right)\right)(\dt^2\ss+\s_5)-\frac{\ud \sf_1}{\ud t}}
 \ls  \Big(\en+\sqrt{\en}\Big)\di,
\end{equation}
where
\begin{equation}
\sf_1= -\int_{-\ell}^{\ell}\frac{J_1-1}{\Big(\sqrt{1+\abs{\p_{1}\z_0}^2}\Big)^3}\abs{\dt^2\p_1\e}^2.
\end{equation}
\end{lemma}
\begin{proof}
It is easy to check that
\begin{align}
&g\ep-\p_1\left(\dfrac{\kp_1\p_{1}\z_0}{\sqrt{1+\abs{\p_{1}\z_0}^2}}
+\dfrac{\kp_1\p_{1}\z_0+\p_1\ep}{\Big(\sqrt{1+\abs{\p_{1}\z_0}^2}\Big)^3}\right)\\
=&\left(g\ep-\p_1\left(\dfrac{\kp_1\p_{1}\z_0}{\sqrt{1+\abs{\p_{1}\z_0}^2}}+\dfrac{\kp_1\p_{1}\z_0}{\Big(\sqrt{1+\abs{\p_{1}\z_0}^2}\Big)^3}\right)\right)
+\p_1\left(\dfrac{\p_1\ep}{\Big(\sqrt{1+\abs{\p_{1}\z_0}^2}\Big)^3}\right)=:I+II.\no
\end{align}
We may directly verify that
\begin{align}
\lms{I}{\infty}\ls& \lms{\ep}{\infty}+\lms{\kp_1}{\infty}\ls\lms{\dt^2\e}{\infty}+\lms{\dt^2k_1}{\infty}\\
\ls& \lms{\dt^2\e}{\infty}\ls \hms{\dt^2\e}{1}\ls\sqrt{\en}.\no
\end{align}
Then we have
\begin{equation}
\abs{\int_{-\ell}^{\ell}I(\dt^2\ss+\s_5)} \ls \lms{I}{\infty}\bigg(\int_{-\ell}^{\ell}\abs{\dt^2\ss}+\int_{-\ell}^{\ell}\abs{\s_5}\bigg).
\end{equation}
We may directly estimate
\begin{align}
\int_{-\ell}^{\ell}\abs{\dt^2\ss}\ls& \bigg(\abs{k_1}+\sum_{j=0}^2\Big(\abs{\dt^{j+1}l}+\abs{\dt^{j+1}r}\Big)\bigg)\\
& \times\bigg(\hms{\dt^3\e}{0}+\hms{\dt^2\e}{1}+\hms{\e}{1}+\sum_{j=0}^2\Big(\abs{\dt^{j+1}l}+\abs{\dt^{j+1}r}\Big)\bigg)\ls \sqrt{\di}\sqrt{\di}=\di,\no
\end{align}
and
\begin{align}
\int_{-\ell}^{\ell}\abs{\s_5}\ls& \bigg(\hms{u}{0}+\hms{\dt u}{0}\bigg)\bigg(\hms{\dt\e}{1}+\hms{\dt^2\e}{1}\bigg)\\
\ls& \bigg(\hm{u}{1}+\hm{\dt u}{1}\bigg)\bigg(\hms{\dt\e}{1}+\hms{\dt^2\e}{1}\bigg)
 \ls \sqrt{\di}\sqrt{\di}=\di.\no
\end{align}
Therefore, we know that
\begin{align}
\abs{\int_{-\ell}^{\ell}I(\dt^2\ss+\s_5)} \ls \sqrt{\en}\di.
\end{align}
Then we turn to the estimate of $II$. We integrate by parts to obtain
\begin{align}
&\int_{-\ell}^{\ell}II(\dt^2\ss+\s_5)=\int_{-\ell}^{\ell}\p_1\left(\dfrac{\p_1\ep}{\Big(\sqrt{1+\abs{\p_{1}\z_0}^2}\Big)^3}\right)(\dt^2\ss+\s_5)\\
=&-\int_{-\ell}^{\ell}\left(\dfrac{\p_1\ep}{\Big(\sqrt{1+\abs{\p_{1}\z_0}^2}\Big)^3}\right)(\dt^2\p_1\ss+\p_1\s_5)
=\int_{-\ell}^{\ell}b\dt^2\p_1\e\Big(\dt^2\p_1\ss+\p_1\s_5\Big),\no
\end{align}
for $b(x_1)=- \Big(1+\abs{\p_{1}\z_0}^2 \Big)^{-3/2}$.\\
\ \\
{\bf{First Term of $\dt^2\p_1\ss$: $\dt^2\p_1\Big((J_1-1)\dt\e\Big)$}\\}
We can directly compute
\begin{align}
\dt^2\p_1\Big((J_1-1)\dt\e\Big) = \dt^2J_1\dt\p_1\e+2\dt J_1\dt^2\p_1\e+(J_1-1)\dt^3\p_1\e.
\end{align}
For the first two terms we estimate
\begin{align}
\abs{\int_{-\ell}^{\ell}b\Big(\dt^2\p_1\e\Big)\Big(\dt^2J_1\dt\p_1\e\Big)}
\ls& \lms{\dt^2\p_1\e}{2}\lms{\dt^2J_1}{\infty}\lms{\dt\p_1\e}{2}\\
\ls& \hms{\dt^2\e}{1}\Big(\abs{\dt^{2}l}+\abs{\dt^{2}r}\Big)\hms{\dt\e}{1}\ls\sqrt{\en}\sqrt{\di}\sqrt{\di}=\sqrt{\en}\di.\no
\end{align}
and
\begin{align}
\abs{\int_{-\ell}^{\ell}b\Big(\dt^2\p_1\e\Big)\Big(2\dt J_1\dt^2\p_1\e\Big)}
\ls& \lms{\dt^2\p_1\e}{2}\lms{\dt J_1}{\infty}\lms{\dt^2\p_1\e}{2}\\
\ls& \hms{\dt^2\e}{1}^2\Big(\abs{\dt l}+\abs{\dt r}\Big)\ls\Big(\sqrt{\di}\Big)^2\sqrt{\en}=\sqrt{\en}\di.\no
\end{align}
For the third term we then write
\begin{align}
\int_{-\ell}^{\ell}b\Big(\dt^2\p_1\e\Big)\Big((J_1-1)\dt^3\p_1\e\Big)=\dt\bigg(\int_{-\ell}^{\ell}b(J_1-1)\abs{\dt^2\p_1\e}^2\bigg)-\int_{-\ell}^{\ell}b\dt J_1\abs{\dt^2\p_1\e}^2,
\end{align}
where we have
\begin{align}
\abs{\int_{-\ell}^{\ell}b\dt J_1\abs{\dt^2\p_1\e}^2}\ls& \lms{\dt J_1}{\infty}\lms{\dt^2\p_1\e}{2}^2\\
\ls&\Big(\abs{\dt l}+\abs{\dt r}\Big)\hms{\dt^2\e}{1}^2\ls \sqrt{\en}\Big(\sqrt{\di}\Big)^2=\sqrt{\en}\di.\no
\end{align}
\ \\
{\bf{Second Term of $\dt^2\p_1\ss$: $\dt^2\p_1\Big(\tilde a\p_1\e\Big)$}\\}
We can directly compute
\begin{align}
\dt^2\p_1\Big(\tilde a\p_1\e\Big) = -\dt^3k_1\p_1\e+\dt^2\tilde a\p_1^2\e-2\dt^2k_1\dt\p_1\e+2\dt \tilde a\dt\p_1^2\e-\dt k_1\dt^2\p_1\e+\tilde a\dt^2\p_1^2\e.
\end{align}
We estimate each term as follows
\begin{align}
\abs{\int_{-\ell}^{\ell}b\Big(\dt^2\p_1\e\Big)\Big(\dt^3k_1\p_1\e\Big)}
\ls&\lms{\dt^2\p_1\e}{2}\lms{\dt^3k_1}{\infty}\lms{\p_1\e}{2}\\
\ls&\hms{\dt^2\e}{1}\Big(\abs{\dt^3 l}+\abs{\dt^3 r}\Big)\hms{\e}{1}\ls\sqrt{\en}\sqrt{\di}\sqrt{\di}=\sqrt{\en}\di, \no
\end{align}
\begin{align}
\abs{\int_{-\ell}^{\ell}b\Big(\dt^2\p_1\e\Big)\Big(\dt^2\tilde a\p_1^2\e\Big)}
\ls&\lms{\dt^2\p_1\e}{2}\lms{\dt^2\tilde a}{\infty}\lms{\p_1^2\e}{2}\\
\ls&\hms{\dt^2\e}{1}\Big(\abs{\dt^3 l}+\abs{\dt^3 r}\Big)\hms{\e}{2}\ls\sqrt{\en}\sqrt{\di}\sqrt{\di}=\sqrt{\en}\di,\no
\end{align}
\begin{align}
\abs{\int_{-\ell}^{\ell}b\Big(\dt^2\p_1\e\Big)\Big(2\dt^2k_1\dt\p_1\e\Big)}
\ls& \lms{\dt^2\p_1\e}{2}\lms{\dt^2k_1}{\infty}\lms{\dt\p_1\e}{2}\\
\ls& \hms{\dt^2\e}{1}\Big(\abs{\dt^2 l}+\abs{\dt^2 r}\Big)\hms{\dt\e}{1}\ls\sqrt{\en}\sqrt{\di}\sqrt{\di}=\sqrt{\en}\di,\no
\end{align}
\begin{align}
\abs{\int_{-\ell}^{\ell}b\Big(\dt^2\p_1\e\Big)\Big(2\dt \tilde a\dt\p_1^2\e\Big)}
\ls& \lms{\dt^2\p_1\e}{2}\lms{\dt \tilde a}{\infty}\lms{\dt\p_1^2\e}{2}\\
\ls& \hms{\dt^2\e}{1}\Big(\abs{\dt^2 l}+\abs{\dt^2 r}\Big)\hms{\dt\e}{2}\ls\sqrt{\en}\sqrt{\di}\sqrt{\di}=\sqrt{\en}\di\no
\end{align}
\begin{align}
\abs{\int_{-\ell}^{\ell}b\Big(\dt^2\p_1\e\Big)\Big(\dt k_1\dt^2\p_1\e\Big)}
\ls&\lms{\dt^2\p_1\e}{2}\lms{\dt k_1}{\infty}\lms{\dt^2\p_1\e}{2}\\
\ls&\hms{\dt^2\e}{1}\Big(\abs{\dt l}+\abs{\dt r}\Big)\hms{\dt^2\e}{1}\ls\sqrt{\en}\sqrt{\di}\sqrt{\di}=\sqrt{\en}\di.\no
\end{align}
For the remaining term we write
\begin{align}
\int_{-\ell}^{\ell}b\Big(\dt^2\p_1\e\Big)\Big(\tilde a\dt^2\p_1^2\e\Big)=\int_{-\ell}^{\ell}b\tilde a\p_1\abs{\dt^2\p_1\e}^2=-\int_{-\ell}^{\ell}\p_1\Big(b\tilde a\Big)\abs{\dt^2\p_1\e}^2+\bigg(b\tilde a\abs{\dt^2\p_1\e}^2\bigg)\bigg|_{-\ell}^{\ell}
\end{align}
and then estimate
\begin{align}
\abs{\int_{-\ell}^{\ell}\p_1\Big(b\tilde a\Big)\abs{\dt^2\p_1\e}^2}\ls& \lms{\p_1\Big(b\tilde a\Big)}{\infty}\lms{\dt^2\p_1\e}{2}^2\\
\ls&\Big(\abs{\dt l}+\abs{\dt r}\Big)\hms{\dt^2\e}{1}^2\ls \sqrt{\en}\Big(\sqrt{\di}\Big)^2=\sqrt{\en}\di.\no
\end{align}
and
\begin{align}
\abs{\bigg(b\tilde a\abs{\dt^2\p_1\e}^2\bigg)\bigg|_{-\ell}^{\ell}} \ls \abs{b\tilde a}\bigg(\abs{\dt^2\p_1\e}^2\bigg|_{-\ell}+\abs{\dt^2\p_1\e}^2\bigg|_{\ell}\bigg)\ls \sqrt{\en}\Big(\sqrt{\di}\Big)^2=\sqrt{\en}\di.
\end{align}
\ \\
{\bf{Third Term of $\dt^2\p_1\ss$: $\dt^2\p_1\Big(\oo\p_1\z_0\Big)$}\\}
We can directly compute
\begin{align}
\dt^2\p_1\Big(\oo\p_1\z_0\Big) = \dt^2\p_1\oo\p_1\z_0+\dt^2\oo\p_1^2\z_0.
\end{align}
We estimate each term via:
\begin{align}
\abs{\int_{-\ell}^{\ell}b\Big(\dt^2\p_1\e\Big)\Big(\dt^2\p_1\oo\p_1\z_0\Big)}
\ls& \lms{\dt^2\p_1\e}{2}\lms{\dt^2\p_1\oo}{2}\\
\ls& \hms{\dt^2\e}{1}\bigg(\sum_{j=0}^2\Big(\abs{\dt^{j+1} l}+\abs{\dt^{j+1} r}\Big)\bigg)^2\ls \sqrt{\en}\Big(\sqrt{\di}\Big)^2=\sqrt{\en}\di.\no
\end{align}
and
\begin{align}
\abs{\int_{-\ell}^{\ell}b\Big(\dt^2\p_1\e\Big)\Big(\dt^2\oo\p_1^2\z_0\Big)}
\ls& \lms{\dt^2\p_1\e}{2}\lms{\dt^2\oo}{2}\\
\ls& \hms{\dt^2\e}{1}\bigg(\sum_{j=0}^2\Big(\abs{\dt^{j+1} l}+\abs{\dt^{j+1} r}\Big)\bigg)^2\ls \sqrt{\en}\Big(\sqrt{\di}\Big)^2=\sqrt{\en}\di.\no
\end{align}
\ \\
{\bf{First Term of $\p_1\s_5$: $\p_1\Big(u\cdot\dt^2\n\Big)$}\\}
We can directly compute
\begin{align}
\p_1\Big(u\cdot\dt^2\n\Big) = \p_1u\cdot\dt^2\n+u\cdot\dt^2\p_1\n.
\end{align}
For the first term, we choose $p=\dfrac{3+\d}{2+2\d}$ and $q=\dfrac{6+2\d}{1-\d}$ such that $\dfrac{1}{p}+\dfrac{2}{q}=1$.
\begin{align}
\abs{\int_{-\ell}^{\ell}b\Big(\dt^2\p_1\e\Big)\Big(\p_1u\cdot\dt^2\n\Big)}
\ls& \lms{\dt^2\p_1\e}{q}\lms{\p_1u}{p}\lms{\dt^2\p_1\e}{q}\\
\ls& \hms{\dt^2\p_1\e}{\frac{1}{2}}\hmwss{\p_1u}{\frac{1}{2}}\hms{\dt^2\p_1\e}{\frac{1}{2}}\no\\
\ls& \hms{\dt^2\e}{\frac{3}{2}}\hmw{u}{2}\hms{\dt^2\e}{\frac{3}{2}}\ls\sqrt{\en}\sqrt{\di}\sqrt{\di}=\sqrt{\en}\di.\no
\end{align}
For the second term, we have
\begin{align}
\int_{-\ell}^{\ell}b\Big(\dt^2\p_1\e\Big)\Big(u\cdot\dt^2\p_1\n\Big)= \int_{-\ell}^{\ell}bu_1\p_1\abs{\dt^2\p_1\e}^2
=-\int_{-\ell}^{\ell}\p_1
\Big(bu_1\Big)\abs{\dt^2\p_1\e}^2+\bigg(bu_1\abs{\dt^2\p_1\e}^2\bigg)\bigg|_{-\ell}^{\ell},
\end{align}
and then bound
\begin{align}
\abs{\int_{-\ell}^{\ell}\p_1(bu_1)\abs{\dt^2\p_1\e}^2} \ls& \lms{\p_1(bu_1)}{p}\lms{\dt^2\p_1\e}{q}^2\\
\ls& \hmw{u}{2}\hms{\dt^2\e}{\frac{3}{2}}^2\ls \sqrt{\en}\Big(\sqrt{\di}\Big)^2=\sqrt{\en}\di \no
\end{align}
and
\begin{align}
\abs{\bigg(bu_1\abs{\dt^2\p_1\e}^2\bigg)\bigg|_{-\ell}^{\ell}} \ls& \bigg(\abs{u_1}\bigg|_{-\ell}+
\abs{u_1}\bigg|_{\ell}\bigg)
\bigg(\abs{\dt^2\p_1\e}^2\bigg|_{-\ell}+\abs{\dt^2\p_1\e}^2\bigg|_{\ell}\bigg)\\
\ls& \Big(\abs{\dt l}+\abs{\dt r}\Big)\bigg(\abs{\dt^2\p_1\e}^2\bigg|_{-\ell}+\abs{\dt^2\p_1\e}^2\bigg|_{\ell}\bigg)\ls \sqrt{\en}\Big(\sqrt{\di}\Big)^2=\sqrt{\en}\di.\no
\end{align}
\ \\
{\bf{Second Term of $\p_1\s_5$: $\p_1\Big(\dt u\cdot\dt\n\Big)$}\\}
We can directly compute
\begin{align}
\p_1\Big(\dt u\cdot\dt\n\Big) = \dt\p_1u\cdot\dt\n+\dt u\cdot\dt\p_1\n.
\end{align}
We estimate each term as follows:
\begin{align}
\abs{\int_{-\ell}^{\ell}b\Big(\dt^2\p_1\e\Big)\Big(\dt\p_1u\cdot\dt\n\Big)}
\ls& \lms{\dt^2\p_1\e}{q}\lms{\dt\p_1u}{p}\lms{\dt\p_1\e}{q}\\
\ls& \hms{\dt^2\p_1\e}{\frac{1}{2}}\hmwss{\dt\p_1u}{\frac{1}{2}}\hms{\dt\p_1\e}{\frac{1}{2}}\no\\
\ls& \hms{\dt^2\e}{\frac{3}{2}}\hmw{\dt u}{2}\hms{\dt\e}{\frac{3}{2}}\ls\sqrt{\en}\sqrt{\di}\sqrt{\di}=\sqrt{\en}\di\no
\end{align}
and
\begin{align}
\abs{\int_{-\ell}^{\ell}b\Big(\dt^2\p_1\e\Big)\Big(\dt u\cdot\dt\p_1\n\Big)}
\ls& \lms{\dt^2\p_1\e}{q}\lms{\dt u}{q}\lms{\dt\p_1^2\e}{p}\\
\ls& \hms{\dt^2\p_1\e}{\frac{1}{2}}\hms{\dt u}{\frac{1}{2}}\hmwss{\dt\p_1^2\e}{\frac{1}{2}}\no\\
\ls& \hms{\dt^2\e}{\frac{3}{2}}\hm{\dt u}{1}\hmwss{\dt\e}{\frac{5}{2}}\ls\sqrt{\en}\sqrt{\di}\sqrt{\di}=\sqrt{\en}\di.\no
\end{align}
\end{proof}

\subsection{Other nonlinear estimates}

We now turn our attention to bounding various other nonlinear terms that appear in the analysis.  We begin with $\oo$.

\begin{lemma}\label{nonlinear lemma 7}
Let $\oo$ be given by \eqref{ap oo}. We have
\begin{equation}
\abs{\kappa\bigg(\dt\rp\dt^2\oo\bigg|_{\ell}+\dt\lp\dt^2\oo\bigg|_{-\ell}\bigg)} \ls  \sqrt{\en}\di.\no
\end{equation}
\end{lemma}
\begin{proof}
By definition $\dt\rp=\dt^3R$ and $\dt\lp=\dt^3L$.
Then we estimate
\begin{align}
&\abs{\kappa\bigg(\dt\rp\dt^2\oo\bigg|_{\ell}+\dt\lp\dt^2\oo\bigg|_{-\ell}\bigg)}
\ls \Big(\abs{\dt^3l}+\abs{\dt^3r}\Big)\Big(\abs{\dt l}+\abs{\dt r}\Big)\Big(\abs{\dt^2 l}+\abs{\dt^2 r}\Big) \\
& +\Big(\abs{\dt^3r}+\abs{\dt^3l}\Big)\abs{k_1}\Big(\abs{\dt^3 l}+\abs{\dt^3 r}\Big)
\ls \sqrt{\di}\sqrt{\di}\sqrt{\en}+\sqrt{\di}\sqrt{\en}\sqrt{\di}=\sqrt{\en}\di.\no
\end{align}
\end{proof}

Next we consider $\qq$ and $\oo$.

\begin{lemma}\label{nonlinear lemma 8}
Let $\qq$ be given by \eqref{ap qq} and $\oo$ be given by \eqref{ap oo}. We have
\begin{align}
\abs{\sigma\dt^2\qq\Big(\dt\lp+\dt^2\oo\Big)\bigg|_{\ell}-\sigma\dt^2\qq\Big(\dt\rp+\dt^2\oo\Big)\bigg|_{-\ell}} \ls  \sqrt{\en}\di.
\end{align}
\end{lemma}
\begin{proof}
By Lemma \ref{estimate_o}, we have
\begin{align}
\abs{\Big(\dt\lp+\dt^2\oo\Big)\Big|_{\ell}} + \abs{\Big(\dt\lp+\dt^2\oo\Big)\Big|_{-\ell}} \ls \sqrt{\di}.
\end{align}
Also, it is easy to check
\begin{align}
\abs{\dt^2\qq\Big|_{\pm\ell}} \ls& \abs{k_1}\abs{\dt^2 k_1} + \abs{\dt k_1}^2 + \abs{\dt^2 k_1}\abs{\p_1\e(\pm\ell)} + \abs{k_1}\abs{\dt^2 \p_1\e(\pm\ell)} \\
&+ \abs{\p_1\e(\pm\ell)}\abs{\dt^2\p_1\e(\pm\ell)} + \abs{\dt\p_1\e(\pm\ell)}^2
\ls \sqrt{\en}\sqrt{\di}.\no
\end{align}
Therefore, our result easily follows.
\end{proof}

We may write $\rr=\rr(\varpi_1,\varpi_2)$ where $\varpi_1=k_1$ and $\varpi_2=\p_1\e$. Let $\p_{\varpi_i}\rr$ denote the derivative of $\rr$ with respect to $\varpi_i$ for $i=1,2$.
\begin{lemma}\label{nonlinear lemma 11}
Let $\rr$ be given by \eqref{ap rr}. We have the estimate
\begin{align}
\abs{\sigma\int_{\Sigma_0}\p_1\dt^2\rr(v\cdot\n)-\frac{\ud \sf_2}{\ud t}} \ls \sqrt{\en}\di,
\end{align}
where
\begin{align}
\sf_2 = \int_{-\ell}^{\ell}\frac{\p_{\varpi_2}\rr J_1}{\sqrt{1+\abs{\p_{1}\z_0}^2}}\abs{\dt^2\p_1\e}^2.
\end{align}
\end{lemma}
\begin{proof}
Using the transport equation in \eqref{system 2} and integrating by parts, for $v=\dt^2u$, we know
\begin{align}
\sigma\int_{\Sigma_0}\p_1\dt^2\rr(v\cdot\n) =& \sigma\int_{\Sigma_0}\p_1\dt^2\rr\Big(bJ_1\dt^3\e+\dt^2(b\tilde a\p_1\e)+\dt u\cdot\dt\n+u\cdot\dt^2\n\Big)\\
=& -\sigma\int_{\Sigma_0}\dt^2\rr\bigg(\p_1\Big(bJ_1\dt^3\e\Big)+\p_1\Big(\dt^2(b\tilde a\p_1\e)\Big)+\p_1\Big(\dt u\cdot\dt\n\Big)+\p_1\Big(u\cdot\dt^2\n\Big)\bigg)\no\\
& +\sigma\bigg(\dt^2\rr\Big(bJ_1\dt^3\e+\dt^2(b\tilde a\p_1\e)+\dt u\cdot\dt\n+u\cdot\dt^2\n\Big)\bigg)\bigg|_{-\ell}^{\ell},\no
\end{align}
where here we have written $b(x_1)=\left(1+\abs{\p_{1}\z_0}^2\right)^{-1/2}$.

We may directly verify that
\begin{align}
\dt\rr =&\p_{\varpi_1}\rr\dt k_1+ \p_{\varpi_2}\rr\dt\p_1\e,\\
\dt^2\rr =&\p_{\varpi_1}\rr\dt^2k_1+\p_{\varpi_1}^2\rr(\dt k_1)^2+\p_{\varpi_1}\p_{\varpi_2}\rr\dt k_1\dt\p_1\e+ \p_{\varpi_2}^2\rr(\dt\p_1\e)^2+\p_{\varpi_2}\rr\dt^2\p_1\e.
\end{align}
It is easy to check that
\begin{align}
\abs{\p_{\varpi_1}\rr}+\abs{\p_{\varpi_2}\rr}+\abs{\p_{\varpi_1}^2\rr}+\abs{\p_{\varpi_1}\p_{\varpi_2}\rr}+\abs{\p_{\varpi_2}^2\rr}\ls \abs{k_1}+\abs{\p_1\e}.
\end{align}
Since the terms related to $k_1$ are easier to estimate, we will focus on the terms related to $\p_{\varpi_2}^2(\dt\p_1\e)^2$ and $\p_{\varpi_2}\dt^2\p_1\e$.  We will proceed with the estimates term by term.

\textbf{First Integral Term, $\p_1\Big(bJ_1\dt^3\e\Big)$:}
We can directly compute
\begin{align}
\p_1\Big(bJ_1\dt^3\e\Big) = \p_1bJ_1\dt^3\e+bJ_1\dt^3\p_1\e.
\end{align}
In the following, we choose $p=\dfrac{3+\d}{2+2\d}$ and $q=\dfrac{6+2\d}{1-\d}$ such that $\dfrac{1}{p}+\dfrac{2}{q}=1$. The $\p_1bJ_1\dt^3\e$ terms can be directly estimated:
\begin{align}
\abs{\int_{-\ell}^{\ell}\p_{\varpi_2}^2\rr (\dt\p_1\e)^2\Big(\p_1bJ_1\dt^3\e\Big)}
\ls& \lms{\dt\p_1\e}{q}\lms{\dt\p_1\e}{q}\lms{\dt^3\e}{p}\\
\ls& \hms{\dt\e}{\frac{3}{2}}\hms{\e}{\frac{3}{2}}\hmwss{\dt^3\e}{\frac{1}{2}}\ls\sqrt{\en}\sqrt{\di}\sqrt{\di}=\sqrt{\en}\di, \no
\end{align}
and
\begin{align}
\abs{\int_{-\ell}^{\ell}\p_{\varpi_2}\rr \dt^2\p_1\e\Big(\p_1bJ_1\dt^3\e\Big)}
\ls& \lms{\p_1\e}{q}\lms{\dt^2\p_1\e}{q}\lms{\dt^3\e}{p}\\
\ls& \hms{\e}{\frac{3}{2}}\hms{\dt^2\e}{\frac{3}{2}}\hmwss{\dt^3\e}{\frac{1}{2}}\ls\sqrt{\en}\sqrt{\di}\sqrt{\di}=\sqrt{\en}\di.\no
\end{align}
The $bJ_1\dt^3\p_1\e$ terms require more efforts. We integrate by parts to obtain
\begin{align}
& \int_{-\ell}^{\ell}\p_{\varpi_2}^2\rr (\dt\p_1\e)^2\Big(bJ_1\dt^3\p_1\e\Big) = -\int_{-\ell}^{\ell}\p_1\p_{\varpi_2}^2\rr (\dt\p_1\e)^2(bJ_1)\dt^3\e  \\
&-\int_{-\ell}^{\ell}\p_{\varpi_2}^2\rr (\dt\p_1\e)(\dt\p_1^2\e)(bJ_1)\dt^3\e
-\int_{-\ell}^{\ell}\p_{\varpi_2}^2\rr (\dt\p_1\e)^2\p_1(bJ_1)\dt^3\e = I+II+III,\no
\end{align}
where the contact point terms vanish since $\dt^3\e(-\ell)=\dt^3\e(\ell)=0$. Then we have
\begin{align}
\abs{I}=&\abs{\int_{-\ell}^{\ell}\p_1\p_{\varpi_2}^2\rr (\dt\p_1\e)^2(bJ_1)\dt^3\e}\ls\lms{\dt\p_1\e}{q}^2\lms{\dt^3\e}{p}\\
\ls& \hms{\dt\e}{\frac{3}{2}}^2\hmwss{\dt^3\e}{\frac{1}{2}}\ls \sqrt{\en}\sqrt{\di}\sqrt{\di}=\sqrt{\en}\di,\no
\end{align}
\begin{align}
\abs{II}=&\abs{\int_{-\ell}^{\ell}\p_{\varpi_2}^2\rr (\dt\p_1\e)(\dt\p_1^2\e)(bJ_1)\dt^3\e}\ls \lms{\dt\p_1\e}{q}\lms{\dt\p_1^2\e}{q}\lms{\dt^3\e}{p}\\
\ls& \hms{\dt\e}{\frac{1}{2}}\hms{\dt\e}{\frac{5}{2}}\hmwss{\dt^3\e}{\frac{1}{2}}\ls \sqrt{\en}\sqrt{\di}\sqrt{\di}=\sqrt{\en}\di,\no
\end{align}
and
\begin{align}
\abs{III}=& \abs{\int_{-\ell}^{\ell}\p_{\varpi_2}^2\rr (\dt\p_1\e)^2\p_1(bJ_1)\dt^3\e}\ls\lms{\dt\p_1\e}{q}^2 \lms{\dt^3\e}{p}\\
\ls& \hms{\dt\e}{\frac{3}{2}}^2\hmwss{\dt^3\e}{\frac{1}{2}}\ls \sqrt{\en}\sqrt{\di}\sqrt{\di}=\sqrt{\en}\di.\no
\end{align}
For the remaining term we compute
\begin{align}
\int_{-\ell}^{\ell}\p_{\varpi_2}\rr \dt^2\p_1\e\Big(bJ_1\dt^3\p_1\e\Big)=\dt\bigg(\int_{-\ell}^{\ell}\p_{\varpi_2}\rr bJ_1\abs{\dt^2\p_1\e}^2\bigg)-\int_{-\ell}^{\ell}b\dt(\p_{\varpi_2}\rr J_1)\abs{\dt^2\p_1\e}^2,
\end{align}
and then estimate
\begin{align}
\abs{\int_{-\ell}^{\ell}b\dt(\p_{\varpi_2}\rr J_1)\abs{\dt^2\p_1\e}^2} \ls& \bigg(\lms{\dt J_1}{\infty}+\lms{\dt\p_1\e}{\infty}\bigg)\lms{\dt^2\p_1\e}{2}^2 \\
\ls&  \Big(\abs{\dt l}+\abs{\dt r}+\hms{\dt\e}{2}\Big)\hms{\dt^2\e}{1}^2\ls \sqrt{\en}\Big(\sqrt{\di}\Big)^2=\sqrt{\en}\di.\no
\end{align}

\textbf{Second Integral Term, $\p_1\Big(\dt^2(b\tilde a\p_1\e)\Big)$:}
We begin by computing
\begin{align}
\dt^2\p_1\Big(b\tilde a\p_1\e\Big) =& \p_1b\dt^2\tilde a\p_1\e+\p_1b\dt\tilde a\dt\p_1\e+\p_1b\tilde a\dt^2\p_1\e\\
& -b\dt^3k_1\p_1\e+b\dt^2\tilde a\p_1^2\e-2b\dt^2k_1\dt\p_1\e+2b\dt \tilde a\dt\p_1^2\e-b\dt k_1\dt^2\p_1\e + b\tilde a\dt^2\p_1^2\e .\no
\end{align}
For the first two of these terms we may bound
\begin{align}
\abs{\int_{-\ell}^{\ell}\p_{\varpi_2}^2\rr (\dt\p_1\e)^2\Big(\p_1b\dt^2\tilde a\p_1\e\Big)}\ls&\lms{\dt\p_1\e}{q}^2\lms{\dt^2\tilde a}{\infty}\lms{\p_1\e}{p}\\
\ls&\hms{\dt\e}{\frac{3}{2}}^2\Big(\abs{\dt^3l}+\abs{\dt^3r}\Big)\hmwss{\e}{\frac{3}{2}}\ls \sqrt{\en}\sqrt{\di}\sqrt{\di}=\sqrt{\en}\di.\no \\
\abs{\int_{-\ell}^{\ell}\p_{\varpi_2}\rr \dt^2\p_1\e\Big(\p_1b\dt^2\tilde a\p_1\e\Big)}\ls&\lms{\dt^2\p_1\e}{2}\lms{\dt^2\tilde a}{\infty}\lms{\p_1\e}{2}\\
\ls&\hms{\dt^2\e}{1}^2\Big(\abs{\dt^3l}+\abs{\dt^3r}\Big)\hms{\e}{1}\ls \sqrt{\en}\sqrt{\di}\sqrt{\di}=\sqrt{\en}\di.\no
\end{align}
Arguing similarly for all but the last term, we conclude that
\begin{align}
\abs{\sigma\int_{\Sigma_0}\dt^2\rr \bigg(  \p_1 \dt^2(b\tilde a\p_1\e)  - b\tilde a\dt^2\p_1^2\e  \bigg) } \ls \sqrt{\en} \di.
\end{align}
The last term is much more complicated. To handle it we first integrate by parts to obtain
\begin{align}
& \int_{-\ell}^{\ell}\p_{\varpi_2}^2\rr (\dt\p_1\e)^2\Big(b\tilde a\dt^2\p_1^2\e\Big)
= -\int_{-\ell}^{\ell}\p_1\p_{\varpi_2}^2\rr (\dt\p_1\e)^2(b\tilde a)\dt^2\p_1\e \\
& - \int_{-\ell}^{\ell}\p_{\varpi_2}^2\rr (\dt\p_1\e)(\dt\p_1^2\e)(b\tilde a)\dt^2\p_1\e
-\int_{-\ell}^{\ell}\p_{\varpi_2}^2\rr (\dt\p_1\e)^2\p_1(b\tilde a)\dt^2\p_1\e+\p_{\varpi_2}^2\rr (\dt\p_1\e)^2\Big(b\tilde a\dt^2\p_1\e\Big)\bigg|_{-\ell}^{\ell} \no \\
=&I+II+III+IV.\no
\end{align}
Then we have
\begin{align}
\abs{I}=&\abs{\int_{-\ell}^{\ell}\p_1\p_{\varpi_2}^2\rr (\dt\p_1\e)^2(b\tilde a)\dt^2\p_1\e}\ls\lms{\dt\p_1\e}{q}^2\lms{\dt^2\p_1\e}{p}\\
\ls& \hms{\dt\e}{\frac{3}{2}}^2\hmwss{\dt^2\e}{\frac{3}{2}}\ls \sqrt{\en}\sqrt{\di}\sqrt{\di}=\sqrt{\en}\di,\no
\end{align}
\begin{align}
\abs{II} =& \abs{\int_{-\ell}^{\ell}\p_{\varpi_2}^2\rr (\dt\p_1\e)(\dt\p_1^2\e)(b\tilde a)\dt^2\p_1\e}\ls \lms{\dt\p_1\e}{q}\lms{\dt\p_1^2\e}{p}\lms{\dt^2\p_1\e}{q}\\
\ls& \hms{\dt\e}{\frac{1}{2}}\hmwss{\dt\e}{\frac{5}{2}}\hms{\dt^2\e}{\frac{3}{2}}\ls \sqrt{\en}\sqrt{\di}\sqrt{\di}=\sqrt{\en}\di,\no
\end{align}
and
\begin{align}
\abs{III} =& \abs{\int_{-\ell}^{\ell}\p_{\varpi_2}^2\rr (\dt\p_1\e)^2\p_1(b\tilde a)\dt^2\p_1\e}\ls\lms{\dt\p_1\e}{q}^2\lms{\dt^2\p_1\e}{p}\\
\ls& \hms{\dt\e}{\frac{3}{2}}^2\hmwss{\dt^2\e}{\frac{3}{2}}\ls \sqrt{\en}\sqrt{\di}\sqrt{\di}=\sqrt{\en}\di.\no
\end{align}
The contact point term
\begin{align}
\abs{IV} \ls& \abs{\p_{\varpi_2}^2\rr (\dt\p_1\e)^2\Big(b\tilde a\dt^2\p_1\e\Big)\bigg|_{-\ell}^{\ell}}\\
\ls& \bigg(\abs{\dt\p_1\e(-\ell)}^2+\abs{\dt\p_1\e(\ell)}^2\bigg)
\bigg(\abs{\dt^2\p_1\e(-\ell)}+\abs{\dt^2\p_1\e(\ell)}\bigg)\ls \sqrt{\en}\sqrt{\di}\sqrt{\di}=\sqrt{\en}\di.\no
\end{align}
On the other hand,
\begin{align}
\int_{-\ell}^{\ell}\p_{\varpi_2}\rr \dt^2\p_1\e\Big(b\tilde a\dt^2\p_1^2\e\Big)=&\int_{-\ell}^{\ell}\p_{\varpi_2}\rr b\tilde a\p_1\abs{\dt^2\p_1\e}^2\\
=&-\int_{-\ell}^{\ell}\p_1\Big(\p_{\varpi_2}\rr b\tilde a\Big)\abs{\dt^2\p_1\e}^2+\bigg(\p_{\varpi_2}\rr b\tilde a\abs{\dt^2\p_1\e}^2\bigg)\bigg|_{-\ell}^{\ell},\no
\end{align}
where we have the bounds
\begin{align}
\abs{\int_{-\ell}^{\ell}\p_1\Big(\p_{\varpi_2}\rr b\tilde a\Big)\abs{\dt^2\p_1\e}^2} \ls& \lms{\p_1\Big(\p_{\varpi_2}\rr b\tilde a\Big)}{p}\lms{\dt^2\p_1\e}{q}^2\\
\ls&\Big(\abs{\dt l}+\abs{\dt r}+\hmwss{\e}{\frac{5}{2}}\Big)\hms{\dt^2\e}{\frac{3}{2}}^2\ls \sqrt{\en} \Big(\sqrt{\di}\Big)^2=\sqrt{\en}\di.\no
\end{align}
and
\begin{align}
\abs{\bigg(\p_{\varpi_2}\rr b\tilde a\abs{\dt^2\p_1\e}^2\bigg)\bigg|_{-\ell}^{\ell}} \ls \abs{\tilde a}\bigg(\abs{\dt^2\p_1\e}^2\bigg|_{-\ell}+\abs{\dt^2\p_1\e}^2\bigg|_{\ell}\bigg)\ls \sqrt{\en}\Big(\sqrt{\di}\Big)^2=\sqrt{\en}\di.
\end{align}

\textbf{Third Integral Term, $\p_1\Big(\dt u\cdot\dt\n\Big)$:}
We can directly compute
\begin{align}
\p_1\Big(\dt u\cdot\dt\n\Big) = \dt\p_1u\cdot\dt\n+\dt u\cdot\dt\p_1\n.
\end{align}
We estimate each term. In the following, we choose $p=\dfrac{3+\d}{2+2\d}$ and $q=\dfrac{6+2\d}{1-\d}$ such that $\dfrac{1}{p}+\dfrac{2}{q}=1$:
\begin{align}
\abs{\int_{-\ell}^{\ell}\p_{\varpi_2}^2\rr (\dt\p_1\e)^2\Big(\dt\p_1u\cdot\dt\n\Big)}
\ls& \lms{\dt\p_1\e}{q}^2\lms{\dt\p_1u}{p}\lms{\dt\p_1\e}{\infty}\\
\ls& \hms{\dt\p_1\e}{\frac{1}{2}}^2\hmwss{\dt\p_1u}{\frac{1}{2}}\hms{\dt\p_1\e}{1}\no\\
\ls& \hms{\dt\e}{\frac{3}{2}}^2\hmw{\dt u}{2}\hms{\dt\e}{2}\ls\sqrt{\en}\sqrt{\di}\sqrt{\di}=\sqrt{\en}\di,\no
\end{align}
\begin{align}
\abs{\int_{-\ell}^{\ell}\p_{\varpi_2}\rr \dt^2\p_1\e\Big(\dt\p_1u\cdot\dt\n\Big)}
\ls& \lms{\dt^2\p_1\e}{q}\lms{\dt\p_1u}{p}\lms{\dt\p_1\e}{q}\\
\ls& \hms{\dt^2\p_1\e}{\frac{1}{2}}\hmwss{\dt\p_1u}{\frac{1}{2}}\hms{\dt\p_1\e}{\frac{1}{2}}\no\\
\ls& \hms{\dt^2\e}{\frac{3}{2}}\hmw{\dt u}{2}\hms{\dt\e}{\frac{3}{2}}\ls\sqrt{\en}\sqrt{\di}\sqrt{\di}=\sqrt{\en}\di,\no
\end{align}
\begin{align}
\abs{\int_{-\ell}^{\ell}\p_{\varpi_2}^2\rr (\dt\p_1\e)^2\Big(\dt u\cdot\dt\p_1\n\Big)}
\ls& \lms{\dt\p_1\e}{q}^2\lms{\dt u}{\infty}\lms{\dt\p_1^2\e}{p}\\
\ls& \hms{\dt\p_1\e}{\frac{1}{2}}^2\hms{\dt u}{1}\hmwss{\dt\p_1^2\e}{\frac{1}{2}}\no\\
\ls& \hms{\dt\e}{\frac{3}{2}}^2\hmw{\dt u}{2}\hmwss{\dt\e}{\frac{5}{2}}\ls\sqrt{\en}\sqrt{\di}\sqrt{\di}=\sqrt{\en}\di,\no
\end{align}
\begin{align}
\abs{\int_{-\ell}^{\ell}\p_{\varpi_2}\rr \dt^2\p_1\e\Big(\dt u\cdot\dt\p_1\n\Big)}
\ls& \lms{\dt^2\p_1\e}{q}\lms{\dt u}{q}\lms{\dt\p_1^2\e}{p}\\
\ls& \hms{\dt^2\p_1\e}{\frac{1}{2}}\hms{\dt u}{\frac{1}{2}}\hmwss{\dt\p_1^2\e}{\frac{1}{2}}\no\\
\ls& \hms{\dt^2\e}{\frac{3}{2}}\hm{\dt u}{1}\hmwss{\dt\e}{\frac{5}{2}}\ls\sqrt{\en}\sqrt{\di}\sqrt{\di}=\sqrt{\en}\di.\no
\end{align}

\textbf{Fourth Integral Term, $\p_1\Big(u\cdot\dt^2\n\Big)$:}
We can directly compute
\begin{align}
\p_1\Big(u\cdot\dt^2\n\Big) = \p_1u\cdot\dt^2\n+u\cdot\dt^2\p_1\n.
\end{align}
We estimate each term. In the following, we choose $p=\dfrac{3+\d}{2+2\d}$ and $q=\dfrac{6+2\d}{1-\d}$ such that $\dfrac{1}{p}+\dfrac{2}{q}=1$:
\begin{align}
&\abs{\int_{-\ell}^{\ell}\p_{\varpi_2}^2\rr (\dt\p_1\e)^2\Big(\p_1u\cdot\dt^2\n\Big)}
\ls \lms{\dt\p_1\e}{\infty}\lms{\dt\p_1\e}{q}\lms{\p_1u}{p}\lms{\dt^2\p_1\e}{q}\\
\ls& \hms{\dt\e}{2}\hms{\dt\e}{\frac{3}{2}}\hmw{u}{2}\hms{\dt^2\e}{\frac{3}{2}}\ls\sqrt{\en}\sqrt{\di}\sqrt{\di}=\sqrt{\en}\di,\no
\end{align}
\begin{align}
\abs{\int_{-\ell}^{\ell}\p_{\varpi_2}\rr \dt^2\p_1\e\Big(\p_1u\cdot\dt^2\n\Big)}
\ls& \lms{\dt^2\p_1\e}{q}\lms{\p_1u}{p}\lms{\dt^2\p_1\e}{q}\\
\ls& \hms{\dt^2\p_1\e}{\frac{1}{2}}\hmwss{\p_1u}{\frac{1}{2}}\hms{\dt^2\p_1\e}{\frac{1}{2}}\no\\
\ls& \hms{\dt^2\e}{\frac{3}{2}}\hmw{u}{2}\hms{\dt^2\e}{\frac{3}{2}}\ls\sqrt{\en}\sqrt{\di}\sqrt{\di}=\sqrt{\en}\di.\no
\end{align}
The last term is much more complicated. We integrate by parts to obtain
\begin{align}
& \int_{-\ell}^{\ell}\p_{\varpi_2}^2\rr (\dt\p_1\e)^2\Big(u\cdot\dt^2\p_1\n\Big) = -\int_{-\ell}^{\ell}\p_1\p_{\varpi_2}^2\rr (\dt\p_1\e)^2u_1\dt^2\p_1\e
-\int_{-\ell}^{\ell}\p_{\varpi_2}^2\rr (\dt\p_1\e)(\dt\p_1^2\e)u_1\dt^2\p_1\e\no\\
& -\int_{-\ell}^{\ell}\p_{\varpi_2}^2\rr (\dt\p_1\e)^2\p_1u_1\dt^2\p_1\e+\p_{\varpi_2}^2\rr (\dt\p_1\e)^2\Big(u_1\dt^2\p_1\e\Big)\bigg|_{-\ell}^{\ell}
 =: I+II+III+IV,\no
\end{align}
Then we have
\begin{align}
\abs{I} =&\abs{\int_{-\ell}^{\ell}\p_1\p_{\varpi_2}^2\rr (\dt\p_1\e)^2u_1\dt^2\p_1\e}\ls\lms{\dt\p_1\e}{\infty}\lms{\dt\p_1\e}{q}\lms{u}{p}\lms{\dt^2\p_1\e}{q} \\
\ls& \hms{\dt\e}{2}\hms{\dt\e}{\frac{3}{2}}\hmw{u}{2}\hms{\dt^2\e}{\frac{3}{2}}\ls \sqrt{\en}\sqrt{\di}\sqrt{\di}=\sqrt{\en}\di \no,
\end{align}
\begin{align}
\abs{II} =& \abs{\int_{-\ell}^{\ell}\p_{\varpi_2}^2\rr (\dt\p_1\e)(\dt\p_1^2\e)u_1\dt^2\p_1\e}\ls \lms{\dt\p_1\e}{\infty}\lms{\dt\p_1^2\e}{p}\lms{u}{q}\lms{\dt^2\p_1\e}{q} \\
\ls& \hms{\dt\e}{2}\hmwss{\dt\e}{\frac{5}{2}}\hm{u}{1}\hms{\dt^2\e}{\frac{3}{2}}\ls \sqrt{\en}\sqrt{\di}\sqrt{\di}=\sqrt{\en}\di, \no
\end{align}
and
\begin{align}
\abs{III}=& \abs{\int_{-\ell}^{\ell}\p_{\varpi_2}^2\rr (\dt\p_1\e)^2\p_1u_1\dt^2\p_1\e}\ls\lms{\dt\p_1\e}{\infty}\lms{\dt\p_1\e}{q}\lms{\p_1u}{p}\lms{\dt^2\p_1\e}{q} \\
\ls& \hms{\dt\e}{2}\hms{\dt\e}{\frac{3}{2}}\hmw{u}{2}\hms{\dt^2\e}{\frac{3}{2}}\ls \sqrt{\en}\sqrt{\di}\sqrt{\di}=\sqrt{\en}\di.\no
\end{align}
The contact point term
\begin{align}
\abs{IV} \ls& \abs{\p_{\varpi_2}^2\rr (\dt\p_1\e)^2\Big(u_1\dt^2\p_1\e\Big)\bigg|_{-\ell}^{\ell}}\\
\ls& \bigg(\abs{\dt\p_1\e(-\ell)}^2+\abs{\dt\p_1\e(\ell)}^2 \bigg)
\bigg(\abs{\dt^2\p_1\e(-\ell)}+\abs{\dt^2\p_1\e(\ell)}\bigg)\Big(\abs{\dt L}+\abs{\dt R}\Big)
\no \\
\ls& \sqrt{\en}\sqrt{\di}\sqrt{\di}=\sqrt{\en}\di.\no
\end{align}
On the other hand,
\begin{align}
\int_{-\ell}^{\ell}\p_{\varpi_2}\rr \dt^2\p_1\e\Big(u\cdot\dt^2\p_1\n\Big) &= \int_{-\ell}^{\ell}\p_{\varpi_2}\rr u_1\p_1\abs{\dt^2\p_1\e}^2 \\
&=-\int_{-\ell}^{\ell}\p_1
\Big(\p_{\varpi_2}\rr u_1\Big)\abs{\dt^2\p_1\e}^2+\bigg(\p_{\varpi_2}\rr u_1\abs{\dt^2\p_1\e}^2\bigg)\bigg|_{-\ell}^{\ell},\no
\end{align}
where we have
\begin{align}
\abs{\int_{-\ell}^{\ell}\p_1\Big(\p_{\varpi_2}\rr u_1\Big)\abs{\dt^2\p_1\e}^2} \ls&  \lms{\p_1\Big(\p_{\varpi_2}\rr u_1\Big)}{p}\lms{\dt^2\p_1\e}{q}^2\\
\ls& \hmw{u}{2}\hms{\dt^2\e}{\frac{3}{2}}^2\ls \sqrt{\en}\Big(\sqrt{\di}\Big)^2=\sqrt{\en}\di, \no
\end{align}
and
\begin{align}
\abs{\bigg(\p_{\varpi_2}\rr u_1\abs{\dt^2\p_1\e}^2\bigg)\bigg|_{-\ell}^{\ell}} \ls& \bigg(\abs{u_1}\bigg|_{-\ell}+\abs{u_1}\bigg|_{\ell}\bigg)
\bigg(\abs{\dt^2\p_1\e}^2\bigg|_{-\ell}+\abs{\dt^2\p_1\e}^2\bigg|_{\ell}\bigg)\\
\ls& \Big(\abs{\dt L}+\abs{\dt R}\Big)\bigg(\abs{\dt^2\p_1\e}^2\bigg|_{-\ell}+\abs{\dt^2\p_1\e}^2\bigg|_{\ell}\bigg)\ls \sqrt{\en}\Big(\sqrt{\di}\Big)^2=\sqrt{\en}\di.\no
\end{align}
Next we turn to the contact point terms:
\begin{align}
\sigma\bigg(\dt^2\rr\Big(bJ_1\dt^3\e+\dt^2(b\tilde a\p_1\e)+\dt u\cdot\dt\n+u\cdot\dt^2\n\Big)\bigg)\bigg|_{-\ell}^{\ell} \\
=\sigma\bigg(\dt^2\rr\Big(\dt^2(b\tilde a\p_1\e)+\dt u\cdot\dt\n+u\cdot\dt^2\n\Big)\bigg)\bigg|_{-\ell}^{\ell},\no
\end{align}
which follows because  $\dt^3\e(-\ell)=\dt^3\e(\ell)=0$. Note the fact that
\begin{align}
\abs{\dt^2\rr(\pm\ell)} \ls \abs{\p_{\varpi_2}^2\rr (\pm\ell)}\abs{\dt\p_1\e(\pm\ell)}^2+\abs{\p_{\varpi_2}\rr (\pm\ell)}\abs{\dt^2\p_1\e(\pm\ell)}\ls\sqrt{\di}.
\end{align}

\textbf{First Contact Point Term, $\dt^2(b\tilde a\p_1\e)$:}
We can directly compute
\begin{align}
\dt^2(b\tilde a\p_1\e) = b\dt^2\tilde a\p_1\e+2b\dt\tilde a\dt\p_1\e+b\tilde a\dt^2\p_1\e.
\end{align}
Then we have
\begin{equation}
\abs{b\dt^2\tilde a\p_1\e\Big|_{\pm\ell}} \ls \bigg(\abs{\dt^3l}+\abs{\dt^3r}\bigg)\abs{\p_1\e(\pm\ell)}\ls\sqrt{\di}\sqrt{\en},
\end{equation}
\begin{equation}
\abs{2b\dt\tilde a\dt\p_1\e\Big|_{\pm\ell}} \ls \bigg(\abs{\dt^2l}+\abs{\dt^2r}\bigg)\abs{\dt\p_1\e(\pm\ell)}\ls\sqrt{\en}\sqrt{\di},
\end{equation}
and
\begin{equation}
\abs{b\tilde a\dt^2\p_1\e\Big|_{\pm\ell}} \ls \bigg(\abs{\dt l}+\abs{\dt r}\bigg)\abs{\dt^2\p_1\e(\pm\ell)}\ls\sqrt{\en}\sqrt{\di}.
\end{equation}

\textbf{Second Contact Point Term, $\dt u\cdot\dt\n$:}
Notice that $u(-\ell,0)=(\dt l,0)$ and $u(\ell,0)=(\dt r,0)$. We have
\begin{align}
\abs{\dt u\cdot\dt\n\Big|_{\pm\ell}} \ls \bigg(\abs{\dt^2l}+\abs{\dt^2r}\bigg)\abs{\dt\p_1\e(\pm\ell)}\ls\sqrt{\en}\sqrt{\di}.
\end{align}

\textbf{Third  Contact Point Term, $u\cdot\dt^2\n$:}
We have
\begin{align}
\abs{u\cdot\dt^2\n\Big|_{\pm\ell}} \ls \bigg(\abs{\dt l}+\abs{\dt r}\bigg)\abs{\dt^2\p_1\e(\pm\ell)}\ls\sqrt{\en}\sqrt{\di}.
\end{align}

\end{proof}

\subsection{Nonlinear estimates in the pressure estimates and free surface estimates}

\begin{lemma}\label{nonlinear lemma 9}
Define the functional $H^1(\Omega_0)\ni w\rt\bro{\s,w}\in\r$ via
\begin{align}
\bro{\s,w} = \int_{\Omega_0}Jw\cdot \s_1-\int_{\Sigma_0}w\cdot\s_3-\int_{\Sigma_{0b}}\s_4\bigg(w\cdot\frac{\t}{\abs{\t}^2}\bigg).
\end{align}
Then
\begin{align}
\abs{\bro{\s,w}}\ls \hm{w}{1}\sqrt{\en}\sqrt{\di}.
\end{align}
\end{lemma}
\begin{proof}
This is a summary of previous estimates for $\s_1$, $\s_3$ and $\s_4$.
\end{proof}

\begin{lemma}\label{nonlinear lemma 12}
Let $\rr$ be given by \eqref{ap rr}. We have the estimate
\begin{align}
\hms{\dt^2\rr}{\frac{1}{2}}\ls \sqrt{\en}\sqrt{\di}.
\end{align}
\end{lemma}
\begin{proof}
Similar to the proof of Lemma \ref{nonlinear lemma 11},
for $s-\dfrac{1}{2}>\dfrac{1}{2}$, we estimate
\begin{align}
\hms{\p_{\varpi_2}\rr \dt^2\p_1\e}{\frac{1}{2}}  & \ls   \Big(\abs{k_1}+\hms{\p_1\e}{s-\frac{1}{2}}\Big)\hms{\dt^2\p_1\e}{\frac{1}{2}} \\
&\ls \hms{\e}{s+\frac{1}{2}}\hms{\dt^2\e}{\frac{3}{2}}
\ls \sqrt{\en}\sqrt{\di} \no.
\end{align}
and
\begin{align}
\hms{\p_{\varpi_2}^2\rr (\dt\p_1\e)^2}{\frac{1}{2}}\ls& \hms{\dt\p_1\e}{s-\frac{1}{2}}\hms{\dt\p_1\e}{\frac{1}{2}}
\ls \hms{\dt\e}{s+\frac{1}{2}}\hms{\dt\e}{\frac{3}{2}}
\ls \sqrt{\di}\sqrt{\en}.
\end{align}
\end{proof}

\section{Nonlinear estimates in the Stokes problem}\label{sec:nlin_stokes}

We will now prove the nonlinear estimate in the Stokes problem when applying $\dt$ on both sides of the equation. Throughout this section, we always assume $\e$ is given and satisfies
\begin{align}
\sup_{0\leq t\leq T}\bigg(\enp(t)+\hmwss{\e(t)}{\frac{5}{2}}+\hmwss{\dt\e(t)}{\frac{3}{2}}\bigg)\leq\vartheta<1,
\end{align}
for some $\vartheta>0$ sufficiently small. Here, we will employ the same techniques as developed in \cite{guo_tice_QS} and Section \ref{nonlinear section}.

\subsection{Estimate of the $G_1$ term}

We now handle the term $G_1$.

\begin{lemma}\label{nonlinear elliptic 1}
Let $G_1 = \s_1$. We have the estimate
\begin{align}
\hmw{G_1}{0}^2 \ls \Big(\en^2+\en\Big)\di.
\end{align}
\end{lemma}
\begin{proof}
We will only present the estimate for the term $-\na_{\dt\a}\cdot\Big(pI-\mu\dm_{\a}u\Big)$.  The term $\mu\na_{\a}\cdot\dm_{\dt\a}u$ may be handled with a similar argument.  We begin by bounding
\begin{align}
\hmw{-\na_{\dt\a}\cdot\Big(pI-\mu\dm_{\a}u\Big)}{0}^2
\ls& \hmw{\dt\a \na p}{0}^2+\hmw{\dt\a \a \na^2u}{0}^2+\hmw{\dt\a\na\a\na u}{0}^2\\
=&: I+II+III.\no
\end{align}
For $I$, we have
\begin{align}
I=&\hmw{\dt\a \na p}{0}^2 \ls \lm{\dt\a}{\infty}^2\hmw{\na p}{0}^2 \ls \lm{\dt\na\be}{\infty}^2\hmw{\na p}{0}^2  \\
\ls& \hms{\dt\e}{s+\frac{1}{2}}^2\hmw{p}{1}^2 \ls \di\en.\no
\end{align}
For $II$, we have
\begin{align}
II =& \hmw{\dt\a \a \na^2u}{0}^2
\ls \lm{\dt\a}{\infty}^2\lm{\a}{\infty}^2\hmw{\na^2u}{0}^2  \\
\ls& \lm{\dt\na\be}{\infty}^2\hmw{\na^2u}{0}^2
 \ls\hms{\dt\e}{s+\frac{1}{2}}^2\hmw{u}{2}^2 \ls \di\en.\no
\end{align}
For $III$, we choose $q\in[1,\infty)$ such that $\dfrac{3}{q}+\dfrac{1}{2+2\d}=\dfrac{1}{2}$, and also $p\in[1,\infty)$ such that $\dfrac{2}{p}+\dfrac{2-s}{2}=\dfrac{1}{2}$. Then we have
\begin{align}
&III =\hmw{\dt\a\na\a\na u}{0}^2
\ls \hmw{\dt\na\be\na^2\be\na u}{0}^2+\hmw{\dt\na\be\na\be\frac{1}{\z_0}\na u}{0}^2  \\
\ls&  \lm{\dt\na\be}{p}^2\lm{\na^2\be}{\frac{2}{2-s}}^2\lm{d^{\d}\na u}{p}^2+\lm{\dt\na\be}{q}^2\lm{\na\be}{q}^2 \lm{\na u}{q}^2 \lm{\frac{d^{\d}}{\z_0}}{2+2\d} \no\\
\ls& \hm{\dt\be}{2}^2\hm{\be}{s+1}^2\hmw{\na u}{1}^2+\hm{\dt\be}{2}^2\hm{\be}{1}^2\hmw{\na u}{1}^2\no\\
\ls& \hms{\dt\e}{\frac{3}{2}}^2\hms{\e}{s+\frac{1}{2}}^2\hmw{u}{2}^2+\hms{\dt\e}{\frac{3}{2}}^2\hms{\e}{\frac{1}{2}}^2\hmw{u}{2}^2 \ls \di\en^2.\no
\end{align}

\end{proof}

\subsection{Estimate of the $G_2$ term}

Now we handle the term $G_2$.

\begin{lemma}\label{nonlinear elliptic 2}
Let $G_2 = \s_2$. We have the estimate
\begin{align}
\hmw{G_2}{1}^2\ls \en\di.
\end{align}
\end{lemma}
\begin{proof}
The estimate may be proved as in Proposition 7.2 of \cite{guo_tice_QS} with minor modifications to accommodate the $1/\zeta_0$ term as in the proof of Lemma \ref{nonlinear elliptic 1}.

\end{proof}

\subsection{Estimate of the $G_3$ term}

Next we handle $G_3$.

\begin{lemma}\label{nonlinear elliptic 3}
Let $G_3$ be given by  $G_3^- = 0$ and
\begin{equation}
G_3^+ =  \frac{\dt\Big(J_1\dt\e\Big)-\dt\Big(\tilde a\p_1\z\Big)}{\sqrt{1+\abs{\p\z_0}^2}}-u\cdot\dt\n.
\end{equation}
We have the estimate
\begin{align}
\hmws{G_3}{\frac{3}{2}}^2\ls \dip+\en\di.
\end{align}
\end{lemma}
\begin{proof}
We estimate
\begin{align}
\hmws{G_3}{\frac{3}{2}}^2 \ls& \hmwe{G_3^+}{\frac{3}{2}}^2 \ls  \hmwss{\dt\Big(J_1\dt\e\Big)}{\frac{3}{2}}^2+\hmwss{\dt\Big(\tilde a\p_1\z\Big)}{\frac{3}{2}}^2+\hmwe{u\cdot\dt\n}{\frac{3}{2}}^2\\
=&: I+II+III.\no
\end{align}
We then bound
\begin{align}
I=&\hmwss{\dt\Big(J_1\dt\e\Big)}{\frac{3}{2}}^2 \ls\hmwss{\dt J_1\dt\e}{\frac{3}{2}}^2+\hmwss{J_1\dt^2\e}{\frac{3}{2}}^2 \\
\ls& \abs{\dt J_1}^2\hmwss{\dt\e}{\frac{3}{2}}^2+\abs{J_1}^2\hmwss{\dt^2\e}{\frac{3}{2}}^2\no\\
\ls& \Big(\abs{\dt l}^2+\abs{\dt r}^2\Big)\hms{\dt\e}{\frac{3}{2}}+\hms{\dt^2\e}{\frac{3}{2}}\ls\en\di+\dip,\no
\end{align}
\begin{align}
II=&\hmwss{\dt\Big(\tilde a\p_1\z\Big)}{\frac{3}{2}}^2 \ls \Big(\abs{\dt^2l}^2+\abs{\dt^2r}^2\Big)\hmwss{\z}{\frac{3}{2}}^2+\Big(\abs{\dt l}^2+\abs{\dt r}^2\Big)\hmwss{\dt\e}{\frac{3}{2}}^2 \\
\ls& \Big(\abs{\dt^2l}^2+\abs{\dt^2r}^2\Big)+\Big(\abs{\dt l}^2+\abs{\dt r}^2\Big)\hmwss{\dt\e}{\frac{3}{2}}^2\ls\dip+\en\di,\no
\end{align}
and
\begin{align}
III=& \hmwe{u\cdot\dt\n}{\frac{3}{2}}^2 \ls \hmwe{u\cdot\dt\n}{1}+\hmwe{\p_1\Big(u\cdot\dt\n\Big)}{\frac{1}{2}}  \\
\ls& \hmwe{u}{\frac{3}{2}}^2\hmwss{\dt\p_1\e}{\frac{3}{2}}^2+\hmwe{\p_1u}{\frac{1}{2}}^2\hmwss{\dt\p_1\e}{\frac{3}{2}}^2
+\hmwe{u}{\frac{3}{2}}^2\hmwss{\dt\p_1^2\e}{\frac{1}{2}}^2\no\\
\ls& \hmw{u}{2}^2\hmwss{\dt\e}{\frac{5}{2}}^2\ls\en\di.\no
\end{align}
\end{proof}

\subsection{Estimate of the $G_4$ term}

Next to estimate is $G_4$.

\begin{lemma}\label{nonlinear elliptic 4}
Let $G_4= \s_4$ be given by $G_4^- = \s_4$ and $G_4^+ = \s_3\cdot\dfrac{\t}{\abs{\t}}$.   We have the estimate
\begin{align}
\hmws{G_4}{\frac{1}{2}}^2\ls \Big(\en^2+\en\Big)\di.
\end{align}
\end{lemma}
\begin{proof}
It is easy to check that
\begin{align}
\hmws{G_4}{\frac{1}{2}}^2\ls\hmws{G_4^-}{\frac{1}{2}}^2+\hmws{G_4^+}{\frac{1}{2}}^2\ls\hmws{\s_4}{\frac{1}{2}}^2+\hmws{\s_3}{\frac{1}{2}}^2.
\end{align}
We will proceed by estimating these term by term.

\textbf{$G_4^-$ Term  $\mu\dm_{\dt\a}u\nu\cdot\tau$:}
We have
\begin{align}
\hmwb{\mu\dm_{\dt\a}u\nu\cdot\tau}{\frac{1}{2}}^2\ls& \hmw{\dm_{\dt\a}u e_2\cdot e_1}{1}^2\\
\ls& \hmw{\dt\a \na u}{0}^2+\hmw{\dt\a \na^2 u}{0}^2+\hmw{\dt\na\a \na u}{0}^2 =: I+II+III.\no
\end{align}
For $I$, we have
\begin{align}
I=& \hmw{\dt\a \na u}{0}^2 \ls \lm{\dt\a}{4}^2\lm{d^{\d}\na u}{4}^2 \ls \lm{\dt\na\be}{4}^2\lm{d^{\d}\na u}{4}^2 \\
\ls& \hms{\dt\e}{\frac{3}{2}}^2\hmw{u}{2}^2 \ls \di\en.\no
\end{align}
For $II$, we have
\begin{align}
II =& \hmw{\dt\a \na^2 u}{0}^2 \ls
\lm{\dt\a}{\infty}^2\hmw{\na^2u}{0}^2
 \ls \lm{\dt\na\be}{\infty}^2\hmw{\na^2u}{0}^2 \\
\ls& \hms{\dt\e}{s+\frac{1}{2}}^2\hmw{u}{2}^2
\ls \di\en.\no
\end{align}
For $III$, we choose $q\in[1,\infty)$ such that $\dfrac{2}{q}+\dfrac{1}{2+2\d}=\dfrac{1}{2}$. Then we have
\begin{align}
III=& \hmw{\dt\na\a \na u}{0}^2
\ls \hmw{\dt\na^2\be\na u}{0}^2+\hmw{\dt\na\be\frac{1}{\z_0}\na u}{0}^2  \\
\ls& \lm{\dt\na^2\be}{\frac{2}{2-s}}^2\lm{d^{\d}\na u}{\frac{2}{s-1}}^2+\lm{\dt\na\be}{q}^2\lm{\frac{d^{\d}}{\z_0}}{2+2\d}\lm{\na u}{q}^2 \no \\
\ls& \hm{\dt\be}{s+1}^2\hmw{\na u}{1}^2+\hm{\dt\be}{2}^2\hmw{\na u}{1}^2\no\\
\ls& \hms{\dt\e}{s+\frac{1}{2}}^2\hmw{u}{2}^2+\hms{\dt\e}{\frac{3}{2}}^2\hmw{u}{2}^2\ls \di\en.\no
\end{align}

\textbf{First Term of $G_4^+$, $\mu\dm_{\dt\a}u\n$:}
We have
\begin{align}
& \hmwe{\mu\dm_{\dt\a}u\n}{\frac{1}{2}}^2  \ls \hmw{\dm_{\dt\a}u (J_1e_2+e_1\p_1\be )}{1}^2 \\
\ls& \hmw{\dt\a \na u}{0}^2+\hmw{\dt\a \na^2 u}{0}^2+\hmw{\dt\na\a \na u}{0}^2+\hmw{\dt\a \na u\p_1^2\be}{0}^2\no\\
=&: I+II+III+IV.\no
\end{align}
For $I$, we have
\begin{align}
I=&\hmw{\dt\a \na u}{0}^2
\ls \lm{\dt\a}{4}^2\lm{d^{\d}\na u}{4}^2
\ls \lm{\dt\na\be}{4}^2\lm{d^{\d}\na u}{4}^2 \\
\ls& \hms{\dt\e}{\frac{3}{2}}^2\hmw{u}{2}^2 \ls \di\en.\no
\end{align}
For $II$, we have
\begin{align}
II=&\hmw{\dt\a \na^2 u}{0}^2 \ls \lm{\dt\a}{\infty}^2\hmw{\na^2u}{0}^2 \ls \lm{\dt\na\be}{\infty}^2\hmw{\na^2u}{0}^2 \\
\ls& \hms{\dt\e}{s+\frac{1}{2}}^2\hmw{u}{2}^2 \ls \di\en.\no
\end{align}
For $III$, we choose $q\in[1,\infty)$ such that $\dfrac{2}{q}+\dfrac{1}{2+2\d}=\dfrac{1}{2}$. Then we have
\begin{align}
III=& \hmw{\dt\na\a \na u}{0}^2 \ls \hmw{\dt\na^2\be\na u}{0}^2+\hmw{\dt\na\be\frac{1}{\z_0}\na u}{0}^2  \\
\ls& \lm{\dt\na^2\be}{\frac{2}{2-s}}^2\lm{d^{\d}\na u}{\frac{2}{s-1}}^2+\lm{\dt\na\be}{q}^2\lm{\frac{d^{\d}}{\z_0}}{2+2\d}\lm{\na u}{q}^2\no\\
\ls& \hm{\dt\be}{s+1}^2\hmw{\na u}{1}^2+\hm{\dt\be}{2}^2\hmw{\na u}{1}^2\no\\
\ls& \hms{\dt\e}{s+\frac{1}{2}}^2\hmw{u}{2}^2+\hms{\dt\e}{\frac{3}{2}}^2\hmw{u}{2}^2\ls \di\en.\no
\end{align}
For $IV$, we have
\begin{align}
& IV = \hmw{\dt\a \na u\p_1^2\be}{0}^2 \ls \lm{\dt\a}{\infty}^2\lm{\na u}{4}^2\lm{\p_1^2\be}{4}^2\no\\
\ls& \lm{\dt\na\be}{\infty}^2\hmw{\na u}{1}^2\hmw{\p_1^2\be}{1}
\ls \hms{\dt\e}{s+\frac{1}{2}}^2\hmw{u}{2}^2\hmwss{\e}{\frac{5}{2}} \ls \di\en^2.\no
\end{align}

\textbf{Second Term of $G_4^+$, $-(pI-\mu\dm_{\a} u)\dt\n$:}
We have
\begin{align}
&\hmwe{-(pI-\mu\dm_{\a} u)\dt\n}{\frac{1}{2}}^2  \ls \hmw{(pI-\mu\dm_{\a} u)\dt(J_1e_2+e_1\p_1\be)}{1}^2 \\
\ls& \hmw{(p +\na u)\dt\p_1\be}{0}^2+\hmw{\na^2 u\dt\p_1\be}{0}^2+\hmw{\na\a \na u\dt\p_1\be}{0}^2+\hmw{(p +\na u)\dt\p_1^2\be}{0}^2\no\\
=&: I+II+III+IV.\no
\end{align}
For $I$, we have
\begin{align}
I=& \hmw{(p +\na u)\dt\p_1\be}{0}^2 \ls \lm{d^{\d}(p +\na u)}{4}^2\lm{\dt\na\be}{4}^2 \\
\ls& \lm{d^{\d}(p +\na u)}{4}^2\lm{\dt\na\be}{4}^2
\ls \Big(\hmw{u}{2}^2+\hmw{p}{1}^2\Big)\hms{\dt\e}{\frac{3}{2}}^2 \ls \en\di.\no
\end{align}
For $II$, we have
\begin{align}
II=& \hmw{\na^2 u\dt\p_1\be}{0}^2 \ls \hmw{\na^2u}{0}^2\lm{\dt\na\be}{\infty}^2  \ls \hmw{\na^2u}{0}^2\lm{\dt\na\be}{\infty}^2 \\
\ls& \hmw{u}{2}^2\hms{\dt\e}{s+\frac{1}{2}}^2 \ls \en\di.\no
\end{align}
For $III$, we choose $q\in[1,\infty)$ such that $\dfrac{3}{q}+\dfrac{1}{2+2\d}=\dfrac{1}{2}$. Then we have
\begin{align}
& III= \hmw{\na\a \na u\dt\p_1\be}{0}^2
\ls \hmw{\na^2\be\na u\dt\p_1\be}{0}^2+\hmw{\na\be\frac{1}{\z_0}\na u\dt\p_1\be}{0}^2 \\
\ls& \lm{\na^2\be}{3}^2\lm{d^{\d}\na u}{3}^2\lm{\dt\p_1\be}{3}^2+\lm{\na\be}{q}^2\lm{\frac{d^{\d}}{\z_0}}{2+2\d}\lm{\na u}{q}^2\lm{\dt\p_1\be}{q}^2\no\\
\ls& \hm{\be}{3}^2\hmw{\na u}{1}^2\hm{\dt\p_1\be}{1}^2+\hm{\be}{2}^2\hmw{\na u}{1}^2\hm{\dt\p_1\be}{1}^2\no\\
\ls& \hms{\e}{\frac{5}{2}}^2\hmw{u}{2}^2\hms{\dt\e}{\frac{3}{2}}^2+\hms{\e}{\frac{3}{2}}^2\hmw{u}{2}^2\hms{\dt\e}{\frac{3}{2}}^2\ls \di\en^2.\no
\end{align}
For $IV$, we have
\begin{align}
&IV= \hmw{(p +\na u)\dt\p_1^2\be}{0}^2 \ls \lm{p +\na u}{4}^2\lm{\dt\p_1^2\be}{4}^2 \\
\ls& \Big(\hmw{\na u}{1}^2+\hmw{p}{1}^2\Big)\hmw{\dt\p_1^2\be}{1}
 \ls \Big(\hmw{u}{2}^2+\hmw{p}{1}^2\Big)\hmwss{\dt\e}{\frac{5}{2}} \ls \en\di.\no
\end{align}

\textbf{Third Term of $G_4^+$, $g\e\dt\n$:}
We estimate
\begin{align}
\hmwe{g\e\dt\n}{\frac{1}{2}}^2
\ls \hmwss{\e}{\frac{1}{2}}^2\hmwe{\dt\p_1\be}{s-\frac{1}{2}}^2
\ls\hmwss{\e}{\frac{1}{2}}^2\hmwss{\dt\e}{s}^2
\ls\en\di.\no
\end{align}

\textbf{Fourth Term of $G_4^+$, $\sigma \p_{1}\left(\dfrac{k_1\p_{1}\z_0}{\sqrt{1+\abs{\p_{1}\z_0}^2}}+\dfrac{k_1\p_{1}\z_0+\p_1\e}{\Big(\sqrt{1+\abs{\p_{1}\z_0}^2}\Big)^3}\right)\dt\n$:}
We estimate
\begin{align}
&\hmwe{\sigma \p_{1}\left(\dfrac{k_1\p_{1}\z_0}{\sqrt{1+\abs{\p_{1}\z_0}^2}}+\dfrac{k_1\p_{1}\z_0+\p_1\e}{\Big(\sqrt{1+\abs{\p_{1}\z_0}^2}\Big)^3}\right)\dt\n}{\frac{1}{2}}^2\\
\ls& \hmwss{\p_1^2\e}{\frac{1}{2}}^2\hmwe{\dt\p_1\be}{s-\frac{1}{2}}^2 \ls \hmwss{\e}{\frac{5}{2}}^2\hmwss{\dt\e}{s}^2 \ls \en\di.\no
\end{align}

\textbf{Fifth Term of $G_4^+$, $\sigma \p_{1}\rr\dt\n$:}
We estimate
\begin{align}
\hmwe{\sigma \p_{1}\rr\dt\n}{\frac{1}{2}}^2 \ls& \hmwss{\p_1\e}{1}^2\hmwss{\p_1^2\e}{\frac{1}{2}}^2\hmwe{\dt\p_1\be}{s-\frac{1}{2}}^2
\ls \hmwss{\e}{\frac{5}{2}}^2\hmwss{\dt\e}{s}^2 \ls\en\di.
\end{align}
\end{proof}

\subsection{Estimate of the $G_5$ term}

Finally, we handle $G_5$.

\begin{lemma}\label{nonlinear elliptic 5}
Let $G_5= \s_3\cdot\dfrac{\n}{\abs{\n}}$. We have the estimate
\begin{align}
\hmws{G_5}{\frac{1}{2}}^2\ls \Big(\en^2+\en\Big)\di.
\end{align}
\end{lemma}
\begin{proof}
This estimate follows from an argument similar to that of Lemma \ref{nonlinear elliptic 2}.
\end{proof}

\subsection{Other nonlinear estimates}

Here we record some other nonlinear estimates.

\begin{lemma}\label{nonlinear elliptic 6}
Let $\rr$ be given by \eqref{ap rr}. We have the estimate
\begin{align}
\hmwss{\p_1\dt\rr}{\frac{1}{2}}^2\ls \en\di.
\end{align}
\end{lemma}
\begin{proof}
We have
\begin{align}
 \hmwss{\p_1\dt\rr}{\frac{1}{2}}^2 \ls& \hmwss{\dt k_1 \p_1^2\e}{\frac{1}{2}}^2+\hmwss{k_1 \dt \p_1^2\e}{\frac{1}{2}}^2  \\
& +\hmwss{\p_1^2\e \dt\p_1\e}{\frac{1}{2}}^2+\hmwss{\p_1\e \dt\p_1^2\e}{\frac{1}{2}}^2
=: I+II+III+IV.\no
\end{align}
We estimate each term as follows:
\begin{equation}
I \ls \abs{\dt k_1}^2\hmwss{\p_1^2\e}{\frac{1}{2}}^2\ls\Big(\abs{\dt l}^2+\abs{\dt r}^2\Big)\hmwss{\e}{\frac{5}{2}}^2\ls\en\di,
\end{equation}
\begin{equation}
II \ls \abs{k_1}^2\hmwss{\dt \p_1^2\e}{\frac{1}{2}}^2\ls\hms{\e}{1}^2\hmwss{\dt \e}{\frac{5}{2}}^2\ls\en\di,
\end{equation}
\begin{equation}
III  \ls \hmwss{\p_1^2\e}{\frac{1}{2}}^2\hmwss{\dt\p_1\e}{s-\frac{1}{2}}^2\ls\hmwss{\e}{\frac{5}{2}}^2\hmwss{\dt\e}{s+\frac{1}{2}}^2\ls\en\di,
\end{equation}
and
\begin{equation}
IV \ls \hmwss{\p_1\e}{s-\frac{1}{2}}^2\hmwss{\dt\p_1^2\e}{\frac{1}{2}}^2\ls\hmwss{\e}{s+\frac{1}{2}}^2\hmwss{\dt\e}{\frac{5}{2}}^2\ls\en\di.
\end{equation}

\end{proof}

\section{Well-posedness and decay}\label{sec:gwp_dec}

Define the functionals
\begin{equation}
\sen= \sum_{j=0}^2\int_{-\ell}^{\ell}\frac{\sigma}{2}\dfrac{\Big(\dt^jk_1\p_{1}\z_0+\dt^j\p_1\e\Big)^2}{\Big(\sqrt{1+\abs{\p_{1}\z_0}^2}\Big)^3},
\end{equation}
\begin{equation}
\sdi=\sum_{j=0}^2\left(\int_{\Omega_0}\frac{J\mu}{2}\abs{\dm_{\a}\dt^ju}^2+\int_{\Sigma_{0b}}\beta(\dt^ju\cdot\tau_0)^2+\kappa\Big((\dt^{j+1}\lp)^2+(\dt^{j+1}\rp)^2\Big)\right),
\end{equation}
and
\begin{equation}
\sf= \sf_1+\sf_2=-\int_{-\ell}^{\ell}\frac{J_1-1}{\Big(\sqrt{1+\abs{\p_{1}\z_0}^2}\Big)^3}\abs{\dt^2\p_1\e}^2
+\int_{-\ell}^{\ell}\frac{\p_{\varpi_2}\rr J_1}{\sqrt{1+\abs{\p_{1}\z_0}^2}}\abs{\dt^2\p_1\e}^2.
\end{equation}

Our next result shows various comparability results for these functionals.

\begin{lemma}\label{main lemma 1}
There exists a universal constant $\vartheta>0$ such that if
\begin{align}
\sup_{0\leq t\leq T}\en(t)\leq\vartheta,
\end{align}
then
\begin{align}\label{main 1}
\sen\ls\enp\ls\sen, \;  \sdi\ls\dipt\ls\sdi, \text{ and }
\abs{\sf}\leq\frac{1}{2}\sen.
\end{align}
\end{lemma}
\begin{proof}
The first two inequalities in \eqref{main 1} follow directly from Lemma \ref{appendix lemma 1}. It remains to prove the estimate of $\sf$. Lemma \ref{estimate_r} implies that
\begin{align}
\abs{\sf_1}=&\abs{\int_{-\ell}^{\ell}\frac{J_1-1}{\Big(\sqrt{1+\abs{\p_{1}\z_0}^2}\Big)^3}\abs{\dt^2\p_1\e}^2}
\ls\hms{\e}{0}\hms{\dt^2\e}{1}^2\ls\sqrt{\en}\sen,\\
\abs{\sf_2}=&\abs{\int_{-\ell}^{\ell}\frac{\p_{\varpi_2}\rr J_1}{\sqrt{1+\abs{\p_{1}\z_0}^2}}\abs{\dt^2\p_1\e}^2}\ls\abs{\p_{\varpi_2}\rr }\hms{\dt^2\p_1\e}{0}^2
\ls\hms{\e}{2}\hms{\dt^2\e}{1}^2\ls\sqrt{\en}\sen.
\end{align}
Hence, for $\vartheta$ small, the desired estimate follows directly.
\end{proof}

Next we state a synthesized version of the energy-dissipation equation.

\begin{lemma}\label{main lemma 2}
There exists a universal constant $\vartheta>0$ such that if
\begin{align}
\sup_{0\leq t\leq T}\en(t)+\int_0^T\di(t)\ud{t}\leq\vartheta,
\end{align}
then there exists a universal constant $C>0$ such that
\begin{align}
\frac{\ud}{\ud t}(\sen-\sf)+C\di\leq 0.
\end{align}
\end{lemma}
\begin{proof}
Let $\vartheta$ be as in Lemma \ref{main lemma 1}. The linearized energy-dissipation structure \eqref{linearize 5} for $\dt^2$, $\dt$ and no time derivatives, combined with the estimates in Lemmas \ref{nonlinear lemma 1} -- \ref{nonlinear lemma 11}, imply that
\begin{align}
\frac{\ud}{\ud t}(\sen-\sf)+C\sdi\leq\sqrt{\en}\di.
\end{align}
In Lemma \ref{main lemma 1}, we have shown
\begin{align}
\sdi\ls\dipt\ls\sdi,
\end{align}
Theorem \ref{pressure theorem} and Lemma \ref{nonlinear lemma 9} imply the pressure estimate
\begin{align}
\hm{p}{0}^2+\hm{\dt p}{0}^2+\hm{\dt^2 p}{0}^2\ls\dipt+\sqrt{\en}\di.
\end{align}
Theorem \ref{surface lemma 3} and Lemmas \ref{nonlinear lemma 9} and \ref{nonlinear lemma 12}, imply the free surface estimate
\begin{align}
\hms{\e}{\frac{3}{2}}^2 &+ \hms{\dt\e}{\frac{3}{2}}^2+\hms{\dt^2\e}{\frac{3}{2}}^2 \\
\ls&\hm{p}{0}^2+\hm{\dt p}{0}^2+\hm{\dt^2 p}{0}^2+ \dipt+\sqrt{\en}\di
\ls \dipt+\sqrt{\en}\di.\no
\end{align}
Theorems \ref{contact point theorem 1} and \ref{contact point theorem 2} imply the contact point estimate
\begin{align}
\sum_{j=0}^2\bigg(\abs{\dt^{j}\p_1\e(-\ell)}^2+\abs{\dt^{j}\p_1\e(\ell)}^2\bigg)+\sum_{j=0}^2\bigg(\abs{\dt^{j}u(-\ell,0)\cdot\n}^2+\abs{\dt^{j}u(\ell,0)\cdot\n}^2\bigg)
\ls\dipt+\sqrt{\en}\di.
\end{align}
Combining these, we deduce that
\begin{align}
\dip\ls\dipt+\sqrt{\en}\di\ls\sdi+\sqrt{\en}\di.
\end{align}
The Stokes problem estimate in Theorem \ref{elliptic theorem} for at most $\dt$, combined with the estimates in Lemma \ref{nonlinear elliptic 1}, \ref{nonlinear elliptic 2}, \ref{nonlinear elliptic 3}, \ref{nonlinear elliptic 4}, \ref{nonlinear elliptic 5}, and \ref{nonlinear elliptic 6}, imply
\begin{align}
\hmw{u}{2}^2+\hmw{p}{1}^2+\hmwss{\e}{\frac{5}{2}}^2+\hmw{\dt u}{2}^2+\hmw{\dt p}{1}^2+\hmwss{\dt \e}{\frac{5}{2}}^2\ls\dip+(\en^2+\en)\di.
\end{align}
Using the transport equation in \eqref{system 3}, we may further obtain
\begin{align}
\hmwss{\dt^2\e}{\frac{3}{2}}^2+\hmwss{\dt^3\e}{\frac{1}{2}}^2\ls \dip+(\en^2+\en)\di.
\end{align}
Collecting all above, we have
\begin{align}
\di\ls\dip+\sqrt{\en}\di\ls\dipt+\sqrt{\en}\di\ls \sdi+\sqrt{\en}\di.
\end{align}
Then we have
\begin{align}
\frac{\ud}{\ud t}(\sen-\sf)+C\di\leq\sqrt{\en}\di.
\end{align}
Thus, if $\vartheta$ is sufficiently small, then we may conclude that
\begin{align}
\frac{\ud}{\ud t}(\sen-\sf)+C\di\leq0.
\end{align}
\end{proof}

We now present the main a priori decay estimates.

\begin{theorem}\label{main theorem 1}
There exists a universal constant $\vartheta>0$ such that if
\begin{align}
\sup_{0\leq t\leq T}\en(t)+\int_0^T\di(t)\ud{t}\leq\vartheta,
\end{align}
then there exists a universal constant $\lambda>0$ such that
\begin{align}
\sup_{0\leq t\leq T}\ue^{\lambda t}\bigg(\enp(t)+\hm{u(t)}{1}^2+\nm{u(t)\cdot\tau_0}_{H^0(\Sigma_{0b})}^2+\hm{p(t)}{0}^2
+\abs{\dt l(t)}^2+\abs{\dt r(t)^2}\\
+\abs{\p_1\e(t,-\ell)}^2+\abs{\p_1\e(t,\ell)}^2+\abs{u(t,-\ell,0)\cdot\n}^2+\abs{u(t,\ell,0)\cdot\n}^2\bigg)\ls&\enp(0).\no
\end{align}
\end{theorem}
\begin{proof}
In Lemma \ref{main lemma 2}, we have shown that
\begin{align}
\frac{\ud}{\ud t}(\sen-\sf)+C\di\leq0.
\end{align}
Also, in Lemma \ref{main lemma 1}, we have proved that
\begin{align}\label{main theorem 1 eq 1}
\sen\ls\enp\ls\sen\ \text{and}\ 0<\frac{1}{2}\sen\leq\sen-\sf\leq\frac{3}{2}\sen.
\end{align}
On the other hand, it is clear that $\sen\ls\di$, and so we deduce the bound
\begin{align}
\frac{\ud}{\ud t}(\sen-\sf)+\lambda(\sen-\sf)\leq0.
\end{align}
Upon integrating this differential inequality, we find that
\begin{align}
\sen(t)\ls\sen(t)-\sf(t)\ls \ue^{-\lambda t}\Big(\sen(0)-\sf(0)\Big)\ls\ue^{-\lambda t}\sen(0),
\end{align}
which, in light of \eqref{main theorem 1 eq 1}, then implies that
\begin{align}
\enp(t)\ls\ue^{-\lambda t}\enp(0).
\end{align}

Next we consider the linearized energy-dissipation structure given in Theorem \ref{linear theorem 2} for the problem with no derivatives applied. Using the transport equation in \eqref{system 2}, and following similar arguments as Lemma \ref{nonlinear lemma 11}, \ref{nonlinear lemma 5}, \ref{nonlinear lemma 6}, \ref{nonlinear lemma 7}, \ref{nonlinear lemma 8}, noting $\e(\pm\ell)=0$, we obtain the bounds
\begin{align}
\abs{\sigma\int_{\Sigma_0}\p_1\rr(u\cdot\n)}\ls&\sqrt{\en}\bigg(\hms{\e}{1}^2+\hms{\dt\e}{1}^2+\abs{\dt l}^2+\abs{\dt r}^2\bigg)\\
\ls& \sqrt{\en}\enp+\sqrt{\en}\bigg(\abs{\dt l}^2+\abs{\dt r}^2\bigg),\no
\end{align}
\begin{align}
&\abs{\int_{-\ell}^{\ell}\left(g\e-\p_1\left(\dfrac{k_1\p_{1}\z_0}{\sqrt{1+\abs{\p_{1}\z_0}^2}}
+\dfrac{k_1\p_{1}\z_0+\p_1\e}{\Big(\sqrt{1+\abs{\p_{1}\z_0}^2}\Big)^3}\right)\right)\ss}\\
\ls&\sqrt{\en}\bigg(\hms{\e}{1}^2+\hms{\dt\e}{1}^2+\abs{\dt l}^2+\abs{\dt r}^2\bigg)\ls \sqrt{\en}\enp+\sqrt{\en}\bigg(\abs{\dt l}^2+\abs{\dt r}^2\bigg),\no
\end{align}
and
\begin{equation}
\abs{-\sigma\qq\Big(\dt l+\oo\Big)\bigg|_{\ell}+\sigma\qq\Big(\dt r+\oo\Big)\bigg|_{-\ell}}\ls\sqrt{\en}\bigg(\abs{\dt l}^2+\abs{\dt r}^2\bigg),
\end{equation}
\begin{equation}
\abs{\kappa\bigg(\dt r\oo\bigg|_{\ell}+\dt l\oo\bigg|_{-\ell}\bigg)}\ls\sqrt{\en}\bigg(\abs{\dt l}^2+\abs{\dt r}^2\bigg),
\end{equation}
\begin{equation}
\abs{-\s_7\Big(\dt r+\oo\Big)\bigg|_{\ell}+\s_6\Big(\dt l+\oo\Big)\bigg|_{-\ell}}\ls\sqrt{\en}\bigg(\abs{\dt l}^2+\abs{\dt r}^2\bigg).
\end{equation}

Therefore,
\begin{align}
&\dt\left(\int_{-\ell}^{\ell}\frac{\sigma}{2}\dfrac{\Big(k_1\p_{1}\z_0
+\p_1\e\Big)^2}{\Big(\sqrt{1+\abs{\p_{1}\z_0}^2}\Big)^3}+\frac{k_1^2}{2}\left(P_0M+\sigma\int_{-\ell}^{\ell}\dfrac{\abs{\p_{1}\z_0}^2}{\sqrt{1+\abs{\p_{1}\z_0}^2}}\right)
+\frac{g}{2}\int_{-\ell}^{\ell}\left(k_1\z_0-\e\right)^2\right)\\
\ls&\hms{\e}{1}^2+\hms{\dt\e}{1}^2\ls\enp,\no
\end{align}
and hence Theorem \ref{linear theorem 2} allows us to bound
\begin{align}
\int_{\Omega_0}\frac{J\mu}{2}\abs{\dm_{\a} u}^2+\int_{\Sigma_{0b}}\beta(u\cdot\tau_0)^2+\kappa\Big((\dt l)^2+(\dt r)^2\Big)\ls \enp+\sqrt{\en}\bigg(\abs{\dt l}^2+\abs{\dt r}^2\bigg),
\end{align}
which in turn means that for $\en$ small,
\begin{align}
\int_{\Omega_0}\frac{J\mu}{2}\abs{\dm_{\a} u}^2+\int_{\Sigma_{0b}}\beta(u\cdot\tau_0)^2+\kappa\Big((\dt l)^2+(\dt r)^2\Big)\ls \enp.
\end{align}
Using the pressure estimate in Theorem \ref{pressure theorem}, we know
\begin{align}
\hm{p(t)}{0}^2\ls\hm{u(t)}{1}^2\ls\enp.
\end{align}
Using the contact point estimates in Theorem \ref{contact point theorem 1} and \ref{contact point theorem 2}, we know
\begin{align}
\abs{\p_1\e(t,-\ell)}^2+\abs{\p_1\e(t,\ell)}^2\ls&  \hms{\e}{0}^2+\abs{\dt l}^2+\abs{\dt r}^2\ls\enp,\\
\abs{u(t,-\ell,0)\cdot\n}^2+\abs{u(t,\ell,0)\cdot\n}^2\ls&  \hms{\e}{0}^2+\abs{\dt l}^2+\abs{\dt r}^2\ls\enp.
\end{align}
Combining the above, we conclude that the stated estimates hold.
\end{proof}

Next we present the a priori bounds at the higher level of regularity.

\begin{theorem}\label{main theorem 2}
There exists a universal constant $\vartheta>0$ such that if
\begin{align}
\sup_{0\leq t\leq T}\en(t)+\int_0^T\di(t)\ud{t}\leq\vartheta,
\end{align}
then
\begin{align}\label{main theorem 2 eq 0}
\sup_{0\leq t\leq T}\en(t)+\int_0^T\di(t)\ud{t}\ls\en(0).
\end{align}
\end{theorem}
\begin{proof}
Again, we know from Lemmas \ref{main lemma 1} and \ref{main lemma 2} that, provided $\vartheta$ is sufficiently small, we have the bounds
\begin{align}
\frac{\ud}{\ud t}(\sen-\sf)+C\di\leq 0, \;
\sen\ls\enp\ls\sen, \text{ and }  0<\frac{1}{2}\sen\leq\sen-\sf\leq\frac{3}{2}\sen.
\end{align}
This allows us to integrate in time to deduce that
\begin{align}
\frac{1}{2}\sen(t)+C\int_0^t\di(s)\ud{s}\leq \frac{3}{2}\sen(0),
\end{align}
which implies for any $t\in[0,T]$,
\begin{align}\label{main 3}
\enp(t)+\int_0^t\di(s)\ud{s}\ls\enp(0).
\end{align}
For a Hilbert space $X$ and $f\in H^1([0,T];X)$, we know that
\begin{align}
\nm{f(t)}_X^2 \ls \nm{f(0)}_X^2 +\int_0^t\Big(\nm{f(s)}_X^2+\nm{\dt f(s)}_X^2\Big)\ud{s}.
\end{align}
Hence, we have
\begin{align}\label{main 4}
&\hmw{u(t)}{2}^2+\hmw{\dt u(t)}{2}^2+\hmw{p(t)}{1}^2+\hmw{\dt p(t)}{1}^2+\hmwss{\e(t)}{\frac{5}{2}}^2+\hms{\dt\e(t)}{\frac{3}{2}}^2\\
&+\sum_{j=0}^1\bigg(\abs{\dt^{j}\p_1\e(t)\Big|_{-\ell}}^2+\abs{\dt^{j}\p_1\e(t)\Big|_{\ell}}^2\bigg)
+\sum_{j=0}^1\bigg(\abs{\dt^{j}u(t)\cdot\n\Big|_{-\ell}}^2+\abs{\dt^{j}u(t)\cdot\n\Big|_{\ell}}^2\bigg)\no\\
\ls&\int_0^t\di(s)\ud{s}+\hmw{u(0)}{2}^2+\hmw{\dt u(0)}{2}^2 \no \\
& +\hmw{p(0)}{1}^2+\hmw{\dt p(0)}{1}^2+\hmwss{\e(0)}{\frac{5}{2}}^2+\hms{\dt\e(0)}{\frac{3}{2}}^2\no\\
&+\sum_{j=0}^1\bigg(\abs{\dt^{j}\p_1\e(0)\Big|_{-\ell}}^2+\abs{\dt^{j}\p_1\e(0)\Big|_{\ell}}^2\bigg)
+\sum_{j=0}^1\bigg(\abs{\dt^{j}u(0)\cdot\n\Big|_{-\ell}}^2+\abs{\dt^{j}u(0)\cdot\n\Big|_{\ell}}^2\bigg).\no
\end{align}
Then \eqref{main 3} and \eqref{main 4} may be combined to conclude that \eqref{main theorem 2 eq 0} holds.
\end{proof}

Now we record the local well-posedness result without giving detailed proof. It can be done using a variant of the argument developed in \cite{tice_zheng}.

\begin{theorem}\label{local theorem}
There exists a universal constant $\vartheta>0$ and $T_0>0$ such that if $0<T<T_0$ and $\en(0)\leq\vartheta,$
then there exists a unique solution $(u,p,\e)$ on the interval $t\in[0,T]$ such that
\begin{align}
\sup_{t\in[0,T]}\en(t)+\int_0^{T}\di(t)\ud{t}\ls&\en(0).
\end{align}
\end{theorem}

Now we provide the global well-posedness result.

\begin{proof}[Proof of Theorem \ref{intro main}]
This follows from the local well-posedness, Theorem \ref{local theorem}, and a standard continuation argument using the a priori estimates of Theorem \ref{main theorem 1} and \ref{main theorem 2}.  For details we refer to \cite{guo_tice_QS}.
\end{proof}

\appendix

\makeatletter
\renewcommand \theequation {%
A.%
\ifnum\c@subsection>\z@\@arabic\c@subsection.%
\fi\@arabic\c@equation} \@addtoreset{equation}{section}
\@addtoreset{equation}{subsection} \makeatother

\section{Analysis tools}

The proofs in this section can be found in Appendix C and Appendix D in \cite{guo_tice_QS}.

\subsection{Weighted sobolev spaces}

Let $M=dist(\cdot, M)$ where $M=\Big\{(-\ell,0), (\ell,0)\Big\}$ is the set of the corner points. Define the weighted Sobolev norm
\begin{align}
\hmw{f}{k}^2=\sum_{\abs{\alpha}\leq k}\int_{\Omega_0}\Big(dist(x,M)\Big)^{2\d}\abs{\p^{\alpha}f}^2\ud{x}.
\end{align}
Then we say $f\in\wwd{k}{\Omega_0}$ if and only if $\hmw{f}<\infty$. We will define the trace space $\hmws{f}{k-\frac{1}{2}}$ in the obvious way and it can be shown that
\begin{align}
\int_{\p\Omega_0}f(v\cdot\tau_0)\ls \hmws{f}{\frac{1}{2}}\hm{v}{1}.
\end{align}
Finally, define the zero-average space
\begin{align}
\mathring{W}_{\d}^k(\Omega_0)=\left\{f\in\wwd{k}{\Omega_0}: \int_{\Omega_0}f(x)\ud{x}=0\right\}.
\end{align}

\begin{lemma}
We have the continuous embedding
\begin{align}
\wwd{1}{\Omega_0}\hookrightarrow H^0(\Omega_0),\quad \wwd{2}{\Omega_0}\hookrightarrow H^1(\Omega_0),\quad H^{-1}(\Omega_0)\hookrightarrow W_{-\d}^0{(\Omega_0)}.
\end{align}
\end{lemma}

\begin{lemma}
Let $k\in\mathbb{N}$ and $\d_1,\d_2\in\r$ with $\d_1<\d_2$. Then we have that
\begin{align}
W_{\d_1}^k{(\Omega_0)}\hookrightarrow W_{\d_2}^k{(\Omega_0)}.
\end{align}
\end{lemma}

\begin{lemma}
Let $k\in\mathbb{N}$ and $0<\d<1$. Then for $1\leq q<\frac{2}{1+\d}$ we have that
\begin{align}
\wwd{k}{\Omega_0}\hookrightarrow W^{k,q}{(\Omega_0)}.
\end{align}
In particular, for $1\leq q<\frac{2}{\d}$ we have
\begin{align}
\wwd{1}{\Omega_0}\hookrightarrow L^{q}{(\Omega_0)}.
\end{align}
\end{lemma}

\begin{lemma}
Suppose that $0<\d<1$ and $1\leq q<\frac{2}{1+\d}$. Then we have that
\begin{align}
\wwd{\frac{1}{2}}{\p\Omega_0}\hookrightarrow L^{q}{(\p\Omega_0)}.
\end{align}
\end{lemma}

\begin{lemma}
Suppose that $0<\d<1$. Then we have that
\begin{align}
\wwd{1}{\Omega_0}\hookrightarrow W^0_{\d-1}{(\Omega_0)}.
\end{align}
\end{lemma}

\begin{lemma}
Suppose that $0<\d<1$. Then for each $q\in[1,\infty)$, we have that
\begin{align}
\nm{\Big(dist(\cdot,M)\Big)^{\d}f}_{L^q(\Omega_0)}\ls \hmw{f}{1}.
\end{align}
\end{lemma}

\begin{corollary}
Assume $1<s<\min\left\{\dfrac{\pi}{\omega},2\right\}$. Then we have
\begin{align}
\wwd{2}{\Omega_0}\hookrightarrow H^s(\Omega_0),\ \ \wwd{1}{\Omega_0}\hookrightarrow H^{s-1}(\Omega_0),\ \ \wwd{\frac{5}{2}}{\Sigma_0}\hookrightarrow H^{s+\frac{1}{2}}(\Sigma_0).
\end{align}
\end{corollary}

\subsection{Product estimates}

\begin{lemma}
Let $\Omega_0\in\r^2$. Suppose that $f\in H^r(\Omega_0)$ for $r\in(0,1)$ and $g\in H^1(\Omega_0)$. Then $fg\in H^{\sigma}(\Omega_0)$ for every $\sigma\in(0,r)$, and
\begin{align}
\nm{fg}_{H^{\sigma}(\Omega_0)}\leq C(r,\sigma)\nm{f}_{H^{r}(\Omega_0)}\nm{g}_{H^{1}(\Omega_0)}.
\end{align}
\end{lemma}

\begin{lemma}
Suppose that $f\in\wwd{1}{\Omega_0}$ for $0<\d<1$ and that $g\in H^{1+\kappa}(\Omega_0)$ for $0<\kappa<1$. Then $fg\in\wwd{1}{\Omega_0}$ and
\begin{align}
\hmw{fg}{1}\ls \hmw{f}{1}\hm{g}{1+\kappa}.
\end{align}
\end{lemma}

\begin{lemma}
Suppose that $f\in\wwd{\frac{1}{2}}{\Sigma_0}$ for $0<\d<1$ and that $g\in H^{\frac{1}{2}+\kappa}(\Sigma_0)$ for $0<\kappa<1$. Then $fg\in\wwd{\frac{1}{2}}{\Sigma_0}$ and
\begin{align}
\hmwe{fg}{\frac{1}{2}}\ls \hmwe{f}{\frac{1}{2}}\hme{g}{\frac{1}{2}+\kappa}.
\end{align}
\end{lemma}

\makeatletter
\renewcommand \theequation {%
B.%
\ifnum\c@subsection>\z@\@arabic\c@subsection.%
\fi\@arabic\c@equation} \@addtoreset{equation}{section}
\@addtoreset{equation}{subsection} \makeatother

\section{Energy-dissipation structure and equilibrium}

\subsection{Proof of Theorem \ref{intro theorem 1}}\label{appendix section 2}

\begin{proof}[Proof of Theorem \ref{intro theorem 1}]
We multiply $u$ on both sides of the Stokes equation and integrate it over $\Omega(t)$ to obtain
\begin{align}
I=\int_{\Omega(t)}\Big(\na_z\cdot S(P,u)\Big)\cdot u=0.
\end{align}
We divide it into several steps.

\emph{Step 1: Diffusion and fixed boundary terms:}
Integrating by parts implies
\begin{align}
I=&\bigg(\int_{\Omega(t)}\frac{\mu}{2}\abs{\dm_z u}^2-\int_{\Omega(t)}P(\na_z\cdot u)\bigg)+\int_{\Sigma_b(t)}\Big(S(P,u)\nu\Big)\cdot u+\int_{\Sigma(t)}\Big(S(P,u)\nu\Big)\cdot u
=I_1+I_2+I_3,
\end{align}
where we may use the divergence-free condition and the boundary condition to simplify
\begin{align}
I_1=&\int_{\Omega(t)}\frac{\mu}{2}\abs{\dm_z u}^2-\int_{\Omega(t)}P(\na_z\cdot u)=\int_{\Omega(t)}\frac{\mu}{2}\abs{\dm_z u}^2,\\
I_2=&\int_{\Sigma_b(t)}\Big(S(P,u)\nu\Big)\cdot u=\int_{\Sigma_b(t)}\bigg(\Big(S(P,u)\nu\Big)\cdot\nu\bigg)(u\cdot\nu)+\int_{\Sigma_b(t)}\bigg(\Big(S(P,u)\nu\Big)\cdot\tau\bigg)(u\cdot\tau)\\
=&\int_{\Sigma_b(t)}\bigg(\Big(S(P,u)\nu\Big)\cdot\tau\bigg)(u\cdot\tau)=\int_{\Sigma_b(t)}\beta\abs{u\cdot\tau}^2.\no
\end{align}

\emph{Step 2: Free surface terms:}
We directly simplify to obtain
\begin{align}
I_3=&\int_{\Sigma(t)}\Big(S(P,u)\nu\Big)\cdot u=\int_{\Sigma(t)}\left(g\z-\sigma \p_{z_1}\left(\dfrac{\p_{z_1}\z}{\sqrt{1+\abs{\p_{z_1}\z}^2}}\right)\right)(u\cdot\nu)\\
=&\int_{\Sigma(t)}\left(g\z-\sigma \p_{z_1}\left(\dfrac{\p_{z_1}\z}{\sqrt{1+\abs{\p_{z_1}\z}^2}}\right)\right)\frac{\dt\z}{\sqrt{1+\abs{\p_{z_1}\z}^2}}\no\\
=&\int_{L(t)}^{R(t)}\left(g\z-\sigma \p_{z_1}\left(\dfrac{\p_{z_1}\z}{\sqrt{1+\abs{\p_{z_1}\z}^2}}\right)\right)\dt\z=I_{3,1}+I_{3,2}.\no
\end{align}

\emph{Step 3: Gravitational term:}
We have
\begin{align}
I_{3,1}=&\int_{L(t)}^{R(t)}g\z\dt\z=\int_{L(t)}^{R(t)}\frac{g}{2}\dt\abs{\z}^2
=\dt\int_{L(t)}^{R(t)}\frac{g}{2}\abs{\z}^2-\frac{g}{2}\abs{\z(R)}^2\dt R+\frac{g}{2}\abs{\z(L)}^2\dt L=\dt\int_{L(t)}^{R(t)}\frac{g}{2}\abs{\z}^2,
\end{align}
since $\z(L)=\z(R)=0$.

\emph{Step 4: Surface tension terms:}
Integrating by parts and using Reynold's transport equation imply
\begin{align}
I_{3,2}=&-\int_{L(t)}^{R(t)}\sigma \p_{z_1}\left(\dfrac{\p_{z_1}\z}{\sqrt{1+\abs{\p_{z_1}\z}^2}}\right)\dt\z\\
=&\int_{L(t)}^{R(t)}\sigma \dfrac{\p_{z_1}\z}{\sqrt{1+\abs{\p_{z_1}\z}^2}}\dt\p_{z_1}\z
-\sigma\dfrac{\p_{z_1}\z(R)}{\sqrt{1+\abs{\p_{z_1}\z(R)}^2}}\dt\z(R)+\sigma\dfrac{\p_{z_1}\z(L)}{\sqrt{1+\abs{\p_{z_1}\z(L)}^2}}\dt\z(L)=A+B+C,\no
\end{align}
where we may simplify
\begin{align}
A=&\int_{L(t)}^{R(t)}\sigma \dfrac{\p_{z_1}\z}{\sqrt{1+\abs{\p_{z_1}\z}^2}}\dt\p_{z_1}\z=\int_{L(t)}^{R(t)}\sigma\dt\sqrt{1+\abs{\p_{z_1}\z}^2}\\
=&\dt\int_{L(t)}^{R(t)}\sigma\sqrt{1+\abs{\p_{z_1}\z}^2}-\sigma\dt R\sqrt{1+\abs{\p_{z_1}\z(R)}^2}+\sigma\dt L\sqrt{1+\abs{\p_{z_1}\z(L)}^2}=A_1+A_2+A_3,\no
\end{align}
and using the transport equation with $u_2(L)=u_2(R)=0$, we have
\begin{align}
B=-\sigma\dfrac{\p_{z_1}\z(R)}{\sqrt{1+\abs{\p_{z_1}\z(R)}^2}}\dt\z(R)
&=-\sigma\dfrac{\p_{z_1}\z(R)}{\sqrt{1+\abs{\p_{z_1}\z(R)}^2}}\Big(-u_1(R)\p_{z_1}\z(R)+u_2(R)\Big) \\
&=\sigma u_1(R)\dfrac{\abs{\p_{z_1}\z(R)}^2}{\sqrt{1+\abs{\p_{z_1}\z(R)}^2}},\no
\end{align}
and
\begin{align}
C=\sigma\dfrac{\p_{z_1}\z(L)}{\sqrt{1+\abs{\p_{z_1}\z(L)}^2}}\dt\z(L)
&=\sigma\dfrac{\p_{z_1}\z(L)}{\sqrt{1+\abs{\p_{z_1}\z(L)}^2}}\Big(-u_1(L)\p_{z_1}\z(L)+u_2(L)\Big) \\
& =-\sigma u_1(L)\dfrac{\abs{\p_{z_1}\z(L)}^2}{\sqrt{1+\abs{\p_{z_1}\z(L)}^2}}.\no
\end{align}

\emph{Step 5: Contact point terms:}
Note that fact that $\dt R=u_1(R)$ and $\dt L=u_1(L)$, we have
\begin{align}
A_2+B=&-\sigma\dt R\sqrt{1+\abs{\p_{z_1}\z(R)}^2}+\sigma u_1(R)\dfrac{\abs{\p_{z_1}\z(R)}^2}{\sqrt{1+\abs{\p_{z_1}\z(R)}^2}}\\
=&-\sigma\dt R\left(\sqrt{1+\abs{\p_{z_1}\z(R)}^2}-\dfrac{\abs{\p_{z_1}\z(R)}^2}{\sqrt{1+\abs{\p_{z_1}\z(R)}^2}}\right)=-\sigma\dt R\dfrac{1}{\sqrt{1+\abs{\p_{z_1}\z(R)}^2}},\no\\
A_3+C=&\sigma\dt L\sqrt{1+\abs{\p_{z_1}\z(L)}^2}-\sigma u_1(L)\dfrac{\abs{\p_{z_1}\z(L)}^2}{\sqrt{1+\abs{\p_{z_1}\z(L)}^2}}\\
=&\sigma\dt L\left(\sqrt{1+\abs{\p_{z_1}\z(L)}^2}-\dfrac{\abs{\p_{z_1}\z(L)}^2}{\sqrt{1+\abs{\p_{z_1}\z(L)}^2}}\right)=\sigma\dt L\dfrac{1}{\sqrt{1+\abs{\p_{z_1}\z(L)}^2}}.\no\\
\end{align}
Using the contact point condition, we have
\begin{align}
A_2+B=&-\sigma\dt R\dfrac{1}{\sqrt{1+\abs{\p_{z_1}\z(R)}^2}}=-\dt R\Big([\gamma]-\w(\dt R)\Big),\\
A_3+C=&\sigma\dt L\dfrac{1}{\sqrt{1+\abs{\p_{z_1}\z(L)}^2}}=\dt L\Big(\w(\dt L)+[\gamma]\Big).
\end{align}
Hence, we have
\begin{align}
A_2+A_3+B+C=-\dt\Big([\gamma](R-L)\Big)+\Big(\w(\dt L)\dt L+\w(\dt R)\dt R\Big).
\end{align}

\emph{Step 6: Synthesis:}
Collecting all terms, we arrive at \eqref{energy-dissipation}.

\end{proof}

\subsection{Proof of Theorem \ref{intro theorem 2}}\label{appendix section 1}

\begin{proof}[Proof of Theorem \ref{intro theorem 2}]
It is well known (see for instance \cite{finn}) that, given the width of the droplet domain, there exists a unique solution to the second and third equations in \eqref{equilibrium} that is smooth, even, and monotonically decreasing from the center. In the present context, we must choose the equilibrium width in order to satisfy the fourth and fifth equations. Here for the sake of completeness, we will give a quick sketch of the construction of solutions and details of how to determine $P_0$ and $R_0-L_0$ given that we specify the mass $M$.

Due to horizontal translational invariance, we may assume without loss of generality that $(R_0+L_0)/2=0$, i.e. the droplet is centered at the origin. Then
consider the new unknown $\ell=(R_0-L_0)/2$, in which case $L_0=-\ell$ and $R_0=\ell$.
After integrating on both sides of the equilibrium equation, we find that $P_0$ is determined by $M$ and $\ell$ via
\begin{align}\label{app:pressure}
P_0=\frac{Mg+2\sqrt{\sigma^2-[\gamma]^2}}{2\ell}>0.
\end{align}
It remains to construct the solution with given mass $M$, and $\ell$ chosen so that the equilibrium equations are satisfied for $P_0$ determined by \eqref{app:pressure}.

Due to reflectional symmetry, it suffices to construct the solution for $z_1\in[-\ell,0]$. Let $r=-z_1$ and $\tan\psi=-\p_r\z_0=\p_{z_1}\z_0$. The variable $\psi$ is the angle formed between the tangent line of $\z_0$ and a line parallel to the $z_1$ axis through $(z_1,\zeta_0(z_1))$, which ranges from $\psi=0$ at the maximum of $\z_0$ in the center to $\psi = \psi_0$ for
\begin{equation}
\psi_0:=\arctan\left(\dfrac{\sqrt{\sigma^2-[\gamma]^2}}{[\gamma]}\right) \in (0,\pi/2)
\end{equation}
at the contact point. Then in these coordinates the equilibrium equation is
\begin{align}
-\sigma\p_r(\sin\psi)=g\z_0-P_0,
\end{align}
which, considering $\dfrac{\ud{r}}{\ud{\psi}}=\left(\dfrac{\ud{\psi}}{\ud{r}}\right)^{-1}$ and $\dfrac{\ud{\z_0}}{\ud{\psi}}=\dfrac{\ud{\z_0}}{\ud{r}}\dfrac{\ud{r}}{\ud{\psi}}$, is equivalent to
\begin{align}
\dfrac{\ud{r}}{\ud{\psi}}=-\frac{\sigma\cos\psi}{g\z_0-P_0},\ \ \dfrac{\ud{\z_0}}{\ud{\psi}}=\frac{\sigma\sin\psi}{g\z_0-P_0}.
\end{align}
Setting $\z_0=0$ at $\psi_0$, we may solve $\dfrac{\ud{\z_0}}{\ud{\psi}}$ equation to get
\begin{align}
\frac{g}{2}\z_0^2-P_0\z_0-[\gamma]+\sigma\cos\psi=0 \text{ and hence }
\z_0(\psi)=\frac{P_0-\sqrt{P_0^2-2g(\sigma\cos\psi-[\gamma])}}{g}.
\end{align}
Then plugging this into the $\ud{r}/\ud{\psi}$ equation and setting $r=0$ at $\psi=0$, we obtain
\begin{align}
r(\psi)=\int_{0}^{\psi}\frac{\sigma\cos\psi}{\sqrt{P_0^2-2g(\sigma\cos\psi-[\gamma])}}\ud{\psi}.
\end{align}

It remains only to enforce the condition $r(\psi_0)=\ell$, which, in light of \eqref{app:pressure}, is equivalent to
\begin{align}
1=\int_{0}^{\psi_0}\frac{2\sigma\cos\psi}{\sqrt{\Big(Mg+2\sqrt{\sigma^2-[\gamma]^2}\Big)^2-8g\ell^2(\sigma\cos\psi-[\gamma])}}\ud{\psi}.
\end{align}
When $\ell\rt0$, we have
\begin{align}
&\int_{0}^{\psi_0}\frac{2\sigma\cos\psi}{\sqrt{\Big(Mg+2\sqrt{\sigma^2-[\gamma]^2}\Big)^2-8g\ell^2(\sigma\cos\psi-[\gamma])}}\ud{\psi}\quad \\
\rt&\int_{0}^{\psi_0}\frac{2\sigma\cos\psi}{Mg+2\sqrt{\sigma^2-[\gamma]^2}}\ud{\psi}=\frac{2\sqrt{\sigma^2-[\gamma]^2}}{Mg+2\sqrt{\sigma^2-[\gamma]^2}}<1.\no
\end{align}
Also, as $\ell\rt\dfrac{Mg+2\sqrt{\sigma^2-[\gamma]^2}}{\sqrt{8g(\sigma-[\gamma])}}$, considering the Taylor expansion of $\cos\psi$ around $\psi=0$, the integral monotonically increases to $\infty$. Hence, there exists a unique $\ell$ such that the integral is exactly $1$. With this choice of $\ell$ and $P_0$, the equilibrium equations are satisfied.
\end{proof}

\makeatletter
\renewcommand \theequation {%
C.%
\ifnum\c@subsection>\z@\@arabic\c@subsection.%
\fi\@arabic\c@equation} \@addtoreset{equation}{section}
\@addtoreset{equation}{subsection} \makeatother

\section{Nonlinear quantities}

In this section we record a number of results about the nonlinear terms appearing in our analysis.

\subsection{Estimates of $J_1,$ $J_2,$ and $A$   }

Recall that are $J_1,$ $J_2,$ and $A$  given by \eqref{JKA_def}.

\begin{lemma}\label{appendix lemma 2}
Suppose that $\hms{\e(t)}{\frac{1}{2}}<\vartheta$ for some $\vartheta>0$ sufficiently small. Then
\begin{equation}
\abs{J_1-1}\ls  \hms{\e}{0} \text{ and }
\abs{J_2} + \abs{A} \ls \hms{\e}{\frac{1}{2}}.
\end{equation}
\end{lemma}
\begin{proof}
Using the conservation of mass in \eqref{linearize 2},
\begin{align}
J_1\int_{-\ell}^{\ell}\e(t,x_1)\ud{x_1}=&\int_{-\ell}^{-\ell}(1-J_1)\z_0(x_1)\ud{x_1},
\end{align}
we have
\begin{align}
J_1-1=-\frac{\ds\int_{-\ell}^{-\ell}\e(x_1)\ud{x_1}}{\ds\int_{-\ell}^{-\ell}\z_0(x_1)\ud{x_1}+\int_{-\ell}^{-\ell}\e(x_1)\ud{x_1}},
\end{align}
Hence, when
\begin{align}
\abs{\int_{-\ell}^{-\ell}\e(x_1)\ud{x_1}}\ls \hms{\e}{0}<\frac{1}{2}\int_{-\ell}^{-\ell}\z_0(x_1)\ud{x_1}=\frac{M}{2},
\end{align}
we know
\begin{align}
\abs{J_1-1}\ls \frac{2}{M}\int_{-\ell}^{-\ell}\e(x_1)\ud{x_1}\ls\hms{\e}{0}.
\end{align}

We now turn to the proof of the $J_2$ estimate, noting first that
\begin{align}
J_2(x)= 1+\dfrac{\be(x)}{\z_0(x_1)}+\dfrac{x_2}{\z_0(x_1)}\p_2\be(x).
\end{align}
The difficulty lies in when $x$ is close to the contact point.
In a neighborhood of the contact points, for $\Omega_0\ni x=(x_1,x_2)$, let $s=\dfrac{x_2}{x_1\mp\ell}$. Using Cauchy's Mean Value Theorem, we know that ,
\begin{align}
\frac{\be(x)}{\z_0(x_1)}=\frac{\be\Big(x_1,s(x_1\mp\ell)\Big)-\be(\pm\ell,0)}{\z_0(x_1)-\z_0(\pm\ell)}
=\frac{\p_1\be\Big(c,s(c\mp\ell)\Big)+s\p_2\be\Big(c,s(c\mp\ell)\Big)}{\p_1\z_0(c)},
\end{align}
for some $c$ close to $\pm\ell$. Since $x\in\Omega_0$ and $\Omega_0$ is convex, we may directly estimate that $s\leq\abs{\p_1\z_0(\pm\ell)}$ and
\begin{align}
\abs{\frac{\be(x)}{\z_0(x_1)}}\ls\hm{\be}{1}\ls\hms{\e}{\frac{1}{2}}.
\end{align}
Also, since for $x\in\Omega_0$, $0\leq x_2\leq\z_0(x_1)$, we have
\begin{align}
\abs{\dfrac{x_2}{\z_0(x_1)}}\leq 1.
\end{align}
In total, we have
\begin{align}
\abs{J_2}\ls\hm{\be}{1}\ls\hms{\e}{\frac{1}{2}}.
\end{align}

Finally, we turn to the $A$ estimate.  We begin by decomposing
\begin{align}
A(x)=&\frac{x_2}{\z_0(x_1)}\bigg(\p_1\be(x)-\frac{\be(x)}{\z_0(x_1)}\p_1\z_0(x_1)\bigg).
\end{align}
As in the estimate of $J_2$ above, we have
\begin{align}
\abs{\dfrac{x_2}{\z_0(x_1)}}\leq 1.
\end{align}
Hence
\begin{align}
\abs{\p_1\be(x)}\ls\hm{\be}{1}\ls\hms{\e}{\frac{1}{2}}.
\end{align}
Also, we know
\begin{align}
\abs{\frac{\be(x)}{\z_0(x_1)}}\ls\hm{\be}{1}\ls\hms{\e}{\frac{1}{2}}.
\end{align}
Combining these then provides the desired bound.

\end{proof}

Using a similar argument as in Lemma \ref{appendix lemma 2}, we obtain the following:
\begin{lemma}\label{appendix lemma 1}
Let $0<\d<1$. There exists a universal $\vartheta\in(0,1)$ such that if $\hmwss{\e}{\frac{5}{2}}\leq\vartheta$, then
\begin{align}
\nm{J-1}_{L^{\infty}(\Omega_0)}+\nm{A}_{L^{\infty}(\Omega_0)}\leq&\frac{1}{2},\\
\nm{\n-1}_{L^{\infty}(\p\Omega_0)}+\nm{K-1}_{L^{\infty}(\Omega_0)}\leq&\frac{1}{2},\\
\nm{K}_{L^{\infty}(\Omega_0)}+\nm{\a}_{L^{\infty}(\Omega_0)}\ls&1.
\end{align}
Also, the map $\Pi$ is a diffeomorphism.
\end{lemma}

\subsection{Nonlinear terms in the energy-dissipation estimates}\label{ap section 1}

In this subsection we record the nonlinearities that appear in \eqref{system 3}.  We begin with the form when  $\dt$ is applied.  In this case we have:
\begin{equation}\label{apf 1}
 \s_1=-\na_{\dt\a}\cdot\Big(pI-\mu\dm_{\a}u\Big)+\mu\na_{\a}\cdot\dm_{\dt\a}u,
\end{equation}
\begin{equation}\label{apf 2}
 \s_2 =  -\na_{\dt\a}\cdot u,
\end{equation}
\begin{equation}\label{apf 3}
 \s_3=\mu\dm_{\dt\a}u\n-(\tilde pI-\mu\dm_{\a} u)\dt\n
+g\e\dt\n-\sigma \p_{1}\left(\dfrac{k_1\p_{1}\z_0}{\sqrt{1+\abs{\p_{1}\z_0}^2}}+\dfrac{k_1\p_{1}\z_0+\p_1\e}{\Big(\sqrt{1+\abs{\p_{1}\z_0}^2}\Big)^3}+\rr\right)\dt\n,
\end{equation}
\begin{equation}\label{apf 4}
\s_4 = \mu\dm_{\dt\a}u\nu\cdot\tau,
\end{equation}
\begin{equation}\label{apf 5}
\s_5 =(u\cdot\dt\n)\sqrt{1+\abs{\p_1\z_0}^2},
\end{equation}
\begin{equation}\label{apf 6}
\s_6 = -\kappa\tilde\ww'(\dt l)\dt^2l,
\end{equation}
and
\begin{equation}\label{apf 7}
\s_7 = -\kappa\tilde\ww'(\dt r)\dt^2r.
\end{equation}

Next we record the form when $\dt^2$ is applied:
\begin{align}
\s_1=&-2\na_{\dt\a}\cdot\Big(\dt pI-\mu\dm_{\a}\dt u\Big)+2\mu\na_{\a}\cdot\dm_{\dt\a}\dt u-\na_{\dt^2\a}\cdot\Big(pI-\mu\dm_{\a}u\Big) \label{apf 1'}\\
&+2\mu\na_{\dt\a}\cdot\dm_{\dt\a}u+\mu\na_{\a}\cdot\dm_{\dt^2\a}u,\no
\end{align}
\begin{equation}\label{apf 2'}
\s_2 = -\na_{\dt^2\a}u-2\na_{\dt\a}\cdot\dt u,
\end{equation}
\begin{align}
\s_3=&\mu\dm_{\dt^2\a}u\n+\mu\dm_{\dt\a}u\dt\n+2\mu\dm_{\dt\a}\dt u\n-( pI-\mu\dm_{\a} u)\dt^2\n-2(\dt pI-\mu\dm_{\a} \dt u)\dt\n \label{apf 3'} \\
&+g\e\dt^2\n-\sigma \p_{1}\left(\dfrac{k_1\p_{1}\z_0}{\sqrt{1+\abs{\p_{1}\z_0}^2}}+\dfrac{k_1\p_{1}\z_0+\p_1\e}{\Big(\sqrt{1+\abs{\p_{1}\z_0}^2}\Big)^3}+\rr\right)\dt^2\n\no\\
& +2g\dt\e\dt\n-2\sigma \p_{1}\left(\dfrac{\dt k_1\p_{1}\z_0}{\sqrt{1+\abs{\p_{1}\z_0}^2}}+\dfrac{\dt k_1\p_{1}\z_0+\dt\p_1\e}{\Big(\sqrt{1+\abs{\p_{1}\z_0}^2}\Big)^3}+\dt\rr\right)\dt\n,\no
\end{align}
\begin{equation}\label{apf 4'}
\s_4 = \mu\dm_{\dt^2\a}u\nu\cdot\tau+2\mu\dm_{\dt\a}\dt u\nu\cdot\tau,
\end{equation}
\begin{equation}\label{apf 5'}
\s_5 = (u\cdot\dt^2\n)\sqrt{1+\abs{\p_1\z_0}^2}+2(\dt u\cdot\dt\n)\sqrt{1+\abs{\p_1\z_0}^2},
\end{equation}
\begin{equation}\label{apf 6'}
\s_6 = -\kappa\tilde\ww'(\dt l)\dt^3l-\kappa\tilde\ww''(\dt l)(\dt^2l)^2,
\end{equation}
and
\begin{equation}\label{apf 7'}
\s_7 = -\kappa\tilde\ww'(\dt r)\dt^3r-\kappa\tilde\ww''(\dt r)(\dt^2r)^2.
\end{equation}

\subsection{Estimates of $\rr$, $\qq$, $\ss$ and $\oo$} \label{rqso_appendix}

In this section we record estimates for the terms $\rr$, $\qq$, $\ss$ and $\oo$, which we recall are defined in \eqref{R_def}, \eqref{Q_def}, \eqref{S_def}, and \eqref{O_def}, respectively.  To arrive at the estimates we first expand each term using the fundamental theorem of calculus.  With the expansions in hand, the estimates then follow from  elementary applications of the product rule and Sobolev embedding.  As such we will omit the proofs of the bounds and only expansions and the form of the estimates.

We begin with the term $\rr$, rewriting it as
\begin{align}\label{ap rr}
\rr=&\dfrac{K_1^2\Big(\p_{1}\z_0+\p_1\e\Big)}{\Big(1+K_1^2\Big(\p_{1}\z_0+\p_1\e\Big)^2\Big)^{\frac{1}{2}}}-\dfrac{\p_{1}\z_0}{\Big(1+\abs{\p_{1}\z_0}^2\Big)^{\frac{1}{2}}}
-\dfrac{k_1\p_{1}\z_0}{\Big(1+\abs{\p_{1}\z_0}^2\Big)^{\frac{1}{2}}}-\dfrac{k_1\p_{1}\z_0+\p_1\e}{\Big(1+\abs{\p_{1}\z_0}^2\Big)^{\frac{3}{2}}}\\
=&\int_0^1\left(\frac{2k_1\Big(k_1(\p_{1}\z_0+\omega \p_1\e)+2(1+\omega k_1)\p_1\e\Big)}{\Big(1+(1+\omega k_1)^2\Big(\p_{1}\z_0+\omega \p_1\e\Big)^2\Big)^{\frac{3}{2}}} \right.\\
&-\left.\frac{3(1+\omega k_1)^2(\p_{1}\z_0+\omega \p_1\e)\Big(k_1(\p_{1}\z_0+\omega \p_1\e)+(1+\omega k_1)\p_1\e\Big)^2}{\Big(1+(1+\omega k_1)^2\Big(\p_{1}\z_0+\omega \p_1\e\Big)^2\Big)^{\frac{5}{2}}}\right)(1-\omega )\ud{\omega }.\no
\end{align}
The estimates for $\rr$ are recorded in the following lemma.
\begin{lemma}\label{estimate_r}
We have the bounds
\begin{equation}
\abs{\rr}\ls \abs{k_1}^2+\abs{\p_1\e}^2,
\end{equation}
\begin{equation}
\abs{\dt\rr}\ls \abs{k_1}\abs{\dt k_1}+\abs{\dt k_1}\abs{\p_1\e}+\abs{k_1}\abs{\dt \p_1\e}+\abs{\p_1\e}\abs{\dt\p_1\e},
\end{equation}
\begin{equation}
\abs{\dt^2\rr}\ls \abs{k_1}\abs{\dt^2 k_1}+\abs{\dt k_1}^2+\abs{\dt^2 k_1}\abs{\p_1\e}+\abs{k_1}\abs{\dt^2 \p_1\e}+\abs{\p_1\e}\abs{\dt^2\p_1\e}+\abs{\dt\p_1\e}^2,
\end{equation}
\begin{equation}
\abs{\p_1\rr}\ls \abs{\p_1\e}\abs{\p_1^2\e}+\abs{k_1}\abs{\p_1^2\e},
\end{equation}
\begin{equation}
\abs{\dt\p_1\rr}\ls \abs{\dt k_1}\abs{\p_1^2\e}+\abs{k_1}\abs{\dt \p_1^2\e}+\abs{\p_1^2\e}\abs{\dt\p_1\e}+\abs{\p_1\e}\abs{\dt\p_1^2\e}.
\end{equation}
\end{lemma}

We now expand the $\qq$ term via
\begin{align}\label{ap qq}
\qq=&\dfrac{1}{\Big(1+K_1^2\Big(\p_{1}\z_0+\p_1\e\Big)^2\Big)^{\frac{1}{2}}}-\dfrac{1}{\Big(1+\abs{\p_{1}\z_0}^2\Big)^{\frac{1}{2}}}
+\dfrac{\p_1\z_0(k_1\p_{1}\z_0+\p_1\e)}{\Big(1+\abs{\p_{1}\z_0}^2\Big)^{\frac{3}{2}}}\\
=&\int_0^1\left(-\frac{\Big(k_1(\p_{1}\z_0+\omega \p_1\e)+(1+\omega k_1)\p_1\e\Big)^2+2k_1\p_1\e(1+\omega k_1)(\p_{1}\z_0+\omega \p_1\e)}{\Big(1+(1+\omega k_1)^2\Big(\p_{1}\z_0+\omega \p_1\e\Big)^2\Big)^{\frac{3}{2}}} \right. \no\\
&+\left.\frac{3(1+\omega k_1)^2(\p_{1}\z_0+\omega \p_1\e)^2\Big(k_1(\p_{1}\z_0+\omega \p_1\e)+(1+\omega k_1)\p_1\e\Big)^2}{\Big(1+(1+\omega k_1)^2\Big(\p_{1}\z_0+\omega \p_1\e\Big)^2\Big)^{\frac{5}{2}}}\right)(1-\omega )\ud{\omega }.\no
\end{align}
The estimates for $\qq$ are then recorded in the following lemma.
\begin{lemma}\label{estimate_q}
We have the estimates
\begin{equation}
\abs{\qq}\ls \abs{k_1}^2+\abs{\p_1\e}^2,
\end{equation}
\begin{equation}
\abs{\dt\qq}\ls \abs{k_1}\abs{\dt k_1}+\abs{\dt k_1}\abs{\p_1\e}+\abs{k_1}\abs{\dt \p_1\e}+\abs{\p_1\e}\abs{\dt\p_1\e},
\end{equation}
\begin{equation}
\abs{\dt^2\qq}\ls \abs{k_1}\abs{\dt^2 k_1}+\abs{\dt k_1}^2+\abs{\dt^2 k_1}\abs{\p_1\e}+\abs{k_1}\abs{\dt^2 \p_1\e}+\abs{\p_1\e}\abs{\dt^2\p_1\e}+\abs{\dt\p_1\e}^2,
\end{equation}
\begin{equation}
\abs{\p_1\qq} \ls \abs{\p_1\e}\abs{\p_1^2\e}+\abs{k_1}\abs{\p_1^2\e},
\end{equation}
\begin{equation}
\abs{\dt\p_1\qq}\ls \abs{\dt k_1}\abs{\p_1^2\e}+\abs{k_1}\abs{\dt \p_1^2\e}+\abs{\p_1^2\e}\abs{\dt\p_1\e}+\abs{\p_1\e}\abs{\dt\p_1^2\e}.
\end{equation}
\end{lemma}

Next we write the $\ss$ term as
\begin{align}\label{ap ss}
\ss=&(J_1-1)\dt\e+\Big(\tilde a\p_1\z- a\p_1\z_0\Big)=(J_1-1)\dt\e+\tilde a\p_1\e+(\tilde a- a)\p_1\z_0.
\end{align}
The estimates for $\ss$ are in the following lemma.
\begin{lemma}\label{estimate_s}
We have the bounds
\begin{equation}
\abs{\ss}\ls \abs{k_1}\abs{\dt\e}+\Big(\abs{\dt L}+\abs{\dt R}\Big)\abs{\p_1\e}+\abs{\oo},
\end{equation}
\begin{equation}
\abs{\dt\ss}\ls \abs{k_1}\abs{\dt^2\e}+\Big(\abs{\dt L}+\abs{\dt R}\Big)\abs{\dt\p_1\e}+\Big(\abs{\dt^2 L}+\abs{\dt^2 R}\Big)\abs{\p_1\e}+\abs{\dt\oo},
\end{equation}
\begin{equation}
\abs{\dt^2\ss}\ls \abs{k_1}\abs{\dt^3\e}+\Big(\abs{\dt L}+\abs{\dt R}\Big)\abs{\dt^2\p_1\e}+\Big(\abs{\dt^3 L}+\abs{\dt^3 R}\Big)\abs{\p_1\e}+\abs{\dt^2\oo}.
\end{equation}
\end{lemma}

Finally, we write the term $\oo$ as
\begin{align}\label{ap oo}
\oo=a-\tilde a=&-(\dt r-\dt l)\bigg(\dfrac{1}{2\ell}-\frac{2\ell}{(2\ell+r-l)^2}\bigg)x_1=-(\dt r-\dt l)\frac{(r-l)(4\ell+r-l)}{2\ell(2\ell+r-l)^2}x_1\\
=&k_1(\dt r-\dt l)\frac{4\ell+r-l}{2\ell(2\ell+r-l)}x_1=k_1(\dt r-\dt l)\bigg(\frac{1}{2\ell}+\frac{1}{2\ell+r-l}\bigg)x_1.\no
\end{align}
The $\oo$ estimates are recorded in the following.
\begin{lemma}\label{estimate_o}
We have the bounds
\begin{equation}
\abs{\oo}\ls \abs{k_1}\Big(\abs{\dt L}+\abs{\dt R}\Big),
\end{equation}
\begin{equation}
\abs{\dt\oo}\ls \Big(\abs{\dt L}^2+\abs{\dt R}^2\Big)+\abs{k_1}\Big(\abs{\dt^2 L}+\abs{\dt^2 R}\Big),
\end{equation}
\begin{equation}
\abs{\dt^2\oo} \ls \Big(\abs{\dt L}+\abs{\dt R}\Big)\Big(\abs{\dt^2 L}+\abs{\dt^2 R}\Big)+\abs{k_1}\Big(\abs{\dt^3 L}+\abs{\dt^3 R}\Big).
\end{equation}

\end{lemma}


\end{document}